\documentclass[onefignum,onetabnum]{siamart190516} 

\usepackage{lipsum}
\usepackage{amsfonts}
\usepackage{graphicx}
\usepackage{epstopdf}
\usepackage{algorithmic}
\usepackage{dsfont}
\usepackage{enumitem}
\usepackage{mathtools}
\usepackage[caption=false]{subfig} 
\ifpdf
  \DeclareGraphicsExtensions{.eps,.pdf,.png,.jpg}
\else
  \DeclareGraphicsExtensions{.eps}
\fi

\newsiamremark{remark}{Remark}
\newsiamremark{hypothesis}{Hypothesis}
\crefname{hypothesis}{Hypothesis}{Hypotheses}
\newsiamthm{claim}{Claim}

\newtheorem*{theorem_recall_1}{Theorem~2.1}
\newtheorem*{theorem_recall_2}{Theorem~2.2}
\newtheorem*{proposition_recall_1}{Proposition~3.1}

\newtheorem*{remark_}{Example}

\newtheorem*{remark__}{Remark}

\theoremstyle{plain}
\newtheorem{assumption}{Assumption}
\numberwithin{assumption}{subsection}

\newtheorem{manualassumptioninner}{Assumption}
\newenvironment{manualassumption}[1]{%
  \renewcommand\themanualassumptioninner{#1}%
  \manualassumptioninner
}{\endmanualassumptioninner}

\newtheorem{manualtheoreminner}{Theorem}
\newenvironment{manualtheorem}[1]{%
  \renewcommand\themanualtheoreminner{#1}%
  \manualtheoreminner
}{\endmanualtheoreminner}

\makeatletter
\newcommand{\customlabel}[2]{%
   \protected@write \@auxout {}{\string \newlabel {#1}{{#2}{\thepage}{#2}{#1}{}} }%
   \hypertarget{#1}{}
}
\makeatother

\DeclareMathOperator*{\argmax}{arg\,max}
\DeclareMathOperator*{\argmin}{arg\,min}

\DeclareMathOperator{\sech}{sech}

\numberwithin{equation}{subsection}

\headers{Two Timescale Stochastic Gradient Descent in Continuous Time}{L. Sharrock and N. Kantas}

\title{Two-Timescale Stochastic Gradient Descent in Continuous Time with Applications to Joint Online Parameter Estimation and Optimal Sensor Placement
}

\author{Louis Sharrock\thanks{Department of Mathematics, Imperial College London, South Kensington, SW7 2AZ, UK
  (\email{louis.sharrock16@imperial.ac.uk})}
\and Nikolas Kantas$^{*}$
 }

\usepackage{amsopn}

\ifpdf
\hypersetup{
  pdftitle={Two-Timescale Stochastic Gradient Descent in Continuous Time with Applications to Joint Online Parameter Estimation and Optimal Sensor Placement},
  pdfauthor={L. Sharrock, N. Kantas}
}
\fi

\begin{document}

\maketitle

\begin{abstract}
In this paper, we establish the almost sure convergence of two-timescale stochastic gradient descent algorithms in continuous time under general noise and stability conditions, extending well known results in discrete time. We analyse algorithms with additive noise and those with non-additive noise. In the non-additive case, our analysis is carried out under the assumption that the noise is a continuous-time Markov process, controlled by the algorithm states. The algorithms we consider can be applied to a broad class of bilevel optimisation problems. We study one such problem in detail, namely, the problem of joint online parameter estimation and optimal sensor placement for a partially observed diffusion process. We demonstrate how this can be formulated as a bilevel optimisation problem, and propose a solution in the form of a continuous-time, two-timescale, stochastic gradient descent algorithm. Furthermore, under suitable conditions on the latent signal, the filter, and the filter derivatives, we establish almost sure convergence of the online parameter estimates and optimal sensor placements to the stationary points of the asymptotic log-likelihood and asymptotic filter covariance, respectively. We also provide numerical examples, illustrating the application of the proposed methodology to a partially observed Bene\v{s} equation, and a partially observed stochastic advection-diffusion equation.
\end{abstract}

\begin{keywords}
  two-timescale stochastic approximation, stochastic gradient descent, recursive maximum likelihood, online parameter estimation, optimal sensor placement, Bene\v{s} filter, Kalman-Bucy filter, stochastic advection-diffusion equation.
\end{keywords}

\section{Introduction} \label{sec:intro}
Many modern problems in engineering, the sciences, economics, and machine learning, involve the optimisation of two or more interdependent performance criteria. These include, among others, unsupervised learning \cite{Heusel2017}, reinforcement learning \cite{Karmakar2018,Konda2003a}, meta-learning \cite{Rajeswaran2019}, 
and hyper-parameter optimisation \cite{Franceschi2018}. 
In this paper, we formulate such problems as unconstrained bilevel optimisation problems, in which the objective is to obtain  $\smash{\alpha^{*}\in\Lambda_{\alpha}\subseteq\mathbb{R}^{d_1}}$, $\smash{\beta^{*}({\alpha}^{*})\in\Lambda_{\beta}\subseteq\mathbb{R}^{d_2}}$,  such that
\begin{alignat}{2}
\alpha^{*}&\in \argmin_{\alpha\in\Lambda_{\alpha}} f\big(\alpha,\beta^{*}(\alpha)\big)~~,~~\beta^{*}(\alpha)&&\in\argmin_{\beta\in\Lambda_{\beta}}g(\alpha,\beta) \label{global_min}
\end{alignat}
where $f,g:\mathbb{R}^{d_1}\times\mathbb{R}^{d_2}\rightarrow\mathbb{R}$ are continuously differentiable functions, and $\Lambda_{\alpha}$, $\Lambda_{\beta}$ are closed subsets of $\mathbb{R}^{d_1}$, $\mathbb{R}^{d_2}$, respectively.  We will assume, as in many applications, that we only have access to noisy estimates of $f$ and $g$. 

There are, unsurprisingly, several significant challenges associated with this optimisation problem. Firstly, in order to evaluate the upper-level objective function, $f(\cdot,\cdot)$, one must obtain the global minimiser of the lower-level objective function $g(\alpha,\cdot)$, for all $\alpha\in\Lambda_{\alpha}$. This may be very difficult, particularly if $g(\alpha,\cdot)$ is a complex function. In many practical applications of interest, one or both of the objective functions may be prohibitively costly to compute (e.g., they may depend on very high-dimensional data), which compounds this problem. Secondly, it may not be possible to compute the gradient of the function $\smash{\beta^{*}(\alpha)}$. Thus, even if we could obtain $\smash{\beta^{*}(\alpha)}$ and evaluate $\smash{f(\alpha,\beta^{*}(\alpha))}$ for all $\alpha\in\Lambda_{\alpha}$, it would not be possible to solve the upper-level optimisation problem directly using gradient-based methods. 

In practice, and with these considerations in mind, it is typical to consider a slightly weaker optimisation problem,  in which the objective is to obtain $\smash{\alpha^{*}}$, $\smash{\beta^{*}}$ such that, simultaneously, $\alpha^{*}$ locally minimises $\smash{f(\cdot,\beta^{*})}$, and $\beta^{*}$ locally minimises $\smash{g(\alpha^{*},\cdot)}$.  That is, such that
\begin{equation}
\alpha^{*} =\argmin_{\alpha\in U_{\alpha^{*}}}f(\alpha,\beta^{*}) ~~~,~~~\beta^{*} = \argmin_{\beta\in U_{\beta^{*}}} g(\alpha^{*},\beta) \label{opt2_}
\end{equation}
where $U_{\alpha^{*}}\subset \Lambda_{\alpha}$ and $U_{\beta^{*}}\subset\Lambda_{\beta}$ are local neighbourhoods of $\alpha^{*}$ and $\beta^{*}$, respectively. We will not assume any form of convexity, and thus we weaken this objective further, seeking values of $\smash{\alpha^{*}}$ and $\smash{\beta^{*}}$ which satisfy the following local stationarity condition
\begin{equation}
\nabla_{\alpha}f(\alpha^{*},\beta^{*}) = 0~~~,~~~\nabla_{\beta}g(\alpha^{*},\beta^{*})=0. \label{eq102}
\end{equation}
In this paper, we analyse the use of gradient methods for this problem,  under the assumption that we continuously observe noisy estimates of these gradients (see Section \ref{sec:method}).  We also consider an important application of this problem which arises in continuous-time state-space models (see Section \ref{sec:apps}).

\subsection{Methodology} \label{sec:method}
\subsubsection{Two-Timescale Stochastic Gradient Descent}
A natural candidate for solving this class of bilevel optimisation problems is two-timescale stochastic gradient descent. Broadly speaking, stochastic gradient descent is a sequential method for determining the minima or maxima of an objective function whose values are only available via noise-corrupted observations (e.g., \cite{Benveniste1990,Borkar2008,Kushner2003}, and references therein). Two-timescale stochastic gradient descent algorithms represent one of the most important and complex subclasses of stochastic gradient descent methods. These algorithms consist of two coupled recursions, which evolve on different timescales (e.g., \cite{Borkar2008,Borkar1997,Tadic2004}). 
In discrete time, this approach has found success in a wide variety of applications, including deep learning \cite{Heusel2017}, reinforcement learning \cite{Karmakar2018,Konda1999,Konda2003},
 signal processing \cite{Bhatnagar2001}, 
 optimisation \cite{Doan2017}, 
 and statistical inference \cite{Yang2019}. Consequently, the analysis of its asymptotic properties has been the subject of a large number of papers (e.g., \cite{Borkar2008,Borkar1997,Karmakar2018,Konda1999,Konda2004,Mokkadem2006,Tadic2004}). 

\subsubsection{Stochastic Gradient Descent in Continuous Time}
Although these papers provide an excellent insight, they only explicitly consider two-timescale algorithms in discrete time. Indeed, to the best of our knowledge, there are no existing works which explicitly consider the almost sure convergence of two-timescale stochastic gradient descent algorithms in continuous time. Even upon restriction to the single timescale case, asymptotic results for continuous-time stochastic approximation are somewhat sparse, and generally apply only to algorithms with relatively simple dynamics (e.g., \cite{Chen1982,Chen1994,Lazrieva2008,Sen1978,Valkeila2000,Yin1993}). 
There are, however, some notable recent exceptions. 
In particular, almost sure convergence of a continuous-time stochastic gradient descent algorithm for the parameters of a fully observed diffusion process was recently established in \cite{Sirignano2017a}, and has since been extended to partially observed and jump diffusion processes \cite{Surace2019}, and jump diffusion processes \cite{Bhudisaksang2020}.

In addition to the mathematical interest, there are several reasons for considering these algorithms in continuous time. Firstly, models in engineering, finance, and the natural sciences are commonly formulated in continuous time, and thus it is natural also to formulate the corresponding statistical learning algorithms in continuous time. In addition, continuous time algorithms are a very good approximation to their discrete time analogues in cases where the sampling is very frequent (e.g., \cite{Levanony1994,Yin1993}). Furthermore, they may highlight, or even overcome, problems which arise in discrete time algorithms when the sampling rate increases \cite{Moore1988}, 
including ill conditioning \cite{Salgado1988}, biased estimates \cite{Chen1980}, or even divergence \cite{Moore1988,Sirignano2017a}. 

In practice, it is evident that any stochastic gradient scheme in continuous time must be discretised. Thus, when designing statistical learning algorithms for continuous time models such as \eqref{signal_finite_dim_inf} - \eqref{obs_finite_dim_inf} (see Section \ref{sec:apps}), it is natural to ask why we would prefer to use a discrete time approximation of a continuous-time stochastic gradient descent algorithm over the traditional approach, which first discretises the continuous-time model, and then applies a classical discrete-time stochastic gradient descent algorithm.   We advocate the first approach for several reasons. Firstly, it allows one to directly apply any appropriate numerical discretisation scheme to the theoretically correct statistical learning equations. This enables direct control of the numerical error of the resulting algorithm, and can result in more accurate and robust parameter updates (see, for example, \cite[Section 6.1]{Sirignano2017a}). Indeed, there is no guarantee that discretising the model dynamics using a numerical scheme with certain numerical properties (e.g., higher order accuracy in time), and then applying traditional stochastic gradient descent, will result in a statistical learning algorithm which also has these properties. On the other hand, one can ensure that the desired numerical properties hold by applying the discretisation of choice directly to the true continuous time learning equation. Another advantage of this approach is that it can lead to more computationally efficient updates, particularly when the dimensions of the model are significantly larger than the number of model parameters. We will present one such example in Section \ref{sec:advection}.  Finally, we note that other important properties of the continuous time model, such as its invariant measure, will not necessarily be shared by the discretised model.

\subsection{Applications} \label{sec:apps}
 Among the many applications for the two-timescale, continuous time stochastic gradient descent studied in this paper, we are primarily motivated by an important bilevel optimisation problem arising in the following family of partially observed diffusion processes

\begin{subequations}
\begin{alignat}{2}
\mathrm{d}x(t) &= A(\theta,x(t))\mathrm{d}t + B(\theta,x(t))\mathrm{d}v(t)~,~~~&&x(0)=x_0, \label{signal_finite_dim_inf} \\
\mathrm{d}y(t) &= C(\theta,\boldsymbol{o},x(t))\mathrm{d}t + \mathrm{d}w(t)~,~~~&&y(0)=0,\label{obs_finite_dim_inf}
\end{alignat}
\end{subequations}
where
$x(t)$ denotes a hidden $\mathbb{R}^{n_x}$-valued signal process, 
$y(t)$ denotes a $\mathbb{R}^{n_y}$-valued observation process, 
and 
$v(t)$ and 
$w(t)$ are independent $\mathbb{R}^{n_x}$- and $\mathbb{R}^{n_y}$-valued Wiener processes,
with incremental covariances ${Q}(\theta)\in\mathbb{R}^{n_x\times n_x}$, 
 $R(\boldsymbol{o})\in\mathbb{R}^{n_y\times n_y}$, which correspond to the signal noise and the measurement noise, respectively. 
Here,    
$\theta\in\Theta\subseteq\mathbb{R}^{n_{\theta}}$ denotes an $n_{\theta}$-dimensional parameter, and $\smash{\boldsymbol{o}=\{\boldsymbol{o}_i\}_{i=1}^{n_y}}$ $\smash{\in\Omega^{n_y}\subseteq(\mathbb{R}^{n_{\boldsymbol{o}}})^{n_y}}$ denotes a set of $n_y$ sensor locations, where $\boldsymbol{o}_i \in\Omega\subseteq \mathbb{R}^{n_{\boldsymbol{o}}}$ for $i=1,\dots,n_y$. 
We will assume that, for all $\theta\in\Theta$, and for all $\boldsymbol{o}\in\Omega^{n_y}$, the initial conditions $x_0\sim p_0(\theta,\boldsymbol{o})$ are independent of $\{w(t)\}_{t\geq0}$ and $\{v(t)\}_{t\geq 0}$. We will also assume that $A(\theta,\cdot)$
, $B(\theta,\cdot)$
, and $C(\theta,\boldsymbol{o},\cdot)$ 
are measurable functions which ensure the existence and uniqueness of strong solutions to these equations for all $t\geq 0$ (e.g., \cite{Bain2009}).

This setting is familiar in classical filtering theory, where the problem is to determine the conditional probability law of the latent signal process, given the history of observations $\mathcal{F}_t^{Y}=\sigma\{y(s):0\leq s\leq t\}$, under the assumption that the parameters are known, and the sensor locations are fixed. In most practical situations of interest, however, the model parameters are unknown, and must be inferred from the data. 
Moreover, the sensors are often not fixed, in which case it may be possible to reduce the uncertainty in the state estimate by determining an optimal sensor placement. In this paper, we aim to perform both of these challenging tasks simultaneously.

\subsubsection{Online Parameter Estimation}
The problem of parameter estimation for continuous-time, partially observed diffusion processes is somewhat well studied, particularly in the offline setting (e.g., \cite{Dabbous1992}). 
We note, however, that the majority of literature on this subject has been written for discrete-time, partially observed processes (e.g., \cite{Cappe2005}) 
and, to a lesser extent, continuous-time, fully observed processes (e.g., \cite{Bishwal2008,Kutoyants1984,Levanony1994}).
 We are primarily concerned with online parameter estimation methods, which recursively estimate the unknown model parameters based on the continuous stream of observations. 
Perhaps the most common approach to this task is recursive maximum likelihood (RML), which uses stochastic gradient descent to recursively seek the value of $\theta$ which maximises an asymptotic log-likelihood function 
(e.g., \cite{Gerencser1984,Gerencser2009,Surace2019}). 
Recently, almost sure convergence of this method in continuous-time, for a non-linear, partially observed finite-dimensional ergodic diffusion process was established in \cite{Surace2019}, extending the results previously obtained for the linear case in \cite{Gerencser1984,Gerencser2009}.
 
\subsubsection{Optimal Sensor Placement}
In contrast to online parameter estimation, the problem of optimal sensor placement for state estimation 
has been studied by a very large number of authors, and in a wide variety of contexts. 
Arguably the first mathematically rigorous treatment of this problem for linear systems was provided by Athans \cite{Athans1972},  
who formulated it as an application of optimal control on the 
Ricatti equation governing the covariance of the optimal filter (see also \cite{Meier1967}). 
Under this framework, sensor locations are treated as control variables, and the optimal sensor locations are obtained as the minima of a suitable objective function, typically defined as the trace of the filter covariance at some finite time (e.g., \cite{Wu2016}), 
or the integral of the trace of the filter covariance over some finite time interval (e.g., \cite{Chen1975}).
One can also consider optimal sensor placement with respect to asymptotic versions of these functions (e.g. \cite{Zhang2018}). 


\subsubsection{Joint Online Parameter Estimation and Optimal Sensor Placement}
In the vast majority of practical applications, parameter estimation and optimal sensor placement are both highly relevant. Moreover, they are often inter-dependent, in the sense that the optimal sensor placement may depend, to a greater or lesser extent, on the current parameter estimate. It would thus be highly convenient to tackle these two problems together and, if possible, in an online fashion.  In fact, doing so may lead to significant performance improvements \cite{Sharrock2020}. 
This is naturally formulated as a bilevel optimisation problem, in which the two objective functions are given by the asymptotic log-likelihood and the asymptotic sensor placement objective, respectively. The theoretical analysis of this problem is significantly complicated, however, by the dynamics in the state-space model \eqref{signal_finite_dim_inf} - \eqref{obs_finite_dim_inf}.


\subsection{Contributions}

\subsubsection{Convergence of Two-Timescale Stochastic Gradient Descent in Continuous Time} In this paper, we establish the almost sure convergence of two-timescale stochastic gradient descent algorithms in continuous time, under general noise and stability conditions. 
We consider algorithms with additive, state-dependent noise, and, importantly, also those with non-additive, state-dependent noise. In the second case, our analysis is carried out under the assumption that the non-additive noise can be represented by an ergodic diffusion process, controlled by the algorithm states. To our knowledge, this is the first rigorous analysis of a two-timescale stochastic approximation algorithm with Markovian dynamics in continuous time.

Our proof of these results closely follows the classical ODE method (e.g., \cite{Benveniste1990,Borkar2000,Kushner2003,Ljung1977}), adapted appropriately to the continuous time setting \cite{Chen1987,Kushner1978}. In the Markovian noise case, it also draws upon well known regularity results relating to the solution of the Poisson equation associated with the infinitesimal generator of the ergodic diffusion process \cite{Pardoux2003,Pardoux2001}. The obtained results cover a broad class of non-linear, two-timescale stochastic gradient descent algorithms in continuous time.  In particular, they can be applied to the stochastic gradient descent algorithm proposed for the joint online parameter estimation and optimal sensor placement (see below). They also include, upon restriction to a single timescale, the continuous-time stochastic gradient descent algorithms recently studied in \cite{Sirignano2017a,Surace2019}. 

\subsubsection{Joint Online Parameter Estimation and Optimal Sensor Placement}
On the basis of our theoretical results, we also propose a solution to the problem of joint online parameter estimation and optimal sensor placement in the form of a two-timescale, stochastic gradient descent algorithm. Under suitable conditions on the process consisting of the latent signal process, the filter, and the filter derivatives, we establish almost sure convergence of the online parameter estimates and recursive optimal sensor placements generated by this algorithm to the stationary points of the asymptotic log-likelihood and the asymptotic filter covariance, respectively. The effectiveness of this algorithm is demonstrated via two numerical examples: a one-dimensional, partially observed stochastic differential equation (SDE) of Bene\v{s} class, and a high-dimensional, partially observed advection-diffusion equation. 

\subsection{Paper Organisation}
The remainder of this paper is organised as follows. In Section \ref{subsec:main1}, we analyse the convergence of continuous-time, two-timescale stochastic gradient descent algorithms with additive noise. In Section \ref{subsec:main2}, we extend our analysis to continuous-time, two-timescale stochastic gradient descent algorithms with Markovian dynamics. In Section \ref{sec:RML_ROSP}, we apply these results to the problem of joint online parameter estimation and optimal sensor placement. In particular, we obtain a continuous-time, two-timescale stochastic gradient descent algorithm for this problem, and prove the almost sure convergence of the recursive parameter estimates and the recursive sensor placements to the stationary points of the asymptotic log-likelihood and the asymptotic sensor objective function, respectively. In Section \ref{sec:numerics}, we provide numerical examples illustrating the performance of the proposed algorithm. Finally, in Section \ref{sec:conclusions}, we offer some concluding remarks. 

\section{Main Results}
\label{sec:main}

We will assume, throughout this section, that $(\Omega,\mathcal{F},\mathbb{P})$ is a complete probability space, equipped with a filtration $(\mathcal{F}_t)_{t\geq 0}$ which satisfies the usual conditions. 

\subsection{Two Timescale Stochastic Gradient Descent in Continuous Time}
\label{subsec:main1}

 Let $f,g:\mathbb{R}^{d_1}\times\mathbb{R}^{d_2}\rightarrow\mathbb{R}$ be continuously differentiable functions. Suppose that, for any inputs $\{\alpha(t)\}_{t\geq 0}$, $\{\beta(t)\}_{t\geq 0}$, it is possible to obtain noisy estimates of $\nabla_{\alpha}f$ and $\nabla_{\beta}g$ according to 
\begin{subequations}
\begin{align}
\mathrm{d}z_1(t)& = \nabla_{\alpha}f(\alpha(t),\beta(t))\mathrm{d}t + \mathrm{d}\xi_1(t) \label{obs1} \\
\mathrm{d}z_2(t)& = \nabla_{\beta}g(\alpha(t),\beta(t))\mathrm{d}t + \mathrm{d}\xi_2(t)  \label{obs2}
\end{align}
\end{subequations}
 where $\{\xi_1(t)\}_{t\geq 0}$ and $\{\xi_2(t)\}_{t\geq 0}$ are $\mathbb{R}^{d_1}$ and $\mathbb{R}^{d_2}$ valued continuous semi-martingales on $(\Omega,\mathcal{F},\mathbb{P})$, which are assumed to be measurable, random functions of $\{\alpha(s)\}_{0\leq s\leq t}$ and $\{\beta(s)\}_{0\leq s< t}$.\footnote{ That is, in a slight abuse of notation, we write $\{\xi_i(t)\}_{t\geq 0}$, $i=1,2$, to denote $\{\xi_i(t,\alpha(t),\beta(t))_{t\geq 0}$.}\textsuperscript{,}\footnote{ In order to aid intuition, it is instructive to consider the formal time derivative of these measurement equations, viz  
\begin{subequations}
\begin{align}
\dot{z}_1(t) &= \nabla_{\alpha} f(\alpha(t),\beta(t)) + \dot{\xi}_1(t),\label{measurement1} \\
\dot{z}_2(t) &=\nabla_{\beta}g(\alpha(t),\beta(t)) + \dot{\xi}_2(t). \label{measurement2} 
\end{align}
\end{subequations}
This formulation, while lacking rigour, is useful in order to emphasise the connection with the standard form of noisy gradient measurements assumed in (two-timescale) stochastic approximation algorithms in discrete time (e.g., \cite{Borkar2008,Tadic2004}).}  The functions $f$ and $g$ are to be regarded as the objective functions in \eqref{opt2_}, while the semi-martingales $\{\xi_i(t)\}_{t\geq 0}$, $i=1,2$, can be considered as additive noise. 
On the basis of these noisy observations, it is natural to seek the stationary points of $f$ and $g$ via the following algorithm:
\begin{subequations}
\begin{align}
\mathrm{d}\alpha(t) &= -\gamma_{1}(t)\left[\nabla_{\alpha} f(\alpha(t),\beta(t))\mathrm{d}t + \mathrm{d}\xi_1(t)\right],~~~t\geq0,~~~\alpha(0)=\alpha_0, \label{eq1} \\
\mathrm{d}\beta(t) &= -\gamma_{2}(t)\left[\nabla_{\beta} g(\alpha(t),\beta(t))\mathrm{d}t + \mathrm{d}\xi_2(t)\right],~~~t\geq0,~~~\beta(0)=\beta_0, \label{eq2}
\end{align}
\end{subequations}
where $\{\gamma_i(t)\}_{t\geq 0}$, $i=1,2$, are positive, non-increasing, deterministic functions known as the learning rates; 
and $\alpha_0\in\mathbb{R}^{d_1}$, $\beta_0\in\mathbb{R}^{d_2}$ are random variables on $(\Omega,\mathcal{F},\mathbb{P})$. 
 We will refer to this algorithm as two-timescale stochastic gradient descent in continuous time. It represents the continuous time, gradient descent analogue of the two-timescale stochastic approximation algorithm originally introduced in \cite{Borkar1997}, and since analysed in numerous works (e.g., \cite{Borkar2008}). It can also be considered a two-timescale generalisation of the continuous time stochastic approximation algorithms introduced in \cite{Driml1961}, and later studied in, for example, \cite{Chen1982,Nevelson1976,Sen1978,Yin1993}.  

Before we proceed, it is worth noting that Algorithm \eqref{eq1} - \eqref{eq2} is not the only possible two-timescale stochastic gradient descent scheme that one can use to simultaneously optimise $f(\alpha,\beta)$ and $g(\alpha,\beta)$. This algorithm is certainly a natural choice if one only has access to noisy estimates of the partial derivatives $\nabla_{\alpha}f(\alpha,\beta)$ and $\nabla_{\beta} g(\alpha,\beta)$, and is interested in solving the bilevel optimisation problem in \eqref{opt2_}. It is less well suited, however, to the stronger version of the bilevel optimisation problem in \eqref{global_min}, since it ignores the dependence of the true upper level objective $f(\alpha,\beta^{*}(\alpha))$ on $\alpha$ in its second argument. As such, if one has access to additional gradient information, then it may be preferable to use higher order updates to capture the dependence on $\beta^{*}(\alpha)$. We provide details of one such approach in Appendix \ref{theorem1_ext1} (see also \cite{Hong2020} in discrete time).

We will analyse Algorithm \eqref{eq1} - \eqref{eq2} under the following set of assumptions. Broadly speaking, these represent the continuous time analogues of standard assumptions used in the almost sure convergence analysis of two-timescale stochastic approximation algorithms in discrete time (see, e.g., \cite[Chapter 6]{Borkar2008} or \cite{Tadic2004}).

\begin{assumption} \label{assumption1}
The learning rates $\{\gamma_i(t)\}_{t\geq 0}$, $i=1,2$, are positive, non-increasing functions, which satisfy
\begin{subequations} 
\begin{align}
\lim_{t\rightarrow\infty}\gamma_1(t)=\lim_{t\rightarrow\infty}\gamma_2(t) = \lim_{t\rightarrow\infty}\frac{\gamma_1(t)}{\gamma_2(t)}&=0, \\
\int_0^{\infty}\gamma_1(t)\mathrm{d}t = \int_0^{\infty}\gamma_2(t)\mathrm{d}t &= \infty. 
\end{align}
\end{subequations}
\end{assumption}

This assumption relates to the asymptotic properties of the learning rates \linebreak $\{\gamma_i(t)\}_{t\geq 0}$, $i=1,2$. It is the continuous time analogue of the standard step-size assumption used in the convergence analysis of two-timescale stochastic approximation algorithms in discrete time (e.g., \cite{Borkar2008,Borkar1997,Tadic2004}). In particular, this assumption implies that the process $\{\alpha(t)\}_{t\geq 0}$ evolves on a slower time-scale than the process $\{\beta(t)\}_{t\geq 0}$.  Thus, intuitively speaking, the fast component, $\beta(\cdot)$, will see the slow component, $\alpha(\cdot)$, as quasi-static, while the slow component will see the fast component as essentially equilibrated \cite{Borkar1997}. A standard choice of step sizes which satisfies this assumption is $\gamma_1(t) = \gamma_1^0(\delta_1+t)^{-\eta_1}$, $\gamma_2(t) = \gamma_2^0(\delta_2+t)^{-\eta_2}$ for $t\geq 0$, where $\gamma_1^0,\gamma_2^0>0$ and $\delta_1,\delta_2>0$ are positive constants, and $\eta_1,\eta_2\in(0,1]$ are constants such that $\eta_1>\eta_2$.

\begin{assumption} \label{assumption2}
The functions $\nabla_{\alpha} f:\mathbb{R}^{d_1}\times\mathbb{R}^{d_2}\rightarrow\mathbb{R}^{d_1}$ and $\nabla_{\beta} g:\mathbb{R}^{d_1}\times\mathbb{R}^{d_2}\rightarrow\mathbb{R}^{d_2}$ are locally Lipschitz continuous. 
\end{assumption} 

This assumption relates to the smoothness of the objective functions $f(\cdot)$ and $g(\cdot)$, and is standard both in two-timescale stochastic approximation algorithms in discrete time \cite{Borkar1997,Karmakar2018,Konda2003a}, and single-timescale stochastic approximation algorithms in continuous time \cite{Chen1994,Sen1978,Yin1993}, 
 although slightly weaker assumptions may also be possible (see, e.g., \cite{Lazrieva2008}). This assumption implies, in particular, that the functions $\nabla_{\alpha} f$ and $\nabla_{\beta} g$ locally satisfy linear growth conditions. 

\begin{assumption} \label{assumption3}
For all $T\in[0,\infty)$, the noise processes $\{\xi_i(t)\}_{t\geq 0}$, $i=1,2$, satisfy 
\begin{align}
\lim_{s\rightarrow\infty} \sup_{t\in[s,s+T]}\left|\left| \int_{s}^{t} \gamma_i(v)\mathrm{d}\xi_i(v)\right|\right|&=0~~~\text{a.s.}
\end{align}
\end{assumption}

This assumption relates to the asymptotic properties of the additive noise processes $\smash{\{\xi_i(t)\}_{t\geq 0}}$, $i=1,2$. It can be regarded as the continuous-time, two-timescale generalisation of the Kushner-Clark condition \cite{Kushner1978}. This assumption is significantly weaker than the noise conditions adopted in many of the existing results on almost sure convergence of continuous-time, single-timescale stochastic approximation algorithms. In particular, it includes the cases when $\{\xi_i(t)\}_{t\geq 0}$, $i=1,2$, are continuous (local) martingales \cite{Sen1978},\footnote{We should remark that the case when the noise process is a local martingale is also considered by Lazrieva et al. (e.g., \cite{Lazrieva2008}) and Valkeila et al. (e.g., \cite{Valkeila2000}).
In these works, however, there is no requirement that this local martingale is continuous.} continuous finite variation processes with zero mean \cite{Yin1993}, or diffusion processes \cite{Chen1994}. It also holds, under certain additional assumptions, for algorithms with Markovian dynamics \cite{Sirignano2017a,Sirignano2020a,Surace2019}. The discrete-time analogue of this condition first appeared in \cite{Tadic2004}, weakening the noise condition originally used in \cite{Borkar1997}. 

\begin{assumption} \label{assumption4}
The iterates $\{\alpha(t)\}_{t\geq 0}$, $\{\beta(t)\}_{t\geq 0}$ are almost surely bounded:
\begin{equation}
\sup_{t\geq 0} \left[||\alpha(t)|| + ||\beta(t)||\right]<\infty.
\end{equation}
\end{assumption}

This assumption is necessary in order to prove almost sure convergence. In general, however, it is far from automatic, and not very straightforward to establish. 
Indeed, sufficient conditions tend to be highly problem specific, or else somewhat restrictive (e.g., \cite{Lazrieva2008,Sen1978,Valkeila2000}). To circumvent this issue, a common approach is to include a truncation or projection device in the algorithm, which ensures that the iterates remain bounded with probability one,  at the expense of an additional error term (e.g., \cite{Chen1987,Kushner2003,Surace2019}). 
This may, however, also introduce spurious fixed points on the boundary of the domain (e.g., \cite{Surace2019}). An alternative method, which avoids this shortcoming, is the `continuous time stochastic approximation procedure with randomly varying truncations', originally introduced in \cite{Chen1987}. 
This procedure can be partially extended to the two-timescale framework, to establish almost sure boundedness of iterates on the fast-timescale. 
It is currently unclear, however, how to fully extend this approach to the two-timescale setting.
Another common approach is to omit the boundedness assumption entirely, and instead state asymptotic results which are local in nature 
That is, which hold almost surely on the event $\Lambda = \{\sup_{t\geq 0} ||\alpha(t)||<\infty\}\cap\{\sup_{t\geq 0}||\beta(t)||<\infty\}$. In the single-timescale setting, it is often then straightforward to establish the global counterparts of these results, by combining them with existing methods for verifying stability (e.g., \cite{Benveniste1990,Borkar2000,Kushner1997}). In contrast, the stability of two-timescale stochastic approximation algorithms has thus far not received much attention. Indeed, to the best our knowledge, the only existing result along these lines is \cite{Lakshminarayanan2017}.

\begin{assumption} \label{assumption5}
For all $\alpha\in\mathbb{R}^{d_1}$, the ordinary differential equation (ODE)
\begin{equation}
\dot{\beta}(t) = -\nabla_{\beta} g(\alpha,\beta(t))
\end{equation} 
has a discrete, countable set of isolated equilibria $\{\beta_i^{*}\}_{i\geq 1} = \{\beta^{*}_i(\alpha)\}_{i\geq 1}$, where $\beta^{*}_i:\mathbb{R}^{d_1}\rightarrow\mathbb{R}^{d_2}$, $i\geq 1$, are locally Lipschitz-continuous maps. 
\end{assumption}

This is a stability condition relating to the fast recursion. It is somewhat weaker than the standard fast-timescale assumption used in the analysis of discrete-time, two-timescale stochastic approximation algorithms, which requires that this equation must have a unique global asymptotically stable equilibrium (e.g., \cite{Borkar1997,Konda1999,Tadic2004}). We note, however, that a similar assumption has previously appeared in \cite{Karmakar2018}. It may be possible to weaken this assumption further - that is, to remove the requirement for a discrete, countable set of equilibria - using the tools recently established in \cite{Tadic2015}. There, in the context of discrete-time, single-timescale stochastic gradient descent, almost sure single-limit point convergence is proved in the case of multiple or non-isolated equilibria, using tools from differential geometry (namely, the Lojasiewicz gradient inequality). It remains an open problem to determine whether these results can be extended to the continuous-time or the two-timescale setting.

In order to state our final assumption, we will require the following definitions. Let $x\in\mathbb{R}^{d}$, and let $h:\mathbb{R}^{d}\rightarrow\mathbb{R}^d$. Consider an ODE of the form $\dot{x}(t) = h(x(t))$. We say that a set $A\subset\mathbb{R}^{d}$ is {invariant} for this ODE if any trajectory $x(t)$ satisfying $x(0)\in A$ satisfies $x(t)\in A$ for all $t\in\mathbb{R}$. In addition, we say that $A$ is internally chain transitive for this ODE if for any $x\in A$, and for any $\varepsilon>0$, $T>0$, there exists $n\in\mathbb{N}$, points $x_0,x_1,\dots,x_n =x$ in $A$, and times $t_1,\dots,t_n \geq T$, such that, for all $1\leq i \leq n$, the trajectory of the ODE initialised at $x_{i-1}$ is in the $\varepsilon$-neighbourhood of $x_{i}$ at time $t_{i}$. We can now state our final assumption.

\begin{assumption} \label{assumption6}
For all $i\geq 1$, the only internally chain transitive invariant sets of the ordinary differential equation
\begin{equation}
\dot{\alpha}(t) = -\nabla_{\alpha} f(\alpha(t),\beta_{i}^{*}(\alpha(t))) \label{fODE}
\end{equation} 
are its equilibrium points. 
\end{assumption}

This is a stability condition relating to the slow recursion. It can be regarded as a slightly weaker version of standard slow-timescale assumption used in the analysis of two-timescale stochastic approximation algorithms, which stipulates that this ordinary differential equation must have a unique, globally asymptotically stable equilibrium (e.g., \cite{Borkar2008,Borkar1997,Konda1999}). This assumption is required in order to rule out the possibility that \eqref{fODE} admits other internally chain transitive invariant sets aside from equilibria, such as cyclic orbit chains (see \cite{Benaim1999}). We remark that, under additional assumptions on $\beta_{i}^{*}(\cdot)$, 
one can replace this with the weaker assumption that \eqref{fODE} has a discrete, countable set of isolated equilibria. 
Unfortunately, without additional assumptions on $\beta_{i}^{*}(\cdot)$, one cannot use this condition directly, since $f(\cdot,\beta_{i}^{*}(\cdot)$ is not, in general, a strict Lyapunov function for \eqref{fODE}. We discuss this point in further detail in Appendix \ref{theorem1_ext1}.

We conclude this commentary with the remark that our condition(s) on the objective function(s) are, broadly speaking, also slightly more general than those adopted in many of the existing results on the convergence of continuous-time, single-timescale stochastic approximation algorithms. In particular, we do not insist on the existence of a unique root for the gradient of the objective functions, as is the case in \cite{Chen1994,Lazrieva2008,Sen1978,Valkeila2000}.

Our main result on the convergence of Algorithm \eqref{eq1} - \eqref{eq2} is contained in the following theorem.

\begin{theorem} \label{theorem1}
Assume that Assumptions \ref{assumption1} - \ref{assumption6} hold. Then, almost surely,  
\begin{equation}
\lim_{t\rightarrow\infty}\nabla_{\alpha}f(\alpha(t),\beta(t)) = \lim_{t\rightarrow\infty}\nabla_{\beta}g(\alpha(t),\beta(t)) = 0.
\end{equation}
\end{theorem}

\begin{proof}
See Appendix \ref{sec:proof1}.
\end{proof}

Our proof of Theorem \ref{theorem1} follows the ODE method. This approach was first introduced in \cite{Ljung1977}, and extensively developed by Kushner et al. (e.g., \cite{Benveniste1990,Kushner2003,Kushner1978,Kushner1984}) and later Bena\"im et al. \cite{Benaim1996,Benaim1999}. 
 It was first used to prove almost sure convergence of a two-timescale stochastic approximation algorithm in \cite{Borkar1997}, which considered a discrete-time stochastic approximation algorithm with state-independent additive noise. It has since also been used to establish the convergence of more general discrete-time, two-timescale stochastic approximation algorithms \cite{Karmakar2018,Konda1999,Tadic2004}. In the context of continuous-time, single-timescale stochastic approximation, this method of proof has largely been neglected, with several notable exceptions \cite{Chen1982,Kushner1978,Yin1993}. 
 While other continuous-time approaches (e.g., \cite{Chen1994,Sen1978,Lazrieva2008,Valkeila2000}) 
 may be more direct, they may also require slightly more restrictive assumptions. Moreover, it is unclear whether these approaches can straightforwardly be adapted to the two-timescale setting, or even to more complex single-timescale algorithms, such as those with Markovian dynamics (e.g., \cite{Sirignano2017a}).  One other advantage of this method of proof is that it is straightforwardly adapted to other variations of Algorithm \eqref{eq1} - \eqref{eq2}, as discussed prior to the statement of our assumptions. In Appendix \ref{theorem1_ext1}, we show rigorously how to use this approach to establish an almost sure convergence result for one such algorithm. 
 
 We should emphasise, at this point, that Theorem \ref{theorem1} establishes almost sure convergence precisely to the {stationary points} of the objective functions $f$ and $g$. In particular, the stated assumptions do not guarantee convergence to the set of local (or global) minima. On this point, two remarks are pertinent. Firstly, results of this type are standard in the recent literature on stochastic gradient descent in continuous time (e.g., \cite{Sirignano2017a,Surace2019}), and the more classical literature on two-timescale stochastic approximation (e.g., \cite{Borkar2008}). Secondly, under additional assumptions, it should be possible to extend our analysis to guarantee that our algorithm converges almost surely to local minima of the two objective functions. Indeed, when a single timescale is considered, there are several existing `avoidance of saddle' type results \cite{Brandiere1998,Ge2015,Mertikopoulos2020,Pemantle1990}. While no explicit results of this type exist in the two-timescale framework, we outline details of the (minimal) assumptions which would be required to obtain such a result in Appendix \ref{theorem1_ext2}, and discuss briefly how they can be used together with the results of this paper. 
 

 Another natural extension of Theorem \ref{theorem1} (and, later, Theorem \ref{theorem1a}) it to establish convergence rates for the continuous-time, two-timescale stochastic gradient descent algorithm. While this is beyond the scope of this paper, let us make some brief remarks regarding such an extension, with reference to some relevant literature. In discrete time, convergence rates of two-timescale stochastic approximation algorithms are the subject of several classical papers (e.g., \cite{Konda2004,Mokkadem2006}), and have also received renewed attention in recent years (e.g., \cite{Dalal2020,Doan2021,Xu2019}). In light of these results (see, in particular, \cite{Mokkadem2006}), it is reasonable to conjecture that, under the appropriate additional assumptions, it is possible to establish a central limit theorem for the continuous-time two-timescale algorithm of the form
\begin{equation}
\begin{pmatrix} \gamma_{1}^{-1/2}(t) (\alpha(t)-\alpha^{*}) \\ \gamma_{2}^{-1/2}(t)  (\beta(t)-\beta^{*}) \end{pmatrix} \stackrel{\mathcal{D}}{\longrightarrow} \mathcal{N}\left(\begin{pmatrix} \vphantom{\gamma_{1}^{-1/2}(t)} 0 \\ \vphantom{\gamma_{1}^{-1/2}(t)} 0 \end{pmatrix}, \begin{pmatrix} \Sigma_{\alpha}  & \vphantom{\gamma_{1}^{-1/2}(t)} 0 \\ \vphantom{\gamma_{1}^{-1/2}(t)} 0 & \Sigma_{\beta} \end{pmatrix} \right).
\end{equation}
where $\Sigma_{\alpha}\in\mathbb{R}^{d_1\times d_1},\Sigma_{\beta}\in\mathbb{R}^{d_2\times d_2}$ are matrices defined in terms of the Hessians of the objective functions $f$ and $g$, and terms appearing in the definition of the additive noise processes $\xi_1$ and $\xi_2$. While many of the techniques used to establish such results in discrete time carry over straightforwardly to continuous time, others require much more careful adaptation, and do not have direct analogues. Here, we suggest that some of the results established in \cite{Chen1994,Yin1993} may prove useful. In the presence of Markovian dynamics (see Section \ref{subsec:main2}), the analysis required to establish convergence rates is even more involved. This being said, recent results in \cite{Sirignano2020a} seem very promising in this direction.


\subsection{Two Timescale Stochastic Gradient Descent in Continuous Time with Markovian Dynamics}
\label{subsec:main2}
 Using the results obtained in Section \ref{subsec:main1}, we now consider the situation in which the noisy estimates of $\nabla_{\alpha}f$ and $\nabla_{\beta}g$ are governed by some additional continuous-time dynamical process.  In particular, we now analyse the convergence of the algorithm
\begin{subequations}
\begin{align}
\mathrm{d}\alpha(t) = -\gamma_{1}(t)\left[F(\alpha(t),\beta(t),{\mathcal{X}}(t))\mathrm{d}t + \mathrm{d}\zeta_1(t)\right],~~~t\geq0,~~~\alpha(0)=\alpha_0, \label{alg2_1} \\
\mathrm{d}\beta(t) = -\gamma_2(t)\left[G(\alpha(t),\beta(t),{\mathcal{X}}(t))\mathrm{d}t + \mathrm{d}\zeta_2(t)\right],~~~t\geq0,~~~\beta(0)=\beta_0, \label{alg2_2}
\end{align}
\end{subequations}
where $F,G:\mathbb{R}^{d_1}\times\mathbb{R}^{d_2}\times\mathbb{R}^{d_3}\rightarrow\mathbb{R}^{d_1},\mathbb{R}^{d_2}$
are Borel measurable functions,
 $\{\zeta_i(t)\}_{t\geq 0}$, $i=1,2$, 
 are $\mathbb{R}^{d_1}$, $\mathbb{R}^{d_2}$ 
 valued continuous semi-martingales on $(\Omega,\mathcal{F},\mathbb{P})$,
 $\{\mathcal{X}(t)\}_{t\geq 0}$ is a $\mathbb{R}^{d_3}$ valued diffusion process on the same probability space, and all other terms are as defined previously.
 In this algorithm, the functions $F(\cdot)$ and $G(\cdot)$ are to be regarded as noisy estimators of $\nabla_{\alpha} f(\cdot)$ and $\nabla_{\beta} g(\cdot)$; the precise relationship between these functions will be clarified below. The semi-martingales $\smash{\{\zeta_i(t)\}_{t\geq 0}}$, $i=1,2$, can once more be considered as additive noise; while the Markov process $\smash{\{\mathcal{X}(t)\}_{t\geq 0}}$ can be regarded as non-additive noise. 

We will refer to this algorithm as two-timescale stochastic gradient descent in continuous time with Markovian dynamics. This algorithm represents the continuous time analogue of the discrete-time, two-timescale stochastic approximation algorithm with state-dependent non-additive noise analysed in \cite[Section IV]{Tadic2004}. In fact, our presentation is slightly more general than in \cite{Tadic2004}, as we also allow for the possibility of additive, state-dependent noise via the terms $\{\zeta_i(t)\}_{t\geq 0}$, $i=1,2$. This increases the number of applications in which our algorithm can be applied (see Section \ref{sec:RML_ROSP}), while not significantly complicating the analysis. The almost sure convergence of two-timescale stochastic approximation algorithms with Markovian dynamics in discrete time has also been studied, under various assumptions, in \cite{Karmakar2018,Konda2003,Konda2003a}. 
Conversely, there are no existing works which provide a rigorous analysis of two-timescale stochastic approximation algorithms with Markovian dynamics in continuous time. 

We analyse this algorithm under the assumption that $\mathcal{X} = \smash{\{{\mathcal{X}}(t)\}_{t\geq 0}}$ is a diffusion process on $\mathbb{R}^{d_3}$, controlled by the algorithm states $\smash{\{\alpha(t)\}_{t\geq 0},\{\beta(t)\}_{t\geq 0}}$. In particular, we suppose that this process evolves according to
\begin{equation}
\vphantom{\sum}\mathrm{d}{\mathcal{X}}(t) = \Phi(\alpha(t),\beta(t),{\mathcal{X}}(t))\mathrm{d}t + \Psi(\alpha(t),\beta(t),{\mathcal{X}}(t))\mathrm{d}b(t),~t\geq0,~{\mathcal{X}}(0)={\mathcal{X}}_0,\label{diffusion} 
\end{equation}
where, for all $\alpha\in\mathbb{R}^{d_1}$, $\beta\in\mathbb{R}^{d_2}$, $\Phi(\alpha,\beta,\cdot):\mathbb{R}^{d_3}\rightarrow\mathbb{R}^{d_3}$ and $\Psi(\alpha,\beta,\cdot):\mathbb{R}^{d_3}\rightarrow \mathbb{R}^{d_3\times d_4}$ are Borel measurable functions; $\mathcal{X}_{0}$ is a random variable defined on $(\Omega,\mathcal{F},\mathbb{P})$; and $\{b(t)\}_{t\geq 0}$ is a $\mathbb{R}^{d_4}$ valued Wiener process on the same probability space. 
We should remark that, whenever $\alpha\in\mathbb{R}^{d_1}$, $\beta\in\mathbb{R}^{d_2}$ are fixed, we will denote the corresponding diffusion process by $\{\mathcal{X}(\alpha,\beta,t)\}_{t\geq 0}$, making explicit the dependence on these parameters.  

Our motivation for this choice of dynamics is threefold:  firstly, the existence, uniqueness, and asymptotic properties of this class of processes are very well studied (e.g., \cite{Ikeda1989,Pardoux2001}). 
Secondly, this choice is sufficiently broad for many practical situations of interest. 
 Finally, under the assumption that $\{\mathcal{X}(\alpha,\beta,t)\}_{t\geq 0}$ is ergodic for all $\alpha\in\mathbb{R}^{d_1}$, $\beta\in\mathbb{R}^{d_2}$, with unique invariant measure $\mu_{\alpha,\beta}(\cdot)$ (see Assumption \ref{assumption2a}), one can obtain an explicit relation between the estimators $F(\cdot)$ and $G(\cdot)$ and the gradients of the objective functions $\nabla_{\alpha}f(\cdot)$ and $\nabla_{\beta}g(\cdot)$. In particular, in this case the true objective functions are often defined as ergodic averages of the noisy estimators:
\begin{subequations}
\begin{align}
\nabla_{\alpha} f(\alpha,\beta) &= \int_{\mathbb{R}^{d_3}} F(\alpha,\beta,x)\mu_{\alpha,\beta}(\mathrm{d}x), \label{ergodic_average1} \\
\nabla_{\beta} g(\alpha,\beta) &= \int_{\mathbb{R}^{d_3}} G(\alpha,\beta,x)\mu_{\alpha,\beta}(\mathrm{d}x). \label{ergodic_average2}
\end{align} 
\end{subequations}
We remark that, in general, it is not possible to obtain the unique invariant measure $\mu_{\alpha,\beta}(\cdot)$ of the ergodic diffusion process $\mathcal{X}$ in closed form, let alone compute these integrals. Thus, in the Markovian framework we typically cannot compute the gradients $\nabla_{\alpha}f$ and $\nabla_{\beta}g$ exactly, even in the absence of the additive noise processes $\{\zeta_{i}(t)\}_{t\geq 0}$, $i=1,2$.

We analyse this algorithm under the following set of assumptions.  Similarly to before, these assumptions can be viewed both as the continuous time analogues of standard assumption used for the almost sure convergence analysis of two-timescale stochastic approximation algorithms with Markovian dynamics in discrete time (e.g., \cite[Section IV]{Tadic2004}), and as the two-timescale generalisation of assumptions more recently introduced to analyse the convergence of single-timescale stochastic gradient descent algorithms with Markovian dynamics in continuous time \cite{Sirignano2017a,Surace2019}.  

\begin{assumption} \label{assumption1a}
The learning rates $\{\gamma_i(t)\}_{t\geq 0}$, $i=1,2$,  satisfy Assumption \ref{assumption1}. Furthermore, 
\begin{align}
\int_0^{\infty}\gamma^2_i(t)\mathrm{d}t<\infty,~~\int_0^{\infty}\left|\dot{\gamma}_i(t)\right|\mathrm{d}t< \infty,
\end{align}
and there exist $r_i>0$, $i=1,2$, such that $\lim_{t\rightarrow\infty} \gamma_i^2(t)t^{\frac{1}{2}+2r_i} = 0$.
\end{assumption}

This assumption represents the two-timescale generalisation of the learning rate assumptions used in the analysis of single-timescale stochastic gradient descent algorithms with Markovian dynamics in continuous time \cite{Sirignano2017a,Sirignano2020a,Surace2019}. 
This assumption is satisfied by the learning rates specified after Assumption \ref{assumption1}, under the additional condition that now $\eta_1,\eta_2\in(\frac{1}{2},1]$. 
We remark, as in \cite{Sirignano2017a}, that the condition relating to the derivatives, namely that $\smash{\int_{0}^{\infty} |\dot{\gamma}_i(t)|\mathrm{d}t<\infty}$, $i=1,2$, is satisfied automatically if the learning rates are chosen to be monotonic functions of $t$.

\begin{manualassumption}{2.2.2a} \label{assumption2a}
The process $\{\mathcal{X}(\alpha,\beta,t)\}_{t\geq 0}$ is ergodic for all $\alpha\in\mathbb{R}^{d_1}$, $\beta\in\mathbb{R}^{d_2}$, with unique invariant probability measure $\mu_{\alpha,\beta}$ on $(\mathbb{R}^{d_3},\mathbb{B}_{d_3})$, where $\mathbb{B}_{d_3}$ denotes the Borel $\sigma$-algebra on $\mathbb{R}^{d_3}$.
\end{manualassumption}

This assumption relates to the asymptotic properties of the non-additive, state-dependent noise process $\{\mathcal{X}(\alpha,\beta,t)\}_{t\geq 0}$. In the context of discrete-time stochastic approximation with Markovian dynamics, the requirement of ergodicity is relatively standard in both single-timescale (e.g., \cite{Benveniste1990,Kushner1984,Kushner1996}) and two-timescale (e.g., \cite{Konda2003,Konda2003a,Tadic2004}) settings, although slightly weaker assumptions are possible (e.g., \cite{Metivier1984}).
This assumption is also central to the existing results on the convergence of stochastic gradient descent with Markovian dynamics in continuous time \cite{Bhudisaksang2020,Sirignano2017a,Surace2019}.

\begin{manualassumption}{2.2.2b} \label{assumption2ai}
For any $q>0$, $\alpha\in\mathbb{R}^{d_1}$, $\beta\in\mathbb{R}^{d_2}$, there exists constants $K_{q},K^{\alpha}_{q},K_{q}^{\beta}>0$, such that 
\begin{subequations}
\begin{align}
\int_{\mathbb{R}^{d_3}} (1+||x||^{q})\mu_{\alpha,\beta}(\mathrm{d}x) &\leq K_q, \\
\int_{\mathbb{R}^{d_3}} (1+||x||^{q})|\nu^{(\alpha)}_{\alpha,\beta,i}(\mathrm{d}x)|&\leq K^{\alpha}_q, \\
\int_{\mathbb{R}^{d_3}} (1+||x||^{q})|\nu^{(\beta)}_{\alpha,\beta,i}(\mathrm{d}x)|&\leq K^{\beta}_q,
\end{align}
\end{subequations}
where $\smash{|\nu^{(\alpha)}_{\alpha,\beta,i}(\mathrm{d}x)|,|\nu^{(\beta)}_{\alpha,\beta,i}(\mathrm{d}x)|}$ denote the total variations of the finite signed measures 
$\smash{\nu^{(\alpha)}_{\alpha,\beta,i}=\partial_{\alpha_i}\mu_{\alpha,\beta}}$, $i=1,\dots,d_1$, and $\smash{\nu^{(\beta)}_{\alpha,\beta,i}=\partial_{\beta_i}\mu_{\alpha,\beta}}$, $i=1,\dots,d_2$.
\end{manualassumption}

This assumption relates to the regularity of the invariant measure and its derivatives. It can be regarded as a two-timescale extension of the regularity conditions used for the convergence analysis of the continuous-time, single-timescale stochastic gradient descent algorithm with Markovian dynamics in \cite{Surace2019}.\footnote{We refer to \cite[Part II]{Benveniste1990} for a detailed discussion of the corresponding conditions used in the convergence analysis of discrete-time stochastic approximation algorithms with Markovian dynamics. We remark only that, in this case, it is typical to require that the transition kernels of the Markov process satisfy certain regularity conditions, rather than the invariant measure (if this exists).} This condition ensures that the objective functions $f(\cdot)$ and $g(\cdot)$, and their first two derivatives, are uniformly bounded in both arguments.\footnote{In the analysis of discrete-time stochastic approximation algorithms with Markovian dynamics, it is not uncommon for boundedness to be assumed \emph{a priori}. See, for example, \cite{Metivier1984} in the single-timescale case, and \cite{Konda2003} in the two-timescale case.}

In order to state the remaining assumptions, we will require the following additional notation. We will say that a function $H:\mathbb{R}^{d_1}\times\mathbb{R}^{d_2}\times\mathbb{R}^{d}\rightarrow \mathbb{R}$ satisfies the polynomial growth property (PGP) if there exist $q,K>0$ such that, for all $\alpha\in\mathbb{R}^{d_1},\beta\in\mathbb{R}^{d_2}$,
\begin{equation}
|H(\alpha,\beta,x)|\leq  K(1+||x||^q).
\end{equation}
We will write $\smash{\mathbb{H}^{i+\delta,j}}(\mathbb{R}^d)$, $i,j\in\mathbb{N}$, $\delta\in(0,1)$, to denote the space of all functions $H:\mathbb{R}^{d_1}\times\mathbb{R}^{d_2}\times\mathbb{R}^{d}\rightarrow \mathbb{R}$ such that $H(\cdot,\cdot,x)\in C^j(\mathbb{R}^{d_1\times d_2})$ and $H(\alpha,\beta,\cdot)\in C^{i}(\mathbb{R}^{d})$; and such that $\smash{\nabla^{i'}_{x} \nabla^{j'}_{\alpha}H(\alpha,\beta,\cdot)}$, $\smash{\nabla^{i'}_{x} \nabla^{j'}_{\beta}H(\alpha,\beta,\cdot)}$ are H\"older continuous with exponent $\delta$, uniformly in $\alpha$ and $\beta$, for $0\leq i'\leq i$, $0\leq j'\leq j$. We will also write $\smash{\mathbb{H}_c^{i+\delta,j}}(\mathbb{R}^d)$ for the subspace consisting of all $\smash{H\in\mathbb{H}^{i+\delta,j}}(\mathbb{R}^d)$ such that $H$ is centered, in the sense that $\smash{\int_{\mathbb{R}^{d_3}} H(\alpha,\beta,x)\mu_{\alpha,\beta}(\mathrm{d}x)=0}$. Finally, we will write $\bar{\mathbb{H}}^{i+\delta,j}(\mathbb{R}^d)$ to denote the subspace consisting of $H\in\mathbb{H}^{i+\delta,j}(\mathbb{R}^d)$ such that $H$ and all of its first and second derivatives with respect to $\alpha$ and $\beta$ satisfy the PGP.

\begin{manualassumption}{2.2.2c} \label{assumption4a}
There exist differentiable functions $f,g:\mathbb{R}^{d_1}\times\mathbb{R}^{d_2}\rightarrow\mathbb{R}$ such that $\nabla_{\alpha}f(\cdot)$ and $\nabla_{\beta}g(\cdot)$ are locally Lipschitz continuous, and unique Borel measurable functions $\tilde{F}:\mathbb{R}^{d_1}\times\mathbb{R}^{d_2}\times\mathbb{R}^{d_3}\rightarrow\mathbb{R}^{d_1}$, $\tilde{G}:\mathbb{R}^{d_1}\times\mathbb{R}^{d_2}\times\mathbb{R}^{d_3}\rightarrow\mathbb{R}^{d_2}$ such that, for all $\alpha\in\mathbb{R}^{d_1}$, $\beta\in\mathbb{R}^{d_2}$, $x\in\mathbb{R}^{d_3}$,
\begin{subequations}
\begin{align}
\mathcal{A}_{\mathcal{X}}\tilde{F}(\alpha,\beta,x) &= \nabla_{\alpha} f(\alpha,\beta)-F(\alpha,\beta,x), \\
 \mathcal{A}_{\mathcal{X}}\tilde{G}(\alpha,\beta,x) &= \nabla_{\beta} g(\alpha,\beta)-G(\alpha,\beta,x),
\end{align}
\end{subequations}
where $\mathcal{A}_{\mathcal{X}}$ is the infinitesimal generator of $\mathcal{X}$. In addition, the functions $\tilde{F}(\alpha,\beta,x)$ and $\tilde{G}(\alpha,\beta,x)$ are in $\bar{\mathbb{H}}^{1+\delta,2}(\mathbb{R}^{d_3})$, and their mixed first partial derivatives with respect to $(\alpha,x)$ and $(\beta,x)$ have the PGP.
\end{manualassumption}

\begin{manualassumption}{2.2.2d} \label{assumption4aiii}
The diffusion coefficient $\Psi$ has the PGP componentwise. In particular, it grows no faster than polynomially with respect to the $x$ variable.
\end{manualassumption}

\begin{manualassumption}{2.2.2e}  \label{assumption4aii}
For all $q>0$, and for all $t\geq 0$, $\mathbb{E}[||\mathcal{X}(t)||^q]<\infty$. Furthermore, there exists $K>0$ such that for all $t$ sufficiently large, 
\begin{subequations}
\begin{align}
\mathbb{E}\left[\sup_{s\leq t}||\mathcal{X}(\alpha,\beta,s)||^q\right]&\leq K\sqrt{t}~,~~~\forall \alpha\in\mathbb{R}^{d_1}~,~\forall \beta\in\mathbb{R}^{d_2}, \\
\mathbb{E}\left[\sup_{s\leq t}||\mathcal{X}(s)||^q\right]&\leq K\sqrt{t}.
\end{align}
\end{subequations}
\end{manualassumption} 

These three assumptions relate to the properties of the ergodic diffusion process $\{\mathcal{X}(\alpha,\beta,t)\}_{t\geq 0}$, and the definitions of the objective functions $f(\cdot)$ and $g(\cdot)$.  In particular, the first condition establishes the relationship between the gradients of the objective functions $\nabla_{\alpha} f(\cdot)$ and $\nabla_{\beta} g(\cdot)$ and the unbiased estimators $F(\cdot)$ and $G(\cdot)$. It also relates to the existence, uniqueness, and properties of solutions of the associated Poisson equations.  The second condition pertains to the growth properties of the ergodic diffusion process, while the third condition provides bounds on its moments. Together, these conditions ensure that error terms which arise due to the noisy estimates of $\nabla_{\alpha}f(\cdot)$ and $\nabla_{\beta}g(\cdot)$, tend to zero sufficiently quickly as $t\rightarrow\infty$. They are therefore essential,  whether or not they are required explicitly,  to existing results on the almost sure convergence of continuous-time stochastic gradient descent with Markovian dynamics \cite{Bhudisaksang2020,Sirignano2017a,Surace2019}. 

The discrete-time analogues of these conditions, and variations thereof, also appear in almost all of the existing convergence results for stochastic approximation algorithms with Markovian dynamics in discrete time (e.g., \cite{Benveniste1990,Kushner2003,Metivier1984,Tadic2015}), including those with two-timescales (e.g., \cite{Konda2003,Konda2003a,Tadic2004}).\footnote{In discrete-time, these conditions are often stated in terms of the Markov transition kernel. They were first introduced in \cite[Section III]{Metivier1984} (see also \cite[Part II]{Benveniste1990}), and later generalised in \cite{Kushner2003}.}\textsuperscript{,}\footnote{Interestingly, the final two equations in Assumption \ref{assumption4aii} are peculiar to the continuous-time setting. In discrete time, only the first moment bound appears in the analysis of algorithms with Markovian dynamics (e.g., \cite{Benveniste1990,Metivier1984,Tadic2015}), including the two-timescale case (e.g., \cite{Konda2003,Tadic2004}).} Our particular choice of assumptions can be considered as the two-timescale, continuous-time generalisation of the conditions appearing in \cite[Section III]{Metivier1984} and \cite[Part II]{Benveniste1990}. It also closely resembles a continuous-time analogue of the assumptions used in \cite[Section IV]{Tadic2004} for a discrete-time, two-timescale stochastic approximation algorithm with non-additive, state-dependent noise.

 It remains only to provide our assumptions on the additive noise processes \linebreak $\{\zeta_i(t)\}_{t\geq 0}$, $i=1,2$. In order to state these assumptions, we will now require an explicit form for these semi-martingales. In particular, we will assume that they evolves according to
\begin{equation}
\mathrm{d}\zeta_i(t) = \zeta_i^{(1)}(\alpha(t),\beta(t),\mathcal{X}(t))\mathrm{d}a_i(t) + \zeta_i^{(2)}(\alpha(t),\beta(t),\mathcal{X}(t)) \mathrm{d}z_i(t) \label{add_noise}
\end{equation}
where, for all $\smash{\alpha\in\mathbb{R}^{d_1}}$, $\smash{\beta\in\mathbb{R}^{d_2}}$, $\smash{\zeta_i^{(1)}(\alpha,\beta,\cdot):\mathbb{R}^{d_3}\rightarrow}$ $\smash{\mathbb{R}^{d_i}}$, $\smash{\zeta_i^{(2)}(\alpha,\beta,\cdot):\mathbb{R}^{d_3}\rightarrow \mathbb{R}^{d_i\times d_5^{i}}}$ are Borel measurable functions; $\{a_i(t)\}_{t\geq 0}$ are predictable, increasing processes, 
and $\smash{\{z_i(t)\}_{t\geq 0}}$ are $\smash{\mathbb{R}^{d_5^{i}}}$ valued Wiener processes defined on $(\Omega,\mathcal{F},\mathbb{P})$. 

\begin{manualassumption}{2.2.3a} \label{assumption5a}
For all $T>0$, the processes $\{\zeta_i(t)\}_{t\geq 0}$, $i=1,2$ satisfy
\begin{align}
\lim_{s\rightarrow\infty} \sup_{t\in[s,s+T]}\left|\left| \int_{s}^{t} \gamma_i(v)\mathrm{d}\zeta_i(v)\right|\right|&=0~,~~~\text{a.s.}
\end{align}
\end{manualassumption}

\begin{manualassumption}{2.2.3b} \label{assumption5b}
The functions $\zeta_i^{(2)}$, $i=1,2$, have the PGP componentwise. In particular, they grow no faster than polynomially with respect to the $x$ variable.
\end{manualassumption}

\begin{manualassumption}{2.2.3c} \label{assumption5c}
There exist constants $A_{z_1,z_2},A_{z_i,b}>0$, $i=1,2$, such that, component-wise,
\begin{equation}
c_{z_1,z_2}(t) = \frac{\mathrm{d}[z_1,z_2](t)}{\mathrm{d}t}\leq A_{z_1,z_2}~,~c_{z_i,b}(t) = \frac{\mathrm{d}[z_i,b](t)}{\mathrm{d}t}\leq A_{z_i,b}.
\end{equation}
where $[\cdot,\cdot]$ denotes the quadratic variation.
\end{manualassumption}

The first of these conditions is identical to the noise condition which appeared in the analysis of the general two-timescale stochastic gradient descent algorithm in Section \ref{subsec:main1}. Once again, this can be regarded as a continuous-time version of the Kushner-Clark condition. The other two assumptions are unique to the continuous-time, two-timescale stochastic gradient descent algorithm with Markovian dynamics introduced in this paper. We should note, however, that similar assumptions have previously appeared in the analysis of the single-timescale stochastic approximation schemes in, for example, \cite{Lazrieva2008,Valkeila2000}.

Our main result on the convergence of Algorithm (\ref{alg2_1}) - (\ref{alg2_2}) 
is contained in the following theorem.

\begin{theorem} \label{theorem1a}
 Assume that Assumptions \ref{assumption1a} - \ref{assumption5c} and \ref{assumption4} hold. In addition, assume that Assumptions \ref{assumption5} - \ref{assumption6} hold for the functions $f(\cdot)$ and $g(\cdot)$ defined in Assumption \ref{assumption4a}. Then, almost surely, 
\begin{equation}
\lim_{t\rightarrow\infty}\nabla_{\alpha}f(\alpha(t),\beta(t)) = \lim_{t\rightarrow\infty}\nabla_{\beta}g(\alpha(t),\beta(t)) = 0.
\end{equation}
\end{theorem}

\begin{proof}
See Appendix \ref{sec:proof2}.
\end{proof}

Our proof of Theorem \ref{theorem1a} is obtained by rewriting Algorithm (\ref{alg2_1}) - (\ref{alg2_2}) in the form of Algorithm (\ref{eq1}) - (\ref{eq2}), viz
\begin{subequations}
\begin{align}
\mathrm{d}\alpha(t) = -\gamma_{1}(t)\bigg[&\nabla_{\alpha} f(\alpha(t),\beta(t))\mathrm{d}t + \underbrace{\big(F(\alpha(t),\beta(t),\mathcal{X}(t))-\nabla_{\alpha} f(\alpha(t),\beta(t))\big)\mathrm{d}t + \mathrm{d}\zeta_1(t)}_{=\hspace{.5mm}\mathrm{d}\xi_1(t)}\bigg], \hspace{-6mm} \label{eq2213a} \\
\mathrm{d}\beta(t) = -\gamma_{2}(t)\bigg[&\nabla_{\beta} g(\alpha(t),\beta(t))\mathrm{d}t + \underbrace{\big(G(\alpha(t),\beta(t),\mathcal{X}(t))-\nabla_{\beta} g(\alpha(t),\beta(t))\big)\mathrm{d}t + \mathrm{d}\zeta_2(t)}_{=\hspace{.5mm}\mathrm{d}\xi_2(t)}\bigg],\hspace{-6mm}  \label{eq2213b}
\end{align}
\end{subequations}
and proving that the conditions of Theorem \ref{theorem1a} (Assumptions \ref{assumption1a} - \ref{assumption5c}) 
imply the conditions of Theorem \ref{theorem1} (Assumptions \ref{assumption1} - \ref{assumption3}). Clearly, if this is the case, then Theorem \ref{theorem1a} follows directly from Theorem \ref{theorem1}. This statement holds trivially for all conditions except those relating to the noise processes. It thus remains to establish that, under the noise conditions in Theorem \ref{theorem1a} (Assumptions \ref{assumption2a} - \ref{assumption4aii}, \ref{assumption5a} - \ref{assumption5c}), the noise condition in Theorem \ref{theorem1} (Assumption \ref{assumption3}) holds for the noise processes $\{\xi_i(t)\}_{t\geq 0}$, $i=1,2$, as defined by \eqref{eq2213a} - \eqref{eq2213b}. The central part of this proof is thus to control terms of the form,
\begin{align}
&\int_0^t \gamma_{1}(s)\left[F(\alpha(s),\beta(s),\mathcal{X}(s))-\nabla_{\alpha} f(\alpha(s),\beta(s))\right]\mathrm{d}s, \\
&\int_0^t \gamma_{1}(s)\left[G(\alpha(s),\beta(s),\mathcal{X}(s))-\nabla_{\beta} g(\alpha(s),\beta(s))\right]\mathrm{d}s.
\end{align}
This is achieved by rewriting each such term using the solution of an appropriate Poisson equation, and applying regularity results. This approach - namely, the use of the Poisson equation - is standard in the almost sure convergence analysis of stochastic approximation algorithms with Markovian dynamics, both in discrete time, including the single-timescale case (e.g. \cite{Benveniste1990,Kushner2003,Metivier1984}) 
and two-timescale case (e.g. \cite{Konda2003,Konda2003a,Tadic2004}), and in continuous time \cite{Bhudisaksang2020,Sirignano2017a,Sirignano2020a,Surace2019}. 

This part of our proof most closely resembles the proofs of \cite[Lemma 3.1]{Sirignano2017a} and \cite[Lemma 1]{Surace2019}, adapted to the current, somewhat more general setting. In general, however, our proof follows an entirely different approach to those in \cite{Sirignano2017a,Surace2019}. Indeed, the ODE method is central to our proof, while the proofs in these papers are based on more classical stochastic descent arguments. In particular, they represent a continuous-time, Markovian extension of the method introduced in \cite{Bertsekas2000}, under the additional assumption that the objective function is bounded from below. This method, broadly speaking, demonstrates that whenever the magnitude of the gradient of the objective function is large, it remains so for a sufficiently long time interval, guaranteeing a decrease in the value of the objective function which is significant and dominates the noise effects. Under the additional assumption that the objective function is bounded from below, it must converge almost surely to some finite value, and its gradient must converge to zero \cite{Bertsekas2000}. 
Crucially, these arguments do not rely on the assumption that the algorithm iterates remain bounded, which represents a significant advantage over the ODE method. It is thus of clear interest to extend this approach to the two-timescale setting. Thus far, however, our attempts to do so have been unsuccessful, due to the presence of the secondary process.\footnote{In the single-timescale case, one proves that when $\nabla f(\cdot)$ is `large', the objective function $f(\cdot)$ decreases by at least $\delta>0$, and that when $\nabla f(\cdot)$ is `small', the objective function $f(\cdot)$ increases by no more than some smaller positive constant amount $0<\delta_1<\delta$. In the two-timescale case, the second of these steps is no longer possible.} As such, this remains an interesting direction for future study. 

We conclude this section with the remark that Theorem \ref{theorem1a}, and its proof, still hold upon restriction to a single-timescale (i.e., under the assumption that either $\alpha(t)$ or $\beta(t)$ is held fixed). In this case, of course, we only require assumptions which pertain to that timescale. In this context, our theorem includes, as a particular case, the convergence result in \cite{Sirignano2017a}. 
Moreover, our proof provides an entirely different proof of that result.





\section{Online Parameter Estimation and Optimal Sensor Placement}
\label{sec:RML_ROSP}

To illustrate the results of the previous section, we now consider the problem of joint online parameter estimation and optimal sensor placement in a partially observed diffusion model. Throughout this section, we will consider the family of partially observed diffusion processes described by equations (\ref{signal_finite_dim_inf}) - (\ref{obs_finite_dim_inf}).

\subsection{Online Parameter Estimation} 
\label{sec:param_est}
We first review the problem of online parameter estimation. We will suppose that the model generates the observation process $\{y(t)\}_{t\geq 0}$ according to a true, but unknown, static parameter $\theta^{*}$. The objective is then to obtain an estimator $\{{\theta}(t)\}_{t\geq 0}$ of $\theta^{*}$ which is both $\mathcal{F}_t^Y$-measurable and recursively computable. That is, an estimator which can be computed online using the continuous stream of observations, without revisiting the past. In this subsection, we will assume that the sensor locations $\boldsymbol{o}\in\Omega^{n_y}$ are fixed. 

One such estimator can be obtained as a modification of the classical offline maximum likelihood estimator (e.g., \cite{Surace2019}). 
We thus recall the expression for the log-likelihood of the observations, or incomplete data log-likelihood, for a partially observed diffusion process (e.g., \cite{Balakrishnan1973}), 
namely
\begin{align}
\mathcal{L}_t(\theta,\boldsymbol{o})= \int_0^t R^{-1}(\boldsymbol{o}) \hat{C}(\theta,\boldsymbol{o},s)\cdot\mathrm{d}y(s)-\frac{1}{2}\int_0^t ||R^{-\frac{1}{2}}(\boldsymbol{o})\hat{C}(\theta,\boldsymbol{o},s)||^2\mathrm{d}s, \label{ll_func}
\end{align}
where $\hat{C}(\theta,\boldsymbol{o},s)$ denotes the conditional expectation of $C(\theta,\boldsymbol{o},x(s))$, given the observation sigma-algebra $\mathcal{F}_s^Y$,  viz
\begin{equation}
\hat{C}(\theta,\boldsymbol{o},s) = \mathbb{E}_{\theta,\boldsymbol{o}}\left[C(\theta,\boldsymbol{o},x(s))|\mathcal{F}_s^Y\right]. \label{Chat}
\end{equation}

In the online setting, a standard approach to parameter estimation is to recursively seek the value of $\theta$ which maximises the asymptotic log-likelihood, viz
\begin{align}
\tilde{\mathcal{L}}(\theta,\boldsymbol{o}) = \lim_{t\rightarrow\infty}\frac{1}{t}\mathcal{L}_t(\theta,\boldsymbol{o}) \label{asymptotic_ll} 
\end{align}
Typically, neither the asymptotic log-likelihood, nor its gradient, are available in analytic form. It is, however, possible to compute noisy estimates of these quantities at any finite time, using the integrand of the log-likelihood and the integrand of its gradient, respectively. This optimisation problem can thus be tackled using continuous-time stochastic gradient ascent, whereby the parameters follow a noisy ascent direction given by the integrand of the gradient of the log-likelihood, evaluated with the current parameter estimate. In particular, initialised at ${\theta}_0\in\Theta$, the parameter estimates $\{{\theta}(t)\}_{t\geq 0}$ are generated according to the SDE \cite{Surace2019} 
\begin{equation}
\mathrm{d}{\theta}(t) = \left\{ \begin{array}{lll} \gamma(t) \big[\hat{C}^{\theta}(\theta(t),\boldsymbol{o},t)\big]^T R^{-1}(\boldsymbol{o})\big[\mathrm{d}y(t)-\hat{C}(\theta(t),\boldsymbol{o},t)\mathrm{d}t\big] & , & {\theta}(t)\in\Theta, \\[1.5mm] 0 & , & {\theta}(t)\not\in\Theta,  \end{array}\right. \label{eq_sde}
\end{equation}
where $\hat{C}^{\theta}(\theta,\boldsymbol{o},t)=\nabla_{\theta}\hat{C}(\theta,\boldsymbol{o},t)$ is used to denote the gradient of $\hat{C}(\theta,\boldsymbol{o},t)$ with respect to the parameter vector.\footnote{We use the convention that the gradient operator adds a covariant dimension to the tensor field upon which it acts. Thus, for example, since $\smash{\hat{C}(\theta,\boldsymbol{o},t)=\mathbb{E}_{\theta,\boldsymbol{o}}[C(\theta,\boldsymbol{o},x(t))|\mathcal{F}_t^Y]}$ takes values in $\mathbb{R}^{n_y}$, its gradient $\smash{\hat{C}^{\theta}(\theta,\boldsymbol{o},t)=\nabla_{\theta}\hat{C}(\theta,\boldsymbol{o},t)}$, takes values in $\mathbb{R}^{n_y\times n_{\theta}}$.} 
Following \cite{Surace2019}, this algorithm includes a projection device which ensures that the parameter estimates $\{\theta(t)\}_{t\geq 0}$ remain in $\
\Theta\subset\mathbb{R}^{n_{\theta}}$ with probability one. This is common for algorithms of this type (e.g., \cite{Ljung1999}). 
In the literature on statistical inference and system identification, this algorithm is commonly referred to as recursive maximum likelihood (RML).

The asymptotic properties of this method for partially observed, discrete-time systems (e.g., \cite{LeGland1995,LeGland1997
,Tadic2010,Tadic2021}), 
and for fully-observed, continuous-time systems (e.g., \cite{Bishwal2008,Kutoyants1984,Levanony1994}), 
have been studied extensively. In comparison, the partially observed, continuous-time case has received relatively little attention. 
The use of a continuous-time RML method for online parameter estimation in a partially-observed linear diffusion process was first proposed in \cite{Gerencser1984}, and later extended in \cite{Gerencser2009}.\footnote{We remark that the SDEs for the estimators considered in \cite{Gerencser1984,Gerencser2009} include an additional second order term, which arises when the It\^o-Venzel formula is applied to the score function.} This approach has more recently been revisited in \cite{Surace2019}. In this paper, the authors derived a RML estimator for the parameters of a general, non-linear partially observed diffusion process, and established the almost sure convergence of this estimator to the stationary points of the asymptotic log-likelihood under appropriate conditions on the process consisting of the latent state, the filter, and the filter derivative. 
This paper extended the results in \cite{Sirignano2017a} to the partially-observed setting, and is the estimator which we consider in the current paper.

\subsection{Optimal Sensor Placement} \label{sec:sensor} We now turn our attention to the problem of optimal sensor placement. We will suppose that the observation process $\{y(t)\}_{t\geq 0}$ is generated using a finite set of $n_y$ sensors. Our objective is to obtain an estimator of the set of $n_y$ sensor locations $\hat{\boldsymbol{o}} = \{\boldsymbol{o}_i\}_{i=1}^{n_y}$ which are optimal with respect to some pre-determined criteria, possibly subject to constraints. Once more, we require our estimator to be $\mathcal{F}_t^Y$-measurable and recursively computable. In this subsection, we will assume that the parameter $\theta\in\Theta$ is fixed. 

A standard approach to this problem is to define a suitable objective function, say $\mathcal{J}_t(\theta,\cdot):\Omega^{n_y}\rightarrow\mathbb{R}$, and then to define the optimal estimator as 
\begin{equation}
\hat{\boldsymbol{o}}(t) = \argmin_{\boldsymbol{o}\in\Omega^{n_y}} \mathcal{J}_t(\theta,\boldsymbol{o}).
\end{equation}
We focus on the objective of optimal state estimation. In this case, following \cite{Burns2015}, 
we will consider an objective function of the form
\begin{equation}
\mathcal{J}_t(\theta,\boldsymbol{o}) = \int_0^{t} \underbrace{\mathrm{Tr}\left[H(s)\hat{\Sigma}(\theta,\boldsymbol{o},s)\right]}_{\hat{j}(\theta,\boldsymbol{o},s)}\mathrm{d}s:=\int_0^{t} \hat{j}(\theta,\boldsymbol{o},s)\mathrm{d}s, \label{jhat}
\end{equation}
where $H(s):\mathbb{R}^{d_x\times d_x}\rightarrow\mathbb{R}^{d_x\times d_x}$ is a  matrix  which allows one to weight significant parts of the state estimate, and $\hat{\Sigma}(\theta,\boldsymbol{o},s)$ denotes the conditional covariance of the latent state $x(s)$, given the history of observations $\mathcal{F}_s^Y$, viz
\begin{align}
\hat{\Sigma}(\theta,\boldsymbol{o},s)&=\mathrm{Cov}_{\theta,\boldsymbol{o}}\left(x(s)|\mathcal{F}_s^Y\right).\label{covariance_cond} 
\end{align}
Broadly speaking, the use of this objective corresponds to seeking the sensor placement which minimises the uncertainty in the estimate of the latent state. Other choices for the objective function are, of course, possible. These include, among many others, the trace of the conditional covariance at some finite, terminal time,
and the trace of the steady-state conditional covariance 
In the online setting, the objective is to recursively estimate the optimal sensor locations $\hat{\boldsymbol{o}}$ in real time using the continuous stream of observations. In this case, one approach is to recursively seek the value of $\boldsymbol{o}$ which minimises an asymptotic version of the objective function (e.g., \cite{Zhang2018}), namely
\begin{equation}
\tilde{\mathcal{J}}(\theta,\boldsymbol{o})= \lim_{t\rightarrow\infty} \frac{1}{t}{\mathcal{J}}_t(\theta,\boldsymbol{o})= \lim_{t\rightarrow\infty}\left[\frac{1}{t}\int_0^{t} \hat{j}(\theta,\boldsymbol{o},s)\mathrm{d}s\right]. \label{asymptotic_obj}
\end{equation}
As in the previous section, typically neither the asymptotic objective function, nor its gradient, are available in analytic form.\footnote{A notable exception to this is the linear Gaussian case, in which case the asymptotic objective function is the solution of the algebraic Ricatti equation, which is independent of the observation process, and can thus be computed prior to receiving any observations (e.g., \cite{Kalman1961}). This independence no longer holds, however, when online parameter estimation and optimal sensor placement are coupled (see Section \ref{sec_joint_RML_OSP}). In this case, the (asymptotic) objective function depends on the parameter estimates via equation (\ref{asymptotic_obj}), and the parameter estimates depend on the observations via equation (\ref{eq_sde}). Thus, implicitly, the sensor placements estimates do now depend on the observations.}
It is, however, possible to compute noisy estimates of these quantities at any finite time, using the integrand of the objective function and its gradient, respectively. Similar to online parameter estimation, this optimisation problem can thus also be tackled using continuous-time stochastic gradient descent, whereby the sensor locations follow a noisy descent direction given by the integrand of the gradient of the objective function, evaluated with the current estimates of the sensor placements.
In particular, initialised at ${\boldsymbol{o}}_0\in\Omega^{n_y}$, the sensor locations $\{\boldsymbol{o}(t)\}_{t\geq 0}$ are generated according to 
\begin{equation}
\mathrm{d}{\boldsymbol{o}}(t)= \left\{ \begin{array}{lll}  -\gamma(t)\big[\hat{j}^{\boldsymbol{o}}(\theta,\boldsymbol{o}(t),t)\big]^T\mathrm{d}t & , & {\boldsymbol{o}}(t)\in\Omega^{n_y} \\[1.5mm] 0 & , & {\boldsymbol{o}}(t)\not\in\Omega^{n_y}  \end{array}\right. \label{sensor_placement}
\end{equation}
where $\smash{\hat{j}^{\boldsymbol{o}}(\theta,\boldsymbol{o},t) = \nabla_{\boldsymbol{o}}\hat{j}(\theta,\boldsymbol{o},t) = \nabla_{\boldsymbol{o}}\mathrm{Tr}[H(s)\hat{\Sigma}(\theta,\boldsymbol{o},s)]}$ is used
to denote the gradient of $\hat{j}(\theta,\boldsymbol{o},t)$ with respect to the sensor locations. Similar to the online parameter estimation algorithm, this recursion includes a projection device to ensure that the sensor placements $\{\boldsymbol{o}(t)\}_{t\geq 0}$ remain in $\Omega^{n_y}\subset\mathbb{R}^{n_yn_{\boldsymbol{o}}}$ with probability one.

\subsection{The Filter and Its Gradients} \label{sec:filter}
 In order to implement either of these algorithms, it is necessary to compute the  conditional expectations  $\hat{C}(\theta,\boldsymbol{o},t)$ and $\hat{j}(\theta,\boldsymbol{o},t)$, as well as their gradients, $\smash{\hat{C}^{\theta}(\theta,\boldsymbol{o},t)}$ and $\smash{\hat{j}^{\boldsymbol{o}}(\theta,\boldsymbol{o},t)}$.
 In principle, this requires one to obtain solutions of the Kushner-Stratonovich equation for arbitrary integrable $\varphi:\mathbb{R}^{n_x}\rightarrow\mathbb{R}$, viz (e.g., \cite{Bain2009,Kushner1964})
\begin{equation}
\mathrm{d}\hat{\varphi}(t)= (\hat{\mathcal{A}_{x}{\varphi}})(t) + \big((\hat{C\varphi})(t) - \hat{C}(t)\hat{\varphi}(t)\big)\cdot \big(\mathrm{d}y(t) - \hat{C}(t)\mathrm{d}t\big), \label{eq_KS}
\end{equation}
where $\hat{\varphi}(t)=\mathbb{E}[\varphi(x(t))|\mathcal{F}_t^Y]$ denotes the conditional expectation of $\varphi(t)$ given the history of observations $\mathcal{F}_t^{Y}$,
and $\mathcal{A}_x$ denotes the infinitesimal generator of the latent signal process.
In general, exact solutions to the Kushner-Stratonovich equation are very rarely available \cite{Maurel1984}.
 In order to make any progress, we must therefore introduce the following additional assumption. 

\begin{assumption} \label{filter_assumption1}
The Kushner-Stratonovich equation admits a finite dimensional recursive solution, or a finite-dimensional recursive approximation.
\end{assumption}

There are a small but important class of filters for which finite-dimensional recursive solutions do exist, namely, the Kalman-Bucy filter 
\cite{Kalman1961}, 
the Bene\v{s} filter \cite{Benes1981}, 
and extensions thereof 
(e.g., \cite{Daum1986,Ocone1982}). In addition, there are a much larger class of processes for which finite-dimensional recursive approximations are available,  and thus, crucially, for which the proposed algorithm can still be applied.   Standard approximation schemes include, among others, the extended Kalman-Bucy filter \cite{DelMoral2018}, the unscented Kalman-Bucy filter \cite{Sarkka2007}, projection filters \cite{Brigo1998}, assumed-density filters \cite{Brigo1999}, 
the ensemble Kalman-Bucy filter (EnKBF) \cite{Moral2018}, and other particle filters (e.g., \cite{DelMoral2000}, \cite[Chapter 9]{Bain2009}, and references therein). 

This assumption implies, in particular, that there exists a finite-dimensional, $\mathcal{F}_t^Y$-adapted process $M(\theta,\boldsymbol{o}) = \{M(\theta,\boldsymbol{o},t)\}_{t\geq 0}$, taking values in $\mathbb{R}^p$, and functions  $\smash{\psi_{C}(\theta,\boldsymbol{o},\cdot):\mathbb{R}^p\rightarrow\mathbb{R}^{n_y}}$, $\smash{\psi_{j}(\theta,\boldsymbol{o},\cdot):\mathbb{R}^p\rightarrow\mathbb{R}^{n_x}}$ such that, in the case of an exact solution, 
\begin{subequations}
\begin{align}
\hat{C}(\theta,\boldsymbol{o},t) &= \psi_{C}(\theta,\boldsymbol{o},M(\theta,\boldsymbol{o},t)), \label{phi_C_def} \\
\hat{j}(\theta,\boldsymbol{o},t) &= \psi_{j}(\theta,\boldsymbol{o},M(\theta,\boldsymbol{o},t)), \label{phi_j_def}
\end{align}
\end{subequations}
or, in the case of an approximate solution, such that these equations hold only approximately. The process $M(\theta,\boldsymbol{o})$ is typically referred to as the finite-dimensional (approximate) filter representation, or more simply, the filter. We provide an illustrative example of one such finite-dimensional filter representation after stating our remaining assumption.

   We are also required to compute the gradients $\hat{C}^{\theta}(\theta,\boldsymbol{o},t)$ and $\hat{j}^{\boldsymbol{o}}(\theta,\boldsymbol{o},t)$ in order to implement our algorithm. We must therefore also introduce the following additional assumption.

\begin{assumption} \label{filter_assumption2}
The finite-dimensional filter representation 
is continuously differentiable with respect to $\theta$ and $\boldsymbol{o}$.
\end{assumption}

Following this assumption, it is possible to define $M^{\theta}(\theta,\boldsymbol{o}) = \{M^{\theta}(\theta,\boldsymbol{o},t)\}_{t\geq 0}$ $ =\{\nabla_{\theta}M(\theta,\boldsymbol{o},t)\}_{t\geq 0}$ and $M^{\boldsymbol{o}}(\theta,\boldsymbol{o})= \{M^{\boldsymbol{o}}(\theta,\boldsymbol{o},t)\}_{t\geq 0} = \{\nabla_{\boldsymbol{o}}M(\theta,\boldsymbol{o},t)\}_{t\geq 0}$ as the $\mathbb{R}^{p\times n_{\theta}}$ and $\mathbb{R}^{p\times n_yn_{\boldsymbol{o}}}$ valued processes consisting of the gradients of the finite dimensional filter representation with respect to $\theta$ and $\boldsymbol{o}$, respectively. We will refer to these processes as the (finite-dimensional) tangent filters. 

It follows, upon formal differentiation of equations \eqref{phi_C_def} and \eqref{phi_j_def}, that, either exactly or approximately, we have

\begin{subequations}
\begin{align}
\hat{C}^{\theta}(\theta,\boldsymbol{o},t) &= \psi_{C}^{\theta}(\theta,\boldsymbol{o},M(\theta,\boldsymbol{o},t),M^{\theta}(\theta,\boldsymbol{o},t)) \label{hatC_theta_0} \\
&= \nabla_{\theta}\psi_{C}(\theta,\boldsymbol{o},M(\theta,\boldsymbol{o},t)) + \nabla_{M}\psi_{C}(\theta,\boldsymbol{o},M(\theta,\boldsymbol{o},t))M^{\theta}(\theta,\boldsymbol{o},t), \label{hatC_theta}
\end{align}
\end{subequations}
and
\begin{subequations}
\begin{align}
\hat{j}^{\boldsymbol{o}}(\theta,\boldsymbol{o},t) &= \psi_{j}^{\boldsymbol{o}}(\theta,\boldsymbol{o},M(\theta,\boldsymbol{o},t),M^{\boldsymbol{o}}(\theta,\boldsymbol{o},t)) \label{hatj_o_0} \\
&= \nabla_{\boldsymbol{o}}\psi_{j}(\theta,\boldsymbol{o},M(\theta,\boldsymbol{o},t)) + \nabla_{M}\psi_{j}(\theta,\boldsymbol{o},M(\theta,\boldsymbol{o},t))M^{\boldsymbol{o}}(\theta,\boldsymbol{o},t). \label{hatj_o}
\end{align}
\end{subequations}
 
 We are now ready to introduce our final assumption on the filter. This assumption will allow us to rewrite the joint online parameter estimation and optimal sensor placement algorithm in the form of Algorithm \eqref{alg2_1} - \eqref{alg2_2}, and thus to apply Theorem \ref{theorem1a}.
 \begin{assumption} \label{filter_assumption3}
 The finite-dimensional filter representation satisfies a stochastic differential equation of the form
\begin{align}
\mathrm{d}M(\theta,\boldsymbol{o},t) &= S(\theta,\boldsymbol{o},M(\theta,\boldsymbol{o},t))\mathrm{d}t+  {T}(\theta,\boldsymbol{o},M(\theta,\boldsymbol{o},t))\mathrm{d}y(t)  \label{eq_filter_sde} \\
&~+  {U}(\theta,\boldsymbol{o},M(\theta,\boldsymbol{o},t))\mathrm{d}a(t), \nonumber
\end{align} 
where $a=\{a(t)\}_{t\geq 0}$ is a $\mathbb{R}^{q}$ valued Wiener process independent of $\smash{\mathcal{F}_t^{X,Y}}$, and  the functions $S$, ${T}$, and ${U}$ map $\mathbb{R}^{n_{\theta}}\times\mathbb{R}^{n_yn_{\boldsymbol{o}}}\times\mathbb{R}^{p}$ to $\mathbb{R}^p$, $\mathbb{R}^{p\times n_yn_{\boldsymbol{o}}}$, and $\mathbb{R}^{p\times q}$, respectively.
 \end{assumption}

This assumption can be shown to hold for a broad class of filters. In particular, the inclusion of the independent noise process means that this SDE holds for a large class of approximate filters, including many of those mentioned after Assumption \ref{filter_assumption1}. It follows from this assumption, upon differentiation of \eqref{eq_filter_sde}, that the finite-dimensional tangent filters 
satisfy the SDEs 
\begin{subequations}
\begin{align}
\mathrm{d} M^{\theta}(\theta,\boldsymbol{o},t) &= S'_{\theta}(\theta,\boldsymbol{o},M(\theta,\boldsymbol{o},t),M^{\theta}(\theta,\boldsymbol{o},t))\mathrm{d}t \label{eq_filter_grad_param_sde} \\
&~+{T}'_{\theta}(\theta,\boldsymbol{o},M(\theta,\boldsymbol{o},t),M^{\theta}(\theta,\boldsymbol{o},t))\mathrm{d}y(t) \nonumber\\
&~+{U}'_{\theta}(\theta,\boldsymbol{o},M(\theta,\boldsymbol{o},t),M^{\theta}(\theta,\boldsymbol{o},t))\mathrm{d}a(t), \nonumber \\[2mm]
\mathrm{d} M^{\boldsymbol{o}}(\theta,\boldsymbol{o},t) &= S'_{\boldsymbol{o}}(\theta,\boldsymbol{o},M(\theta,\boldsymbol{o},t),M^{\boldsymbol{o}}(\theta,\boldsymbol{o},t))\mathrm{d}t \label{eq_filter_grad_sensor_sde}\\
&~+{T}'_{\boldsymbol{o}}(\theta,\boldsymbol{o},M(\theta,\boldsymbol{o},t),M^{\boldsymbol{o}}(\theta,\boldsymbol{o},t))\mathrm{d}y(t) \nonumber \\
&~+{U}'_{\boldsymbol{o}}(\theta,\boldsymbol{o},M(\theta,\boldsymbol{o},t),M^{\boldsymbol{o}}(\theta,\boldsymbol{o},t))\mathrm{d}a(t). \nonumber 
\end{align}
\end{subequations}
where, for example, the tensor field $S'_{\theta}$ is obtained explicitly according to
\begin{align}
S'_{\theta}(\theta,\boldsymbol{o},M(\theta,\boldsymbol{o},t),M^{\theta}(\theta,\boldsymbol{o},t)) &= \nabla_{\theta}S(\theta,\boldsymbol{o},M(\theta,\boldsymbol{o},t)) \\
&~+ \nabla_{M}S(\theta,\boldsymbol{o},M(\theta,\boldsymbol{o},t))M^{\theta}(\theta,\boldsymbol{o},t). \nonumber
\end{align}
with analogous expressions for the tensor fields $T'_{\theta}$, $U'_{\theta}$, $S'_{\boldsymbol{o}}$, $T'_{\boldsymbol{o}}$ and $U'_{\boldsymbol{o}}$ . 

We can now summarise the evolution equations for the latent signal, the finite-dimensional filter, and the finite-dimensional tangent filters, into a single SDE. In particular, let us define  $\smash{\mathcal{X}(\theta,\boldsymbol{o})=\{\mathcal{X}(\theta,\boldsymbol{o},t)\}_{t\geq 0}}$ as the $\mathbb{R}^{N}$ valued diffusion process consisting of the concatenation of the latent signal, the (vectorised) finite-dimensional filter, and the (vectorised) finite-dimensional tangent filters, with $N = n_x+ p + pn_{\theta} + p n_yn_{\boldsymbol{o}}$. That is, in a slight abuse of notation, 

\begin{equation}
\smash{\mathcal{X}(\theta,\boldsymbol{o},t) = \big(x(t),\mathrm{vec}(M(\theta,\boldsymbol{o},t),  \mathrm{vec}(M^{\theta}(\theta,\boldsymbol{o},t)),\mathrm{vec}(M^{\boldsymbol{o}}(\theta,\boldsymbol{o},t))\big)^T.}
\end{equation}


It then follows straightforwardly, stacking the equation for the signal process \eqref{signal_finite_dim_inf}, the filter \eqref{eq_filter_sde}, and tangent filters \eqref{eq_filter_grad_param_sde} - \eqref{eq_filter_grad_sensor_sde}, and substituting the equation for the observation process \eqref{obs_finite_dim_inf}, that
 
\begin{equation}
\mathrm{d}\mathcal{X}(\theta,\boldsymbol{o},t) = \Phi(\theta,\boldsymbol{o},\mathcal{X}(\theta,\boldsymbol{o},t))\mathrm{d}t + \Psi(\theta,\boldsymbol{o},\mathcal{X}(\theta,\boldsymbol{o},t))\mathrm{d}b(t), \label{compact_SDE}
\end{equation}
where the functions $\Phi$ and $\Psi$ take values in $\mathbb{R}^{N}$ and $\mathbb{R}^{N\times(n_x+n_y+q)}$, respectively, and 
where $\smash{b = \{b(t)\}_{t\geq 0}}$ is the $\smash{\mathbb{R}^{n_x+n_y+q}}$ valued Wiener process obtained by concatenating the signal noise process $\smash{v=\{v(t)\}_{t\geq 0}}$, the observation noise process $\smash{w=\{v(t)\}_{t\geq 0}}$, and the independent noise process arising in the equations for the finite-dimensional filter representation $\smash{a=\{a(t)\}_{t\geq 0}}$.

\begin{remark_}
To help to illustrate the notation introduced in this section, let us consider a simple one-dimensional linear Gaussian model with a single unknown parameter, and a single sensor location, viz
\begin{alignat}{3}
\mathrm{d}x(t) &= -\theta x(t)\mathrm{d}t + \mathrm{d}v(t)~&&,~~~&&x(0)=x_0, \label{linear_gaussian1} \\
 \mathrm{d}y(t) &= x(t)\mathrm{d}t + \mathrm{d}w(t)~&&,~~~&&y(0)=0, \label{linear_gaussian2} 
\end{alignat}
where $v=\{w(t)\}_{t\geq 0}$ and $w=\{v(t)\}_{t\geq 0}$ are one-dimensional Brownian motions with incremental variances $Q(\theta)= 1$ and $R({o}) = ({o}-{o}_0)^2$, and $\smash{x_0\sim \mathcal{N}(0,\frac{1}{2\theta})}$. Clearly, this is an example of a partially observed diffusion process of the form \eqref{signal_finite_dim_inf} - \eqref{obs_finite_dim_inf}, with $\theta \in\mathbb{R}$, $o \in\mathbb{R}$, and operators $A(\theta,x) = -\theta x$, $B(\theta,x) =1$, and $C(\theta,\boldsymbol{o},x) = x$. We can also identify, using \eqref{Chat} and \eqref{jhat}, the conditional expectations 
\begin{subequations}
\begin{align}
\hat{C}(\theta,\boldsymbol{o},t) &= \mathbb{E}_{\theta,\boldsymbol{o}}[C(\theta,\boldsymbol{o},x(t))|\mathcal{F}_t^{Y}] = \mathbb{E}_{\theta,\boldsymbol{o}}[x(t)|\mathcal{F}_t^Y] \label{Chat_gauss0} \\
\hat{j}(\theta,\boldsymbol{o},t) &= \mathrm{Tr}[\mathrm{Var}_{\theta,\boldsymbol{o}}[x(t)|\mathcal{F}_t^Y]]= \mathrm{Var}_{\theta,\boldsymbol{o}}\left[x(t)|\mathcal{F}_t^Y\right]. \label{jhat_gauss0}
\end{align}
\end{subequations}
Let us consider each of the assumptions in turn introduced in this section in turn, starting with Assumption \ref{filter_assumption1}. For the linear Gaussian model, it is well known that the optimal filter has a Gaussian distribution with mean $\smash{\hat{x}(\theta,o,t) = \mathbb{E}_{\theta,o}[x(t)|\mathcal{F}_t^{Y}]}$ and variance $\smash{\hat{\Sigma}(\theta,o,t) =\mathrm{Var}_{\theta,o}[x(t)|\mathcal{F}_t^Y]}$, both of which can be computed recursively (the precise form of these equations is presented below in \eqref{linear_gaussian_filter}). 
This is known as the Kalman-Bucy filter \cite{Kalman1961}.  We thus have a $p=2$ dimensional filter representation $\smash{M(\theta,\boldsymbol{o},t)}$ $\smash{= (\hat{x}(\theta,\boldsymbol{o},t),\hat{\Sigma}(\theta,\boldsymbol{o},t))^T}$. 
It follows straightforwardly that, in the case,
\begin{subequations}
\begin{align}
 \psi_{C}(\theta,\boldsymbol{o},M(\theta,\boldsymbol{o},t))&:=\hat{C}(\theta,\boldsymbol{o},t)= \hat{x}(\theta,\boldsymbol{o},t) \label{hatC} \\
\psi_{j}(\theta,\boldsymbol{o},M(\theta,\boldsymbol{o},t))&:= \hat{j}(\theta,\boldsymbol{o},t) = \hat{\Sigma}(\theta,\boldsymbol{o},t). \label{hatj}
\end{align}
\end{subequations}
We next consider Assumption \ref{filter_assumption2}. In the current example, it is clear that the two-dimensional filter $M(\theta,\boldsymbol{o},t)$ is continuously differentiable with respect to both $\theta$ and $o$. Indeed, this follows directly from the differentiability of $A(\theta,x)$, $B(\theta,x)$, $C(\theta,\boldsymbol{o},x)$, $Q(\theta)$ and $R(\boldsymbol{o})$ with respect to these variables. We can thus define the tangent filters $\smash{M^{\theta}(\theta,\boldsymbol{o},t) = (\hat{x}^{\theta}(\theta,\boldsymbol{o},t), \hat{\Sigma}^{\theta}(\theta,\boldsymbol{o},t))^T}$ and $\smash{M^{\boldsymbol{o}}(\theta,\boldsymbol{o},t)= (\hat{x}^{\boldsymbol{o}}(\theta,\boldsymbol{o},t), \hat{\Sigma}^{\boldsymbol{o}}(\theta,\boldsymbol{o},t))^T}$, 
and compute
\begin{subequations}
\begin{align}
 \psi_{C}^{\theta}(\theta,\boldsymbol{o},M(\theta,\boldsymbol{o},t),M^{\theta}(\theta,\boldsymbol{o},t)) &:=\hat{C}^{\theta}(\theta,\boldsymbol{o},t)=
\hat{x}^{\theta}(\theta,\boldsymbol{o},t)
\label{Chat_gauss} \\
\psi_{j}^{\boldsymbol{o}}(\theta,\boldsymbol{o},M(\theta,\boldsymbol{o},t),M^{\boldsymbol{o}}(\theta,\boldsymbol{o},t))&:=\hat{j}^{\boldsymbol{o}}(\theta,\boldsymbol{o},t)= \hat{\Sigma}^{\boldsymbol{o}}(\theta,\boldsymbol{o},t) \label{jhat_gauss} 
\end{align}
\end{subequations}
Finally, we consider Assumption \ref{filter_assumption3}. The Kalman-Bucy filter evolves according to the following SDE
 \begin{align}
 \underbrace{\begin{pmatrix} \mathrm{d}\hat{x}(\theta,\boldsymbol{o},t) \\ \mathrm{d}\hat{\Sigma}(\theta,\boldsymbol{o},t) \end{pmatrix}}_{\mathrm{d}M(\theta,\boldsymbol{o},t)} &= \underbrace{\begin{pmatrix} -\theta \hat{x}(\theta,\boldsymbol{o},t) - 
 (\boldsymbol{o}-\boldsymbol{o}_0)^{-2} \hat{x}(\theta,\boldsymbol{o},t)\hat{\Sigma}(\theta,\boldsymbol{o},t) \label{linear_gaussian_filter} \\[1.5mm]
 1 - 2\theta \hat{\Sigma}(\theta,\boldsymbol{o},t) - 
 (\boldsymbol{o}-\boldsymbol{o}_0)^{-2}\hat{\Sigma}^2(\theta,\boldsymbol{o},t) \end{pmatrix}}_{S(\theta,\boldsymbol{o},M(\theta,\boldsymbol{o},t))} \mathrm{d}t \\[1mm]
 &\hspace{32mm}+ \underbrace{\begin{pmatrix} 
 (\boldsymbol{o}-\boldsymbol{o}_0)^{-2} \hat{\Sigma}(\theta,\boldsymbol{o},t) \\[1.5mm] 0 \end{pmatrix}}_{T(\theta,\boldsymbol{o},M(\theta,\boldsymbol{o},t))} \mathrm{d}y(t), \nonumber 
 \end{align}
Thus, the filter does indeed evolve according to an SDE of the form \eqref{eq_filter_sde}, with the final term identically equal to zero. Taking formal derivatives of this SDE, we can obtain the SDEs for the tangent filters, namely
\allowdisplaybreaks
\begin{align} 
\underbrace{\begin{pmatrix} \\ \mathrm{d}\hat{x}^{\theta}(\theta,\boldsymbol{o},t) \\  \\[4mm] \mathrm{d}\hat{\Sigma}^{\theta}(\theta,\boldsymbol{o},t) \\[-2mm] ~ \end{pmatrix}}_{\mathrm{d}M^{\theta}(\theta,\boldsymbol{o},t)} &= \underbrace{
\begin{pmatrix} 
-\hat{x}(\theta,\boldsymbol{o},t) - \theta\hat{x}^{\theta}(\theta,\boldsymbol{o},t)  \\
- 
(\boldsymbol{o}-\boldsymbol{o}_0)^{-2}\hat{x}^{\theta}(\theta,\boldsymbol{o},t)\hat{\Sigma}(\theta,\boldsymbol{o},t) \\
- 
(\boldsymbol{o}-\boldsymbol{o}_0)^{-2}\hat{x}(\theta,\boldsymbol{o},t)\hat{\Sigma}^{\theta}(\theta,\boldsymbol{o},t) \\[2mm]
- 2 \hat{\Sigma}(\theta,\boldsymbol{o},t) - 2 \theta \hat{\Sigma}^{\theta}(\theta,\boldsymbol{o},t) \\
-2(\boldsymbol{o}-\boldsymbol{o}_0)^{-2}\hat{\Sigma}(\theta,\boldsymbol{o},t)\hat{\Sigma}^{\theta}(\theta,\boldsymbol{o},t) 
 \end{pmatrix}}_{S'_{\theta}(\theta,\boldsymbol{o},M(\theta,\boldsymbol{o},t))} \mathrm{d}t \label{linear_gaussian_filter_tangent1}
\\[1mm]
&\hspace{16mm}+ \underbrace{
\begin{pmatrix} 
\vphantom{-\theta\hat{x}^{\boldsymbol{o}}(\theta,\boldsymbol{o},t) + 2c^2(\boldsymbol{o}-\boldsymbol{o}_0)^{-3} \hat{x}(\theta,\boldsymbol{o},t)\hat{\Sigma}(\theta,\boldsymbol{o},t)} \\ 
(\boldsymbol{o}-\boldsymbol{o}_0)^{-2}\hat{\Sigma}^{\theta}(\theta,\boldsymbol{o},t) \\ 
\vphantom{- c^2(\boldsymbol{o}-\boldsymbol{o}_0)^{-2}\hat{x}(\theta,\boldsymbol{o},t)\hat{\Sigma}^{\theta}(\theta,\boldsymbol{o},t)} \\[2mm] 
0 \\
\vphantom{-2(\boldsymbol{o}-\boldsymbol{o}_0)^{-2}\hat{\Sigma}(\theta,\boldsymbol{o},t)\hat{\Sigma}^{\theta}(\theta,\boldsymbol{o},t) } 
\end{pmatrix}}_{T'_{\theta}(\theta,\boldsymbol{o},M(\theta,\boldsymbol{o},t))} \mathrm{d}y(t). \nonumber 
\end{align}
and, similarly,
\begin{align}
\underbrace{\begin{pmatrix} \\ ~\\[-1mm] \mathrm{d}\hat{x}^{\boldsymbol{o}}(\theta,\boldsymbol{o},t) \\[1mm]  \\[4mm] \mathrm{d}\hat{\Sigma}^{\boldsymbol{o}}(\theta,\boldsymbol{o},t) \\[-2mm] ~ \end{pmatrix}}_{\mathrm{d}M^{\boldsymbol{o}}(\theta,\boldsymbol{o},t)} &= \underbrace{
\begin{pmatrix} 
-\theta\hat{x}^{\boldsymbol{o}}(\theta,\boldsymbol{o},t) \\
+ 2
(\boldsymbol{o}-\boldsymbol{o}_0)^{-3} \hat{x}(\theta,\boldsymbol{o},t)\hat{\Sigma}(\theta,\boldsymbol{o},t) \\
- 
(\boldsymbol{o}-\boldsymbol{o}_0)^{-2}\hat{x}^{\boldsymbol{o}}(\theta,\boldsymbol{o},t)\hat{\Sigma}(\theta,\boldsymbol{o},t) \\
- 
(\boldsymbol{o}-\boldsymbol{o}_0)^{-2}\hat{x}(\theta,\boldsymbol{o},t)\hat{\Sigma}^{\boldsymbol{o}}(\theta,\boldsymbol{o},t) \\[3mm]
- 2\theta \hat{\Sigma}^{\boldsymbol{o}}(\theta,\boldsymbol{o},t) + 2(\boldsymbol{o}-\boldsymbol{o}_0)^{-3} \hat{\Sigma}^2(\theta,\boldsymbol{o},t) \\
- 2(\boldsymbol{o}-\boldsymbol{o}_0)^{-2} \hat{\Sigma}(\theta,\boldsymbol{o},t)\hat{\Sigma}^{\boldsymbol{o}}(\theta,\boldsymbol{o},t) \\
\end{pmatrix}
}_{S'_{\boldsymbol{o}}(\theta,\boldsymbol{o},M(\theta,\boldsymbol{o},t))} \mathrm{d}t  \label{linear_gaussian_filter_tangent2} \\[1mm]
&\hspace{20mm}+  \underbrace{
\begin{pmatrix} \\[-2.2mm] 
-2(\boldsymbol{o}-\boldsymbol{o}_0)^{-3}\hat{\Sigma}(\theta,\boldsymbol{o},t) \\ 
+(\boldsymbol{o}-\boldsymbol{o}_0)^{-2}\hat{\Sigma}^{\boldsymbol{o}}(\theta,\boldsymbol{o},t) 
\\[-2.2mm] 
\vphantom{- c^2(\boldsymbol{o}-\boldsymbol{o}_0)^{-2}\hat{x}(\theta,\boldsymbol{o},t)\hat{\Sigma}^{\boldsymbol{o}}(\theta,\boldsymbol{o},t)} \\[3mm] 
0 \\ 
\vphantom{- 2(\boldsymbol{o}-\boldsymbol{o}_0)^{-2} \hat{\Sigma}(\theta,\boldsymbol{o},t)\hat{\Sigma}^{\boldsymbol{o}}(\theta,\boldsymbol{o},t)} 
\end{pmatrix}}_{T'_{\boldsymbol{o}}(\theta,\boldsymbol{o},M(\theta,\boldsymbol{o},t))} \mathrm{d}y(t).  \nonumber 
\end{align}
Finally, we can concatenate the (one-dimensional) signal, the (two-dimensional) filter, and the two (two-dimensional) tangent filters into a single diffusion process, namely, 
\begin{align}
\mathcal{X}(\theta,\boldsymbol{o},t) = 
&\big(x(t),
\underbrace{\hat{x}(\theta,\boldsymbol{o},t),\hat{\Sigma}(\theta,\boldsymbol{o},t)}_{\mathrm{vec}(M(\theta,\boldsymbol{o},t))},\underbrace{\hat{x}^{\theta}(\theta,\boldsymbol{o},t),\hat{\Sigma}^{\theta}(\theta,\boldsymbol{o},t)}_{\mathrm{vec}(M^{\theta}(\theta,\boldsymbol{o},t))},
\underbrace{\hat{x}^{\boldsymbol{o}}(\theta,\boldsymbol{o},t),\hat{\Sigma}^{\boldsymbol{o}}(\theta,\boldsymbol{o},t)}_{\mathrm{vec}(M^{\boldsymbol{o}}(\theta,\boldsymbol{o},t))}\big)^T.  \hspace{-5mm}
\end{align}
This process evolves according to an SDE of the form \eqref{compact_SDE}, which we obtain by stacking the signal equation \eqref{linear_gaussian1}, the Kalman-Bucy filtering equations \eqref{linear_gaussian_filter}, and the tangent Kalman-Bucy filtering equations \eqref{linear_gaussian_filter_tangent1} - \eqref{linear_gaussian_filter_tangent2}, before substituting the observation equation \eqref{linear_gaussian2}. For brevity, the explicit form of this equation is omitted. 
\end{remark_}

\subsection{Joint Parameter Estimation and Optimal Sensor Placement} \label{sec_joint_RML_OSP} 

We can finally now turn our attention to the problem of simultaneous online parameter estimation and online optimal sensor placement. As outlined in the introduction, we cast this as an unconstrained bilevel optimisation problem, in which the objective is to obtain  ${\theta}^{*}\in\Theta$, ${\boldsymbol{o}}^{*}(\theta^{*})\in\Omega^{n_y}$ such that
\begin{alignat}{2}
{\theta}^{*}&\in \argmin_{\theta\in\Theta} \left[-\tilde{\mathcal{L}}\big(\theta,\boldsymbol{o}^{*}(\theta)\big)\right]~~~,~~~{\boldsymbol{o}}^{*}(\theta)&&\in \argmin_{\boldsymbol{o}\in\Omega^{n_y}} \tilde{\mathcal{J}}\big({\theta},\boldsymbol{o}\big).
\end{alignat}
We should remark that, depending on our primary objective, we may instead specify $\tilde{\mathcal{J}}$ as the upper-level objective function, and $-\tilde{\mathcal{L}}$ as the lower-level objective function. Indeed, the subsequent methodology
 is generic to either case. As previously, we will consider a weaker version of this problem, in which we simply seek to obtain joint stationary points of $\tilde{\mathcal{L}}$ and $\tilde{\mathcal{J}}$. 

\subsubsection{The `Ideal' Algorithm}

To solve this bilevel optimisation problem, we propose a continuous-time, stochastic gradient descent algorithm, which combines the schemes in Sections \ref{sec:param_est} and \ref{sec:sensor}, c.f., \eqref{eq_sde} and \eqref{sensor_placement}. In particular, suppose some initialisation at ${\theta}_0\in\Theta$, ${\boldsymbol{o}}_0\in\Omega^{n_y}$. Then, simultaneously, we generate parameter estimates $\{\theta(t)\}_{t\geq 0}$ and optimal sensor locations $\{\boldsymbol{o}(t)\}_{t\geq 0}$ according to 
\begin{subequations}
\begin{alignat}{2}
\mathrm{d}{\theta}(t) \hspace{-.5mm} &=
 \left\{ 
\begin{array}{l} 
 \gamma_{1}(t) \big[\hat{C}^{\theta}(\theta(t),\boldsymbol{o}(t),t)\big]^T R^{-1}(\boldsymbol{o}(t))\big[\mathrm{d}y(t)-\hat{C}(\theta(t),\boldsymbol{o}(t),t)\mathrm{d}t\big] \\   0 \end{array} \right. 
 &&\begin{array}{ll} 
\hspace{-3mm} ,  &  \hspace{-1.5mm} {\theta}(t)\in\Theta, \\
\hspace{-3mm} , & \hspace{-1.5mm} {\theta}(t)\not\in\Theta,  
 \end{array} \hspace{-5mm}
 \label{RML_ROSP_1'}  \\
\mathrm{d}{\boldsymbol{o}}(t)  \hspace{-.5mm} &=  \left\{ 
\begin{array}{l}  
-\gamma_{2}(t)\big[\hat{j}^{\boldsymbol{o}}(\theta(t),\boldsymbol{o}(t),t)\big]^T\mathrm{d}t \\
0 
\end{array} 
\right.
 &&\begin{array}{ll} 
\hspace{-3mm}, &  \hspace{-1.5mm} {\boldsymbol{o}}(t)\in\Omega^{n_y},  \\
\hspace{-3mm} , & \hspace{-1.5mm} {\boldsymbol{o}}(t)\not\in\Omega^{n_y}.
 \end{array} \hspace{-5mm}
 \label{RML_ROSP_2'}
\end{alignat}
\end{subequations}
 

\subsubsection{The Implementable Algorithm}
 As outlined previously, it is typically not possible to implement Algorithm \eqref{RML_ROSP_1'} - \eqref{RML_ROSP_2'} in its current form, since it depends on the possibly intractable conditional expectations $\hat{C},\hat{C}^{\theta}$, and $\hat{j}^{\boldsymbol{o}}$. 
We can, however, obtain an implementable version of this algorithm by replacing these quantities by their (possibly approximate) finite-dimensional filter representations $\psi_{C}$, $\psi_{C}^{\theta}$, and $\psi_{j}^{\boldsymbol{o}}$. 
 For the purpose of our theoretical analysis, it will also be useful to rewrite Algorithm \eqref{RML_ROSP_1'} - \eqref{RML_ROSP_2'} in the form of Algorithm \eqref{alg2_1} - \eqref{alg2_2}, the generic two-timescale algorithm analysed in Section \ref{subsec:main2}.\footnote{We note that this also requires us to replace $\mathrm{d}y(t)$ in \eqref{RML_ROSP_1'} using the observation equation \eqref{obs_finite_dim_inf}.} After following these steps, we finally arrive at  
\begin{subequations}
\begin{alignat}{2}
\mathrm{d}{\theta}(t) &=
 \left\{ \begin{array}{l} -\gamma_{1}(t)\big[F({\theta}(t),{\boldsymbol{o}}(t),\mathcal{X}(t))\mathrm{d}t+\mathrm{d}\zeta_1(t)\big] \\ 0 \end{array} \right.
 &&\begin{array}{ll}  , &  {\theta}(t)\in\Theta, \\   , &  {\theta}(t)\not\in\Theta,  \end{array} \label{RML_ROSP_eq1} \\
\mathrm{d}{\boldsymbol{o}}(t)&= \left\{ \begin{array}{l}  -\gamma_{2}(t)\big[G({\theta}(t),{\boldsymbol{o}}(t),\mathcal{X}(t))\mathrm{d}t\big] \\ 0 \end{array} \right.
 &&\begin{array}{ll}   , & {\boldsymbol{o}}(t)\in\Omega^{n_y}, \\ , &  {\boldsymbol{o}}(t)\not\in\Omega^{n_y},  \end{array}  \label{RML_ROSP_eq2} 
\end{alignat}
\end{subequations}

where $F$ and $G$ are the $\mathbb{R}^{n_{\theta}}$- and $\mathbb{R}^{n_yn_{\boldsymbol{o}}}$-valued functions defined according to
\begin{align}
F(\theta(t),\boldsymbol{o}(t),\mathcal{X}(t))&= -\big[\psi_{C}^{\theta}(\theta(t),\boldsymbol{o}(t),M(t),M^{\theta}(t))\big]^T R^{-1}(\boldsymbol{o}(t)) \label{F_def_}  \\
&\hspace{7.5mm}\big[C(\theta^{*},\boldsymbol{o}(t),x(t)) - \psi_{C}(\theta(t),\boldsymbol{o}(t),M(t))\big], \nonumber  \\[2mm]
G(\theta(t),\boldsymbol{o}(t),\mathcal{X}(t))& = \psi_{j}^{\boldsymbol{o}}(\theta(t),\boldsymbol{o}(t),M(t),M^{\boldsymbol{o}}(t))^T, \label{G_def_} 
\intertext{
where $\zeta_{1}$ is the $\mathbb{R}^{n_{\theta}}$-valued semi-martingale which evolves according to the SDE}
\mathrm{d}\zeta_1(\theta(t),\boldsymbol{o}(t),t) &= \underbrace{\big[\psi_{C}^{\theta}(\theta(t),\boldsymbol{o}(t),M(t),M^{\theta}(t))\big]^T\mathcal{R}^{-1}(\boldsymbol{o}(t))}_{\zeta_1^{(2)}(\theta(t),\boldsymbol{o}(t),\mathcal{X}(t))}\mathrm{d}w(t), \label{zeta_1_def}
\end{align}
and where $\smash{\mathcal{X}=\{\mathcal{X}(t)\}_{t\geq 0} = \{\mathcal{X}(\theta(t),\boldsymbol{o}(t),t)\}_{t\geq 0}}$
 is the $\mathbb{R}^{N}$-valued diffusion process defined in \eqref{compact_SDE},  
consisting of latent state, the filter, and the tangent filters, now integrated along the path of the algorithm iterates. We emphasise that this algorithm can be implemented for both exact (e.g. Kalman-Bucy, Ben\^es) and approximate (e.g., ensemble Kalman-Bucy, unscented Kalman-Bucy, projection) filters.

\begin{remark_}
Let us return to the one-dimensional linear Gaussian example considered in the previous section. We can now provide the specific joint online parameter estimation and optimal sensor placement algorithm for this model. In particular, substituting our previous expressions for $C(\theta,\boldsymbol{o},x)$, $R(\boldsymbol{o})$, $\psi_{C}(\theta,\boldsymbol{o},M)$, $\psi_{C}^{\theta}(\theta,\boldsymbol{o},M,M^{\theta})$, and $\psi_{j}^{\boldsymbol{o}}(\theta,\boldsymbol{o},M,M^{\boldsymbol{o}})$, c.f. \eqref{linear_gaussian2}, \eqref{hatC}, \eqref{Chat_gauss} and \eqref{jhat_gauss}, into the equations for $F$, $G$, and $\zeta_1$, c.f. \eqref{F_def_} , \eqref{G_def_} and \eqref{zeta_1_def}, 
we obtain the update equations
\begin{subequations}
\begin{alignat}{2}
\mathrm{d}\theta(t) &= -\gamma_{1}(t)\left[-\hat{x}^{\theta}(\theta(t),\boldsymbol{o}(t),t) (\boldsymbol{o}(t)-\boldsymbol{o}_0)^{-2} (x(t) - \hat{x}(\theta(t),\boldsymbol{o}(t),t) )\mathrm{d}t \right. \\
 &\hspace{16mm} \left.+\hat{x}^{\theta}(\theta(t),\boldsymbol{o}(t),t) (\boldsymbol{o}(t)-\boldsymbol{o}_0)^{-2} \mathrm{d}w(t))\right] \nonumber \\[2mm]
\mathrm{d}\boldsymbol{o}(t) &= -\gamma_{2}(t) \left[\hat{\Sigma}^{\boldsymbol{o}}(\theta(t),\boldsymbol{o}(t),t)\right] \mathrm{d}t.
\end{alignat}
\end{subequations}
where the filter mean $\hat{x}(\theta(t),\boldsymbol{o}(t),t)$, and the filter derivatives $\hat{x}^{\theta}(\theta(t),{o}(t),t)$ and $\hat{\Sigma}^{\boldsymbol{o}}(\theta(t),\boldsymbol{o}(t),t)$, evolve according to the Kalman-Bucy filter equation \eqref{linear_gaussian_filter}, and the tangent Kalman-Bucy filter equations \eqref{linear_gaussian_filter_tangent1} - \eqref{linear_gaussian_filter_tangent2}, now evaluated along the path of the algorithm iterates.
\end{remark_}

\subsubsection{Main Result}
We will analyse Algorithm \eqref{RML_ROSP_eq1} - \eqref{RML_ROSP_eq2} under most of the assumptions introduced in Section \ref{subsec:main2} for the general two-timescale gradient descent algorithm with Markovian dynamics,\footnote{In particular, we now no longer require two of the conditions relating to the additive, state-dependent noise processes $\{\zeta_i(t)\}_{t\geq 0}$, namely Assumptions \ref{assumption5a} and \ref{assumption5c}, as these can be shown to follow directly from Assumption \ref{assumption5b}.} in addition to the assumptions introduced in Section \ref{sec:filter} for the filter and filter derivatives. 

 In order to state our main result, we must first define the representations of the asymptotic log-likelihood and the asymptotic sensor placement objective, c.f. \eqref{asymptotic_ll} and \eqref{asymptotic_obj}, in terms of the (possibly approximate) finite dimensional filter. In particular, we will write
\begin{align}
\tilde{\mathcal{L}}^{\text{(filter)}}(\theta,\boldsymbol{o}) &= \lim_{t\rightarrow\infty}\frac{1}{t}\left[\int_0^t R^{-1}(\boldsymbol{o}) \psi_{C}(\theta,\boldsymbol{o},M(\theta,\boldsymbol{o},s))\cdot\mathrm{d}y(s) \right. \label{approxL} \\
&\hspace{17mm}\left.- \frac{1}{2}\int_0^t ||R^{-\frac{1}{2}}(\boldsymbol{o})\psi_{C}(\theta,\boldsymbol{o},M(\theta,\boldsymbol{o},s))||^2\mathrm{d}s\right] \nonumber \\[2mm]
\tilde{\mathcal{J}}^{\text{(filter)}}(\theta,\boldsymbol{o})&=\lim_{t\rightarrow\infty}\frac{1}{t}\left[\int_0^t \psi_{j}(\theta,\boldsymbol{o},M(\theta,\boldsymbol{o},s))\mathrm{d}s\right]. \label{approxJ}
\end{align}


We are now ready to state our main result on the convergence of Algorithm \eqref{RML_ROSP_eq1} - \eqref{RML_ROSP_eq2}. 


\begin{proposition} \label{theorem2}
Assume that Assumptions \ref{assumption1a}, \ref{assumption2a} -  \ref{assumption4aii}, \ref{assumption5b}, \ref{assumption4} - \ref{assumption6}, and \ref{filter_assumption1} - \ref{filter_assumption3} hold. Then, with probability one, 
\begin{align}
\lim_{t\rightarrow\infty}\nabla_{\theta}\tilde{\mathcal{L}}^{ \mathrm{(filter)}}(\theta(t),\boldsymbol{o}(t)) = \lim_{t\rightarrow\infty}\nabla_{\boldsymbol{o}}\tilde{\mathcal{J}}^{ \mathrm{(filter)}}(\theta(t),\boldsymbol{o}(t)) = 0,
\end{align}
or
\begin{align}
\lim_{t\rightarrow\infty}(\theta(t),\boldsymbol{o}(t)) \in \{(\theta,\boldsymbol{o}):\theta\in \partial\Theta\cup \boldsymbol{o}\in \partial\Omega^{n_y}\}.
\end{align}
\end{proposition}

\begin{remark__}
We assume here that the stated assumptions hold for the diffusion process $\mathcal{X}$ defined in \eqref{compact_SDE}, the functions $F$ and $G$ defined in \eqref{F_def_} and \eqref{G_def_}, and the semi-martingale $\zeta_1$ defined in \eqref{zeta_1_def}. Moreover, where necessary, we replace the algorithm iterates $(\alpha,\beta)$ by $(\theta,\boldsymbol{o})$, and the functions $f$ and $g$ by $\tilde{\mathcal{L}}^{(\text{filter})}$ and $\tilde{\mathcal{J}}^{(\text{filter})}$, as defined in \eqref{approxL} and \eqref{approxJ}.
\end{remark__}

\begin{proof}
See Appendix \ref{appendix:proof_theorem2}.
\end{proof}

Proposition \ref{theorem2} is obtained as a corollary of Theorem \ref{theorem1a}. In particular, Algorithm \eqref{RML_ROSP_eq1} - \eqref{RML_ROSP_eq2}  is a special case of Algorithm \eqref{alg2_1} - \eqref{alg2_2}, in which the additive noise for the slow process is defined by equation \eqref{zeta_1_def}, and the additive noise for the fast process is identically equal to zero. Aside from notational differences, the modifications in the statement of this theorem, when compared to Theorem \ref{theorem1a}, are due solely to the inclusion of the projection which ensures that the algorithm iterates remain in the open sets $\Theta\in\mathbb{R}^{n_{\theta}}$, $\Omega^{n_y}\in\mathbb{R}^{n_yn_{\boldsymbol{o}}}$ with probability one.

Proposition \ref{theorem2} extends Theorem 1 in \cite{Surace2019}, in which almost sure convergence of the online parameter estimate was established under slightly weaker conditions. In particular, the almost sure convergence results in \cite{Surace2019} does not depend on almost sure boundedness of the algorithm iterates. The method of proof, however, is entirely different (see discussion in Section \ref{subsec:main2}). We remark, as in the previous section, that our theorem (and its proof) still holds upon restriction to a single-timescale; that is, under the assumption that only the parameters are estimated, while the sensor locations are fixed, or vice versa. In this case, of course, we only require assumptions which relate to the quantity of interest. Thus, upon restriction to a single-timescale (i.e., assuming that the sensors are fixed), our theorem reduces to the result in \cite{Surace2019}, while our proof provides an entirely different proof for that result.

\subsubsection{Extensions for Approximate Filters}
Proposition \ref{theorem2} guarantees that the online parameter estimates and the optimal sensor placements generated by Algorithm \eqref{RML_ROSP_eq1} - \eqref{RML_ROSP_eq2} converge to the stationary points of the finite-dimensional filter representations of the asymptotic log-likelihood and the sensor placement objective function, namely, $\tilde{\mathcal{L}}^{\text{(filter)}}(\theta)$ and $\tilde{\mathcal{J}}^{\text{(filter)}}(\theta)$. In the case that one can obtain exact solutions to the Kushner Stratonovich equation (e.g., using the Kalman-Bucy filter for a linear Gaussian model), these representations will be exact, and thus this proposition implies convergence to the stationary points of the `true' objective functions $\tilde{\mathcal{L}}(\theta)$ and $\tilde{\mathcal{J}}(\theta)$. 

On the other hand, if it is only possible to obtain approximate solutions to the Kushner-Stratonovich equation (e.g., using a continuous time particle filter for a non-linear model), Proposition \ref{theorem2} still guarantees convergence, but now to the stationary points of an `approximate' asymptotic log-likelihood and an `approximate' asymptotic sensor placement objective function, 
namely, the representations of these functions in terms of the approximate finite-dimensional filter. In this case, it is clear that the asymptotic properties of the online parameter estimates and optimal sensor placements with respect to the `true' objective functions will be determined by the properties of the approximate filter. In particular, in order to obtain convergence (e.g., in $\mathbb{L}^p$) to the stationary points of the true objective functions, one now requires bounds, preferably uniform in time, on terms such as
\begin{align}
\mathbb{E}\left[||\psi_{C}(\theta,\boldsymbol{o},M(t)) - \hat{C}(\theta,\boldsymbol{o},t)||^{n}\right]~,~\mathbb{E}\left[||\psi_{j}(\theta,\boldsymbol{o},M(t)) - \hat{j}(\theta,\boldsymbol{o},t)||^{n}\right]
\end{align}
We discuss this point in greater depth in Appendix \ref{app:theorem2_ext}, and sketch the details of how one can obtain an $\mathbb{L}^p$ convergence result of this type for the Ensemble Kalman-Bucy Filter (EnKBF) (e.g., \cite{DeWiljes2018,Moral2018}). 

\subsubsection{Sufficient Conditions}
We conclude this section with some brief remarks on the assumptions required for Proposition \ref{theorem2} (see also \cite{Surace2019}). The majority of these assumptions are fairly classical, namely, those on the learning rate (Assumption \ref{assumption1a}), the additive noise process $\zeta_1$ (Assumption \ref{assumption5b}), the stability of the algorithm iterates (Assumption \ref{assumption4}), and the stationary points of the asymptotic objective functions (Assumptions \ref{assumption5} - \ref{assumption6}). Meanwhile, our assumptions on the filter (Assumptions \ref{filter_assumption1} - \ref{filter_assumption3}) are relatively weak, and are satisfied by many exact and approximate filters. In fact, using the results recently established in \cite{Bhudisaksang2020}, it may be possible to relax Assumption \ref{filter_assumption3} further, and allow the evolution equation for the filter to include a jump process. This would further extend the applicability of this result, allowing for a broader class of continuous-time particle filters (e.g., \cite{DelMoral2000}).


It remains to consider the assumptions relating to the diffusion process $\mathcal{X}$ (Assumptions \ref{assumption2a} - \ref{assumption4aii}). 
These include ergodicity (Assumption \ref{assumption2a}), uniformly bounded moments (Assumption \ref{assumption2ai}, Assumption \ref{assumption4aii}), polynomial growth for the diffusion term in the associated SDE (Assumption \ref{assumption4a}), and existence and regularity of solutions of the Poisson equations associated with the generator of this process and the asymptotic objective functions (Assumption \ref{assumption4aiii}). We provide sufficient conditions for one of these assumptions (Assumption \ref{assumption4a}) in Appendix \ref{appendix:loglik_sufficient}. In particular, we show that this assumption can be replaced by the slightly weaker assumptions that (i) certain functions appearing in the definition of Algorithm \eqref{RML_ROSP_eq1} - \eqref{RML_ROSP_eq2} have the polynomial growth property and (ii) for certain functions satisfying the polynomial growth property, the Poisson equation admits a unique solution which also has this property. 

In general, while our conditions are certainly necessary in order to establish almost sure convergence, they are somewhat strong, and in general must be verified on a case by case basis. This being said, in the linear Gaussian case, one can obtain sufficient conditions which are straightforward to verify (see Appendix A in \cite{Sharrock2020}). In particular, these conditions coincide with standard conditions required for stability of the Kalman-Bucy filter, and are thus arguably the weakest under which an asymptotic result of this type can be established.

More broadly, the problem of obtaining more easily verifiable sufficient conditions remains open. Indeed, the diffusion process is generally highly degenerate, and thus standard sufficient conditions for non degenerate elliptic diffusion processes (see Appendix \ref{appendix_sufficient_conditions_elliptic}), do not apply (e.g., \cite{Pardoux2003,Pardoux2001}). 
In the case that an exact, finite-dimensional solution to the Kushner-Stratonovich equation exists, ergodicity of the optimal filter follows directly from ergodicity of the latent signal process and the non-degeneracy of the observation process (e.g., \cite{Budhiraja2003,Kunita1991}), 
but ergodicity of the tangent filter(s) must still established. Meanwhile, in the case that only an approximate, finite-dimensional solution to the Kushner-Stratonovich equation exists, there is no guarantee that the approximate filter is ergodic, let alone the tangent filter(s).

\section{Numerical Examples}
\label{sec:numerics}

To illustrate the results of Section \ref{sec:RML_ROSP}, we now provide two examples of joint online parameter estimation and optimal sensor placement. In both cases, we study the numerical performance of the proposed two-timescale stochastic gradient descent algorithm, and verify numerically the convergence of the parameter estimates and the sensor placements. In the first case, we also provide explicit derivations of the parameter and sensor update equations. In the second case, these details,  as well an explicit verification of the conditions of Proposition \ref{theorem2},  appear in a separate paper \cite{Sharrock2020}.

\subsection{One-Dimensional Benes Filter} We first consider a one-dimensional, partially observed diffusion process defined by

\begin{alignat}{2}
\mathrm{d}x(t) &= \mu\sigma\tanh\left[\frac{\mu}{\sigma}x(t)\right]\mathrm{d}t + \mathrm{d}w(t)~,~~~&&x(0)=0, \label{benes1} \\
\mathrm{d}y(t) &= cx(t) \mathrm{d}t + \mathrm{d}v(t)~,~~~&&y(0)=0,\label{benes2}
\end{alignat}

where $w=\{w(t)\}_{t\geq 0}$ and $v=\{v(t)\}_{t\geq 0}$ are independent, one-dimensional Brownian motions with incremental variances $q(\theta)$  $= \sigma^2$ and $r(\boldsymbol{o}) = \tau^2 + (\boldsymbol{o}-\boldsymbol{o}_0)^2$, respectively, for some fixed positive constant $\tau\in\mathbb{R}_{+}$. We assume that the initial condition $x_0\in\mathbb{R}$, that the parameters $\mu,c\in\mathbb{R}$ and $\sigma\in\mathbb{R}_{+}$, respectively, and that the sensor location $\boldsymbol{o}\in\mathbb{R}$. We thus have a three-dimensional parameter vector $\theta = (\mu,c,\sigma)\in \mathbb{R}^2\times\mathbb{R}_{+}$, and a single, one-dimensional sensor location $\boldsymbol{o}\in\mathbb{R}$. 

This system has an analytic, finite-dimensional solution, known as the Bene\v{s} filter \cite{Benes1981}. Namely, the conditional law of the latent signal process $x=\{x(t)\}_{t\geq 0}$ given the history of observations $\mathcal{F}_t^Y = \sigma(y(s):0\leq s\leq t)$ is a weighted mixture of two normal distributions \cite[Chapter 6]{Bain2009}, which takes the form
\begin{align}
\pi_t &= w^{+}(\theta,\boldsymbol{o},t) \hspace{.5mm}\mathcal{N}\left(\frac{A^{+}(\theta,\boldsymbol{o},t)}{2B(\theta,\boldsymbol{o},t)},\frac{1}{2B(\theta,\boldsymbol{o},t)}\right) \\
&+ w^{-}(\theta,\boldsymbol{o},t)\hspace{.5mm}\mathcal{N}\left(\frac{A^{-}(\theta,\boldsymbol{o},t)}{2B(\theta,\boldsymbol{o},t)},\frac{1}{2B(\theta,\boldsymbol{o},t)}\right) \nonumber,
\end{align}
where
\begin{subequations}
\begin{align}
w^{\pm}(\theta,\boldsymbol{o},t)&=\frac{\exp\big(\tfrac{A^{\pm}(\theta,\boldsymbol{o},t)^2}{4B(\theta,\boldsymbol{o},t)}\big)}{\exp\big(\frac{A^{+}(\theta,\boldsymbol{o},t)^2}{4B(\theta,\boldsymbol{o},t)}\big)+\exp\big(\frac{A^{-}(\theta,\boldsymbol{o},t)^2}{4B(\theta,\boldsymbol{o},t)}\big)}\\
A^{\pm}(\theta,\boldsymbol{o},t)&=\pm\frac{\mu}{\sigma} + cr^{-1}(\boldsymbol{o})\int_0^t\frac{\sinh(c\sigma r^{-\frac{1}{2}}(\boldsymbol{o}) s)}{\sinh(c\sigma r^{-\frac{1}{2}}(\boldsymbol{o}) t)}\mathrm{d}y(s) \\
B(\theta,\boldsymbol{o},t)&= \frac{cr^{-\frac{1}{2}}(\boldsymbol{o})}{2\sigma}\coth(c\sigma r^{-\frac{1}{2}}(\boldsymbol{o}) t).
\end{align}
\end{subequations}
It follows, in particular, that the optimal filter has a (non-unique) two-dimensional representation, which we will write as $M(\theta,\boldsymbol{o},t)= (m(\theta,\boldsymbol{o},t),P(\theta,\boldsymbol{o},t))^T$. In this case, we choose to define
\begin{subequations}
\begin{align}
m(\theta,\boldsymbol{o},t) &= c \frac{A^{\pm}(\theta,\boldsymbol{o},t)\mp \frac{\mu}{\sigma}}{2B(\theta,\boldsymbol{o},t)} = c\sigma r^{-\frac{1}{2}}(\boldsymbol{o}) \frac{\int_0^t\sinh\left(cr^{-\frac{1}{2}}(\boldsymbol{o})\sigma s\right)\mathrm{d}y(t)}{\cosh\left(cr^{-\frac{1}{2}}(\boldsymbol{o}) \sigma t\right)} \\
P(\theta,\boldsymbol{o},t) &= \frac{1}{2B(\theta,\boldsymbol{o},t)} = \frac{\sigma r^{\frac{1}{2}}(\boldsymbol{o})}{c}\tanh\left(cr^{-\frac{1}{2}}(\boldsymbol{o})\sigma t\right).
\end{align}
\end{subequations}
This choice implies that the finite-dimensional filter evolves according to an SDE of the required form, namely (e.g., \cite{Sarkka2019}) 
\begin{equation}
\mathrm{d}M(\theta,\boldsymbol{o},t) = \begin{pmatrix} -c^2r^{-1}(\boldsymbol{o})P(\theta,\boldsymbol{o},t)m(\theta,\boldsymbol{o},t) \\ \sigma^2-c^2r^{-1}(\boldsymbol{o})P^2(\theta,\boldsymbol{o},t) \end{pmatrix} \mathrm{d}t + \begin{pmatrix} cr^{-1}(\boldsymbol{o}) P(\theta,\boldsymbol{o},t) \\ 0 \end{pmatrix} \mathrm{d}y(t). \label{M_sde}
\end{equation}

The equations for the tangent filters, namely $M^{\mu}(\theta,\boldsymbol{o},t)$, $M^{\sigma}(\theta,\boldsymbol{o},t)$, $M^{c}(\theta,\boldsymbol{o},t)$ and $M^{\boldsymbol{o}}(\theta,\boldsymbol{o},t)$, can then be obtained by (formal) differentiation of this equation with respect to the relevant variable. Illustratively, for the first parameter, we have
\begin{equation}
\mathrm{d}M^{\mu}(\theta,\boldsymbol{o},t) = 
\begin{pmatrix} 
-c^2r^{-1}(\boldsymbol{o})P^{\mu}(\theta,\boldsymbol{o},t)m(\theta,\boldsymbol{o},t) \\
~-c^2r^{-1}(\boldsymbol{o}) P(\theta,\boldsymbol{o},t)m^{\mu}(\theta,\boldsymbol{o},t)  \\[3mm] 
-2c^2r^{-1}(\boldsymbol{o})P(\theta,\boldsymbol{o},t)P^{\mu}(\theta,\boldsymbol{o},t)
\end{pmatrix} 
\mathrm{d}t 
+ 
\begin{pmatrix} 
cr^{-1}(\boldsymbol{o}) P^{\mu}(\theta,\boldsymbol{o},t) \\ 
~\\[2mm]
0
\end{pmatrix} 
\mathrm{d}y(t).
\end{equation}

We should remark that $m(\theta,\boldsymbol{o},t)$ and $P(\theta,\boldsymbol{o},t)$ do not correspond directly to the mean $\hat{x}(\theta,\boldsymbol{o},t)$ and variance $\hat{\Sigma}(\theta,\boldsymbol{o},t)$ of the optimal filter.  However, these quantities can be computed 
as \cite{Sarkka2019}
\begin{subequations}
\begin{align}
\hat{x}(\theta,\boldsymbol{o},t) 
&= m(\theta,\boldsymbol{o},t) + \frac{\mu}{\sigma} P(\theta,\boldsymbol{o},t)\tanh\left(\frac{\mu}{\sigma} m(\theta,\boldsymbol{o},t)\right),  \label{posterior_mean} \\
\hat{\Sigma}(\theta,\boldsymbol{o},t)
&=P(\theta,\boldsymbol{o},t) + \frac{\mu^2}{\sigma^2}\left(1-\tanh^2(\frac{\mu}{\sigma} m(\theta,\boldsymbol{o},t))\right)P^2(\theta,\boldsymbol{o},t). \label{posterior_cov}
\end{align}
\end{subequations}
We    can then compute the conditional expectations 
\begin{subequations}
\begin{align}
\hat{C}(\theta,\boldsymbol{o},t) &=
 \psi_{C}(\theta,\boldsymbol{o},M(\theta,\boldsymbol{o},t))= c\hat{x}(\theta,\boldsymbol{o},t), \\
\hat{j}(\theta,\boldsymbol{o},t)&=
\psi_{j}(\theta,\boldsymbol{o},M(\theta,\boldsymbol{o},t))= \mathrm{Tr}\big[\hat{\Sigma}(\theta,\boldsymbol{o},t)\big].
\end{align}
\end{subequations}
It is now straightforward to obtain the explicit form of the two-timescale, joint online parameter estimation and optimal sensor placement algorithm for this system. In particular, we have 
\begin{subequations}
\begin{align}
\mathrm{d}\mu(t)&=-\gamma_{1,\mu}(t) \hspace{.2mm} c \hspace{.4mm} \hat{x}^{\mu}(\theta(t),\boldsymbol{o}(t),t)\left[c\hspace{.4mm}\hat{x}(\theta(t),\boldsymbol{o}(t),t)\mathrm{d}t-\mathrm{d}y(t)\right] \\
\mathrm{d}\sigma(t)&=-\gamma_{1,\sigma}(t) \hspace{.2mm}  c \hspace{.4mm} \hat{x}^{\sigma}(\theta(t),\boldsymbol{o}(t),t)\left[c\hspace{.4mm}\hat{x}(\theta(t),\boldsymbol{o}(t),t)\mathrm{d}t-\mathrm{d}y(t)\right] \\
\mathrm{d}c(t)&=-\gamma_{1,c}(t) \left[\hat{x}(\theta(t),\boldsymbol{o}(t),t) + c\hat{x}^{c}(\theta(t),\boldsymbol{o}(t),t)\right]\left[c\hspace{.4mm}\hat{x}(\theta(t),\boldsymbol{o}(t),t)\mathrm{d}t-\mathrm{d}y(t)\right] \hspace{-5mm} \\[1mm]
\mathrm{d}\boldsymbol{o}(t)&=-\gamma_{2,\boldsymbol{o}}(t)\mathrm{Tr}\big[\hat{\Sigma}^{\boldsymbol{o}}(\theta(t),\boldsymbol{o}(t),t)\big].
\end{align}
\end{subequations}
where $\hat{x}^{\mu}$, $\hat{x}^{\sigma}$, $\hat{x}^c$ and $\hat{\Sigma}^{\boldsymbol{o}}$ are the filter derivatives of the posterior mean and the posterior variance, respectively. These quantities are obtained by differentiating \eqref{posterior_mean} - \eqref{posterior_cov} with respect to the relevant variable, and substituting the filter and the relevant tangent filter where appropriate.

The performance of the two-timescale stochastic gradient descent algorithm is illustrated in Figure \ref{fig1}. In this simulation, we assume that the parameters $\sigma^2 =\sigma^2_{*}= 4$, $c = c_{*}=0.7$ and $\tau^2=\tau_{*}^2=2$ are fixed, while the parameter $\mu$ is learned. The true value of this parameter is given by $\mu^{*} = 3$, and we consider two initial parameter estimates $\mu_0 = \{1,7\}$. Meanwhile, the optimal sensor placement and the initial sensor placement is given by $\boldsymbol{o}_{*} = 4$, and we consider initial sensor placements $\boldsymbol{o}_0 = \{2,6\}$. We remark that, as in any gradient based algorithm, the convergence of the proposed scheme may be sensitive to initialisation. In this case, however, it appears to be robust to this choice. 

 For simplicity, we choose to integrate all SDEs using a standard Euler-Maruyama discretisation, with  $\Delta t = 0.01$, although similar results are obtained for other choices of $\Delta t$.  We provide results for several choices of learning rates of the form $\gamma_{\mu}(t) = \gamma_{\mu}^{0} t^{-\eta_{\mu}}$ and $\gamma_{o}(t) = \gamma_{o}^{0} t^{-\eta_o}$, where $0<\gamma_{\mu}^{0},\gamma_{o}^{0}<\infty$, and $0<\eta_{\mu}<\eta_{o}<1$. We consider both learning rates for which the learning rate condition in Proposition \ref{theorem2} is satisfied (when $0.5\leq \eta_{\mu},\eta_{o}\leq 1$), and learning rates for which this condition is violated (when $0<\eta_{\mu},\eta_{o}<0.5$). In this case, the online parameter estimates and optimal sensor converge to their true values, regardless of the choice of learning rate or the initialisation. We note, however, that the rate of convergence does depend on the choice of learning rate. 

\begin{figure}[tbhp]
\centering
  \subfloat[Online parameter estimates.]{\label{fig_1a}\includegraphics[trim = 0 0 0 1cm,width=0.42\textwidth,clip]{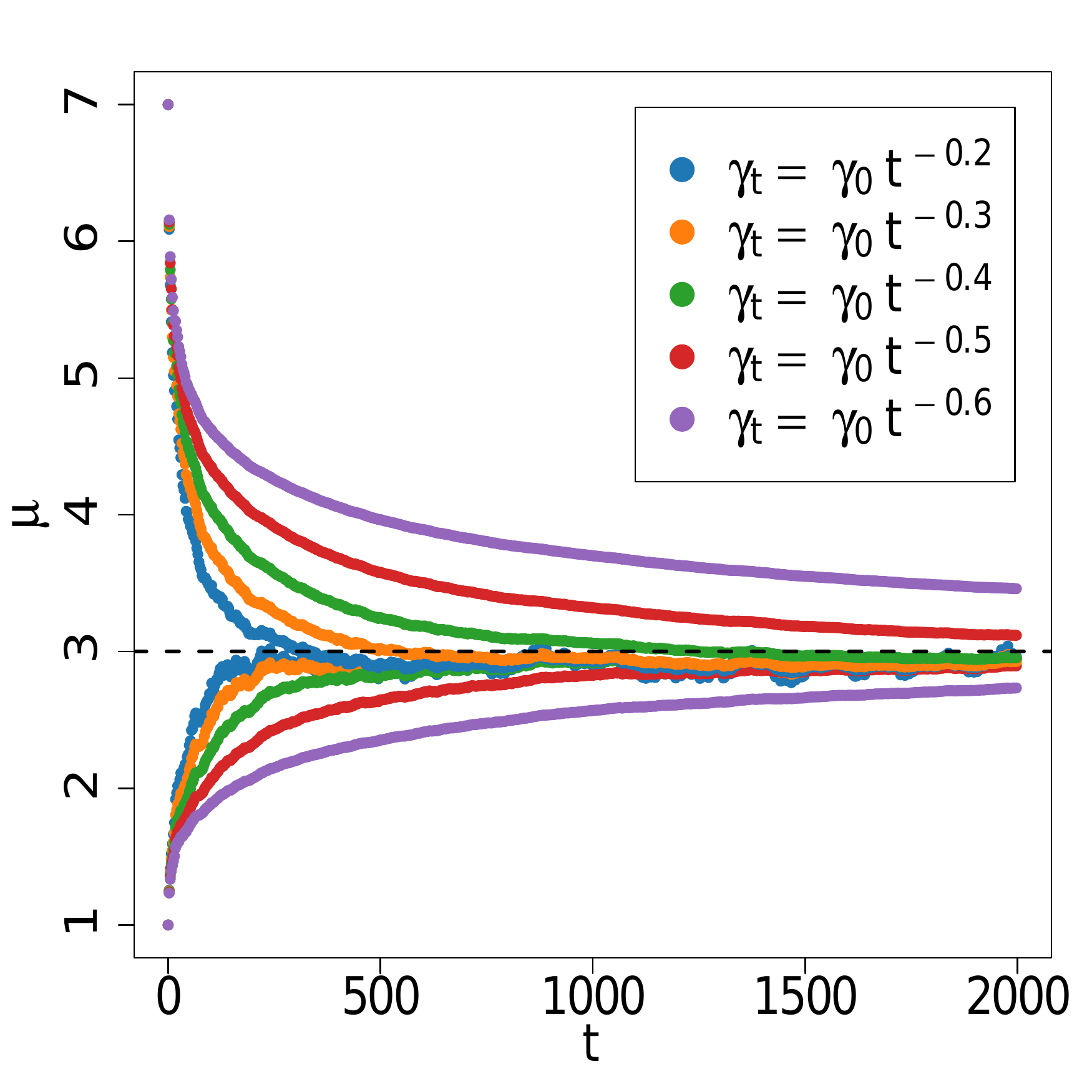}}
  \hspace{6mm}
  \subfloat[Optimal sensor placements.]{\label{fig_1b}\includegraphics[trim = 0 0 0 1cm,width=0.42\textwidth,clip]{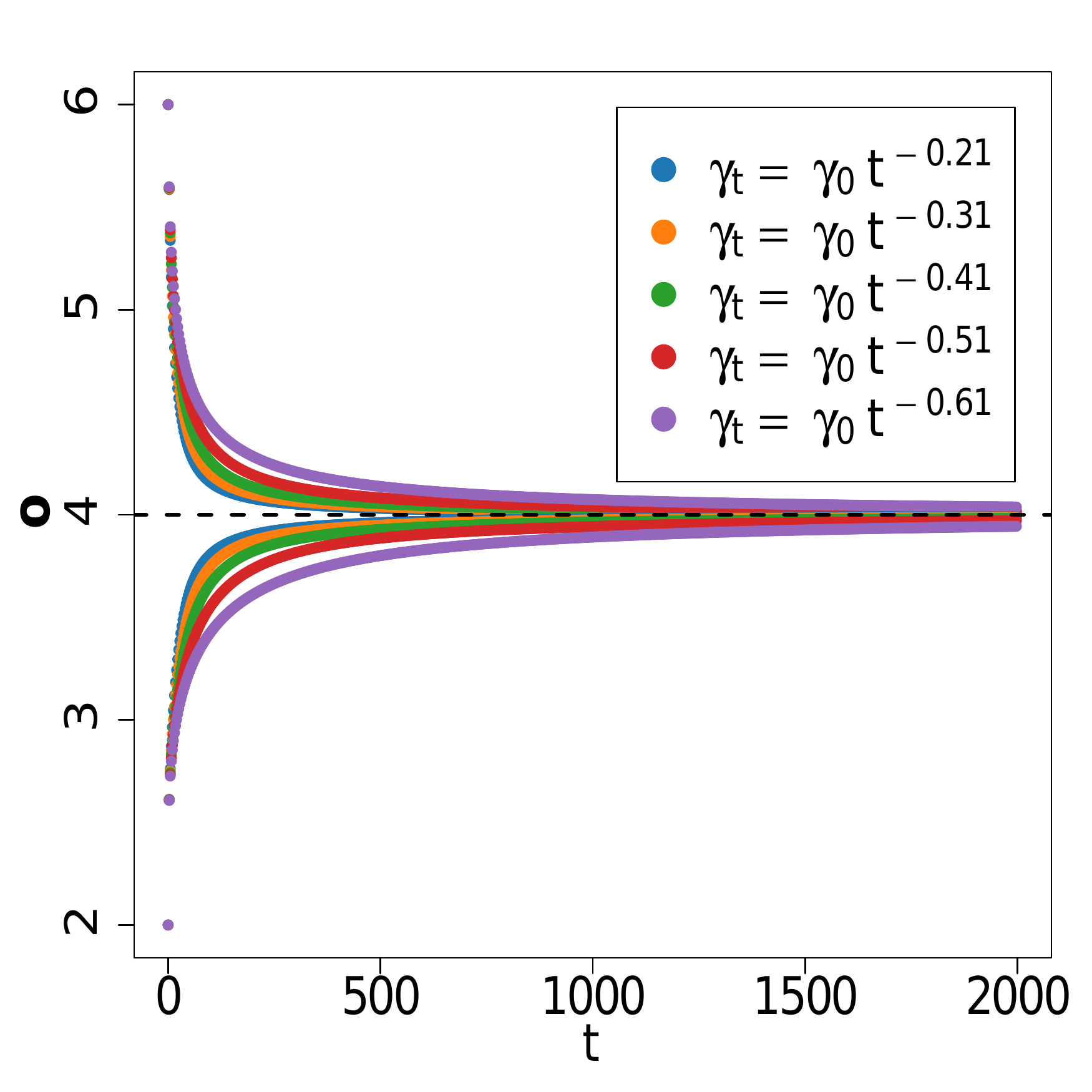}}
  \vspace{-2mm}
\caption{The sequence of online parameter estimates \& optimal sensor placements for the partially observed Bene\v{s} SDE, for several choices of the learning rates $\{\gamma_{\theta}(t)\}_{t\geq 0}$, $\{\gamma_{\boldsymbol{o}}(t)\}_{t\geq 1}$.}
\label{fig1}
\end{figure}

To obtain (optimal) convergence rates which are independent of the choice of learning rate, a standard approach in discrete-time, including the two-timescale case \cite{Mokkadem2006}, is to use Polyak-Ruppert averaging \cite{Polyak1990,Ruppert1988}. 
In the spirit of this scheme, in the continuous time, two-timescale setting, we can consider new parameter estimates $\{\bar{\theta}(t)\}_{t\geq 0}$ and $\{\bar{\boldsymbol{o}}(t)\}_{t\geq0}$ defined according to
\begin{align}
\bar{\theta}(t) = \frac{1}{t}\int_0^t \theta(s)\mathrm{d}s~~~,~~~\bar{\boldsymbol{o}}(t) = \frac{1}{t}\int_0^t \boldsymbol{o}(s)\mathrm{d}s.
\end{align}
A detailed theoretical analysis of this approach, which extends the results in \cite{Mokkadem2006} to the continuous time setting using the tools established in \cite{Pardoux2003,Pardoux2001,Sirignano2017a,Sirignano2020a}, is beyond the scope of this paper. We do provide tentative numerical evidence, however, to suggest that such results can also be expected to hold in continuous time. In particular, in Figure \ref{fig2}, we plot the sequence of averaged optimal sensor placements (Figure \ref{fig_2a}) and the corresponding $\mathbb{L}^{1}$ error for large times (Figure \ref{fig_2b}), for several choices of the learning rate. The latter illustration, in particular, indicates that the convergence rate is now independent of the learning rate. One can obtain similar results for the corresponding sequence of averaged online parameter estimates (plots omitted).

\begin{figure}[tbhp]
\centering
  \subfloat[Optimal sensor placements.]{\label{fig_2a}\includegraphics[trim = 0 0 0 1cm,width=0.42\textwidth,clip]{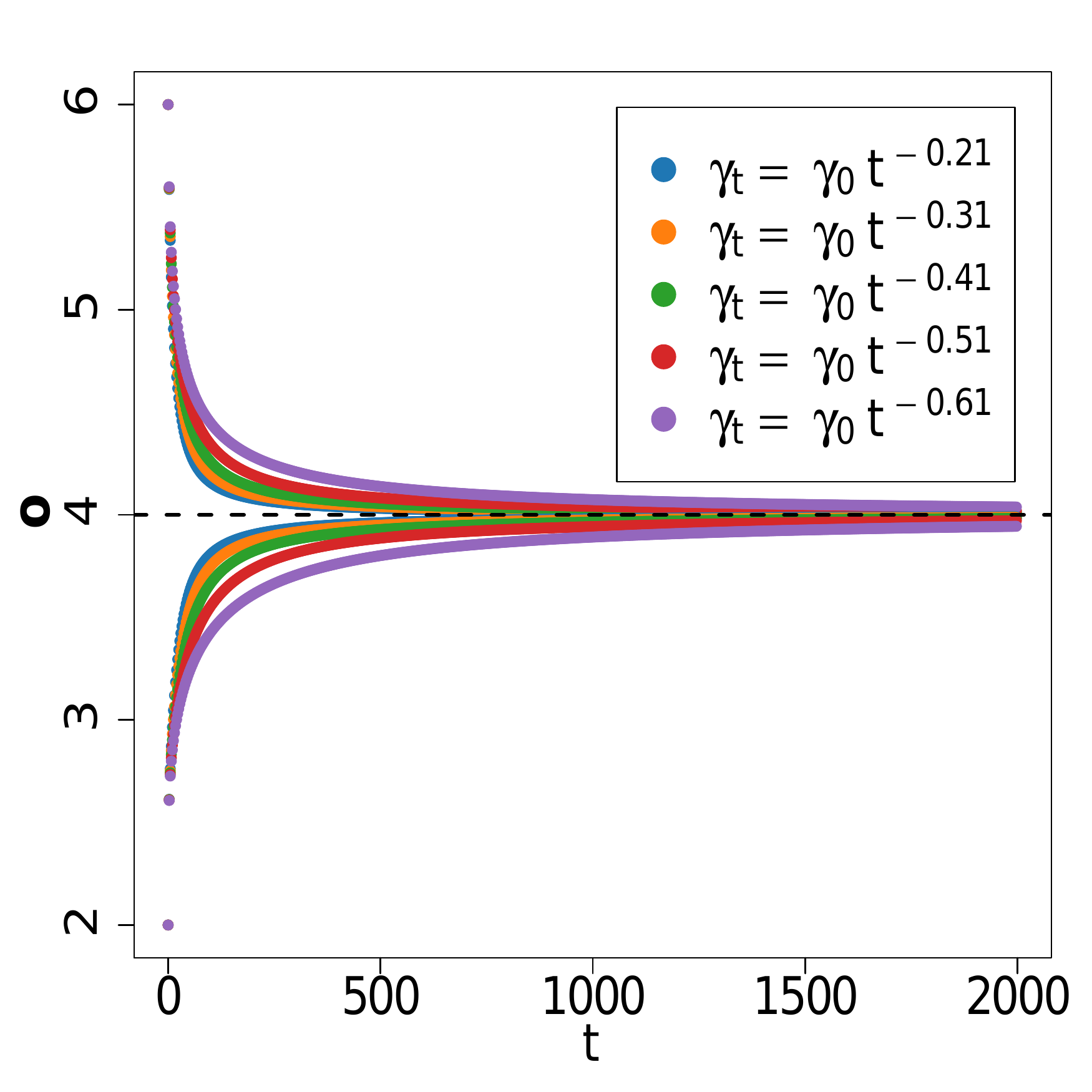}}
  \hspace{6mm}
  \subfloat[$\mathbb{L}_1$ error (log-log plot).]{\label{fig_2b}\includegraphics[trim = 0 0 0 1cm,width=0.42\textwidth,clip]{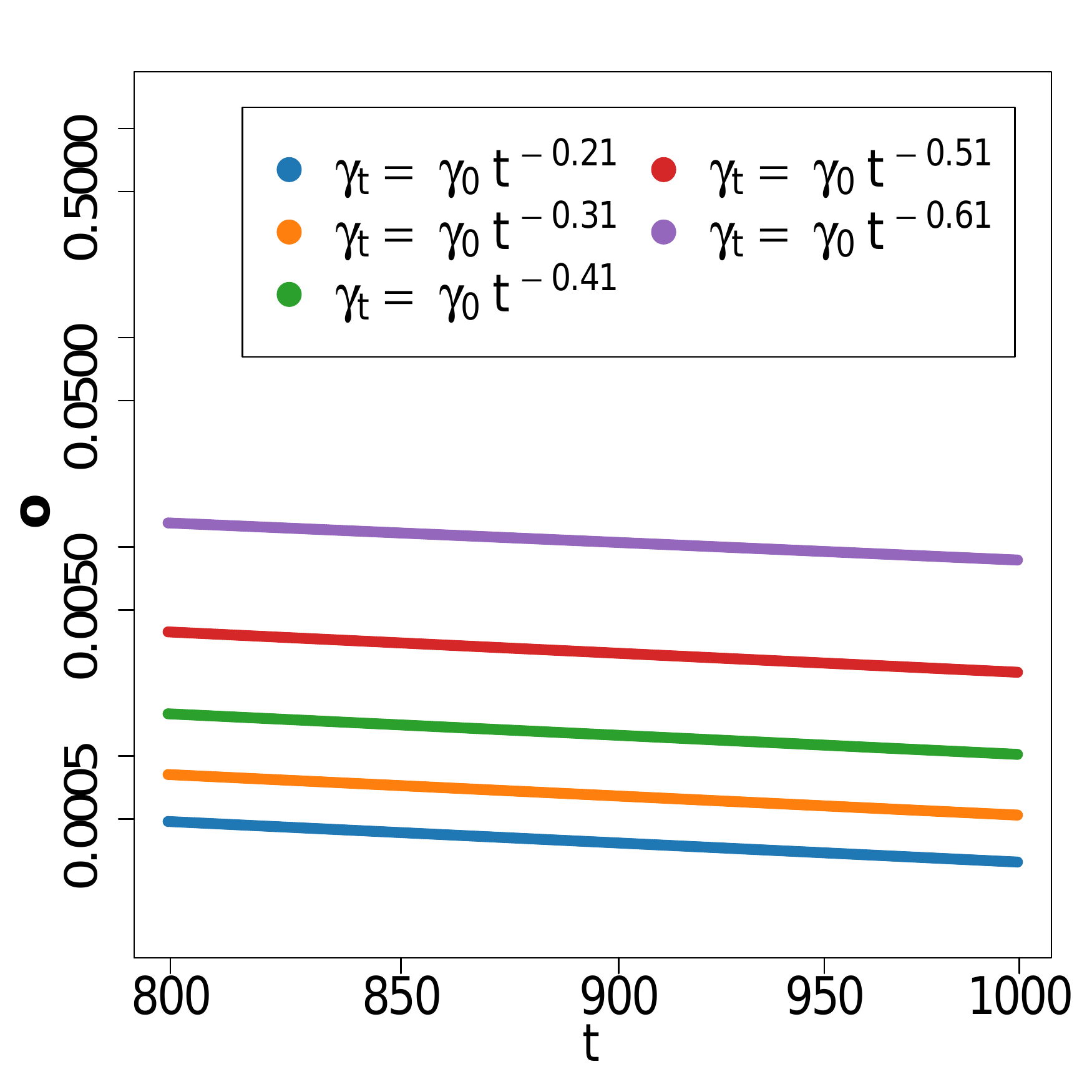}}
  \vspace{-2mm}
\caption{The sequence of averaged optimal sensor placements for the partially observed Bene\v{s} SDE, for several choices of the learning rates $\{\gamma_{\theta}(t)\}_{t\geq 0}$, $\{\gamma_{\boldsymbol{o}}(t)\}_{t\geq 1}$.}
\label{fig2}
\end{figure}

We conclude this section by investigating the performance of the stochastic gradient descent algorithm under the assumption that the true model parameters and the optimal sensor placements are no longer static, but now change in time. This is a scenario of particular practical interest. In this case, we must specify constant learning rates for both the parameter estimates and the sensor placements. While this violates the learning rate condition in Proposition \ref{theorem2}, it is a standard choice when the model parameters are dynamic (e.g., \cite{Soderstrom1983}). In particular, although there is no longer any guarantee that that the algorithm iterates will converge to the stationary points of the two objective functions, they can be expected to oscillate around these points, with amplitude proportional to the learning rate. The performance of the algorithm is shown in Figure \ref{fig3}. As anticipated, the online parameter estimates (sensor placements) are able to track changes in the true model parameter (optimal sensor placement) in real time. It is worth noting that, while here we have considered the case in which the model parameters change discontinuously in time, we obtain similar results when the model parameters change continuously in time. 

\begin{figure}[tbhp]
\centering
  \subfloat[Online parameter estimates.]{\label{fig_3a}\includegraphics[trim = 0 0 0 1cm,width=0.38\textwidth,clip]{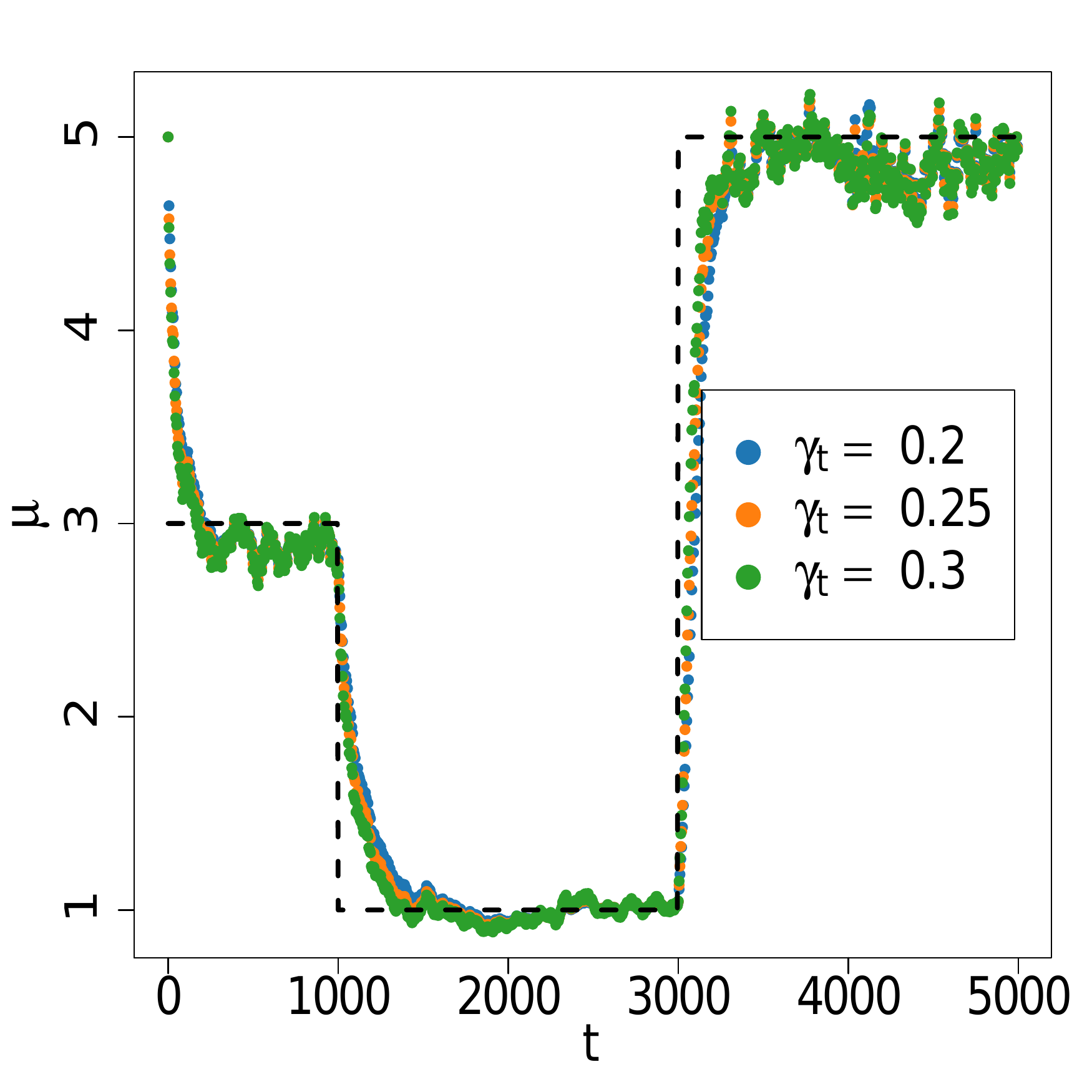}}
  \hspace{6mm}
  \subfloat[Optimal sensor placements.]{\label{fig_3b}\includegraphics[trim = 0 0 0 1cm,width=0.38\textwidth,clip]{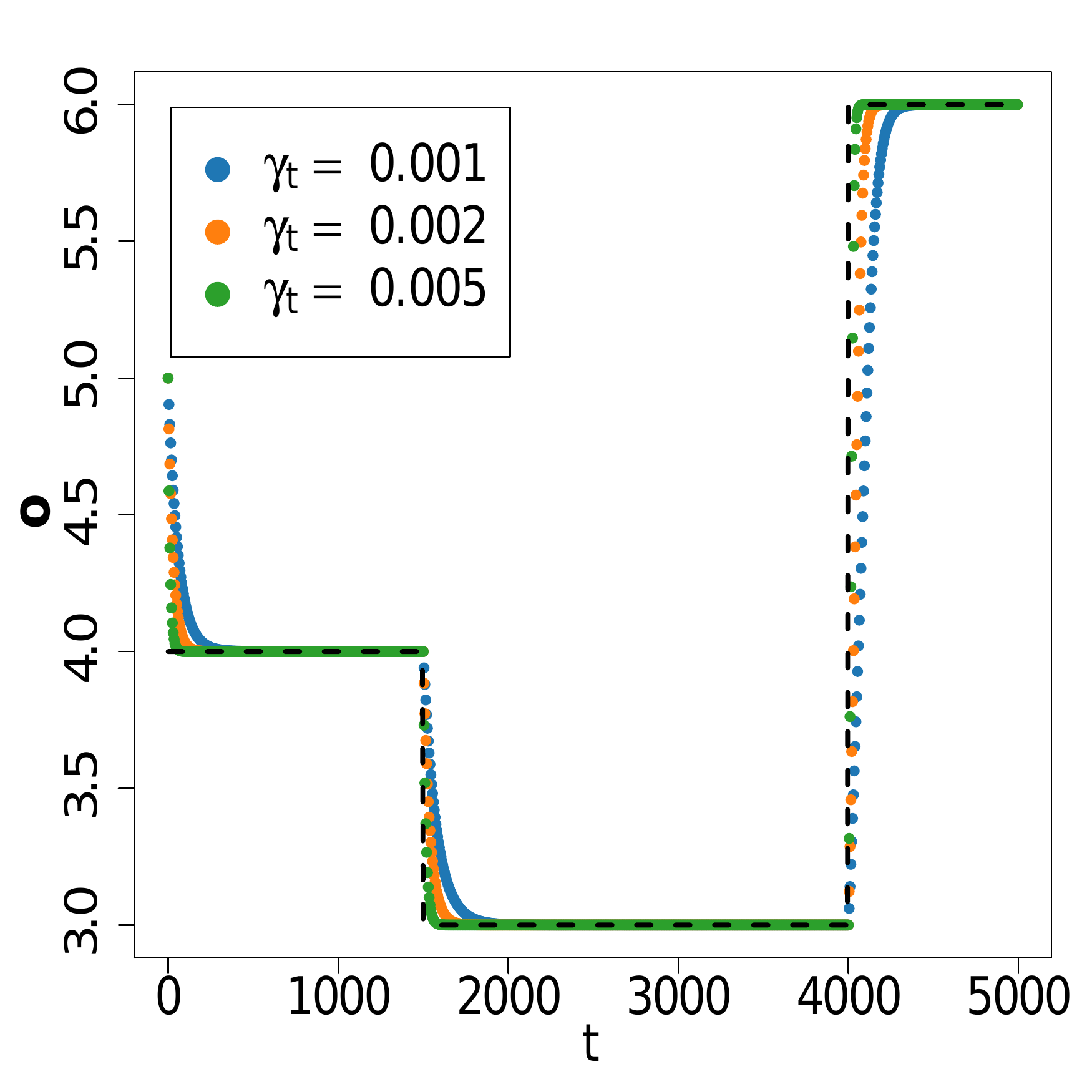}}
  \vspace{-2mm}
\caption{Sequence of online parameter estimates \& optimal sensor placements for the partially observed Bene\v{s} SDE, in the case of a time-varying parameter and optimal sensor location.}
\label{fig3}
\end{figure}

\subsection{Stochastic Advection-Diffusion Equation} \label{sec:advection}
In this section, we present results for a high-dimensional, partially observed linear diffusion process obtained via a Galerkin discretisation of the stochastic advection-diffusion equation on the two-dimensional unit torus $\mathcal{H} = [0,1]^2$, 
namely, 
\begin{subequations}
\begin{alignat}{2}
\mathrm{d}x(t) &= {A}(\theta,x(t))\mathrm{d}t + \mathrm{d}w(t)~,~~~&&x(0)=x_0\in\mathcal{H}, \label{adv_diff_1} \\
\mathrm{d}y(t) &= C(\boldsymbol{o},x(t))\mathrm{d}t + \mathrm{d}v(t)~,~~~&&y(0)=0,\label{adv_diff_2}
\end{alignat}
\end{subequations}
where ${A}(\theta,\cdot):\mathcal{H} \rightarrow\mathcal{H}$ is the `advection-diffusion operator',  defined according to
\begin{align}
\mathcal{A}(\theta,x)&=-\boldsymbol{\mu}^T\nabla x +\nabla\cdot\Sigma\hspace{1mm}\nabla x -\zeta x,
\end{align}
 with $\boldsymbol{\mu}=(\mu_1,\mu_2)^T$ $\in\mathbb{R}^2$ a drift parameter, $\zeta\in\mathbb{R}_{+}$ a damping parameter, and $\Sigma=\Sigma(\rho_1,\gamma,\alpha)\in\mathbb{R}^{2\times2}$ a diffusion matrix. Following \cite{Sigrist2015},
  we parametrise $\Sigma$ as
\begin{equation}
\Sigma^{-1} = \frac{1}{\rho_1^2}\left(\begin{array}{cc} \cos{\alpha} & \sin{\alpha} \\ -\gamma\sin{\alpha} & \cos{\alpha} \end{array}\right)^T\left(\begin{array}{cc} \cos{\alpha} & \sin{\alpha} \\ -\gamma\sin{\alpha} & \cos{\alpha} \end{array}\right).
\end{equation} 
for a range parameter $\rho_1\in\mathbb{R}_{+}$, an anisotropic amplitude $\gamma\in\mathbb{R}_{+}$, and an anisotropic direction $\alpha\in[0,\frac{\pi}{2}]$. 
Meanwhile, the signal noise process $w = \{w(t)\}_{t\geq 0}$ is a $\mathcal{H}$-valued Wiener process with incremental covariance operator $Q(\theta):\mathcal{H}\rightarrow\mathcal{H}$. As in \cite{Sigrist2015}, 
we assume that this covariance operator is diagonal with respect to the Fourier basis $\{\phi_{\boldsymbol{k}}(\boldsymbol{s})\}_{\boldsymbol{k}\in\mathbb{Z}^2/\{\boldsymbol{0}\}}$, where $\smash{\phi_{\boldsymbol{k}}(\mathbf{s})=\exp(2\pi i \boldsymbol{k}^T\mathbf{s})}$, and defined according to
\begin{equation}
Q(\theta)\phi_{\boldsymbol{k}}(\boldsymbol{s}) = \eta_{\boldsymbol{k}}^2(\theta)\phi_{\boldsymbol{k}}(\boldsymbol{s})~,~~~ \eta^2_{\boldsymbol{k}}(\theta) = \frac{\sigma^2}{(2\pi)^2}\left(\boldsymbol{k}^T\boldsymbol{k}+\frac{1}{\rho_0^2}\right)^{-2},
\end{equation}
where $\sigma_0\in\mathbb{R}_{+}$ is a marginal variance parameter, and $\rho\in\mathbb{R}_{+}$ is a spatial range parameter. This defines a Wiener process with the so-called Mat\'ern covariance function in space (e.g., \cite{Whittle1954}). In the measurement equation, the `observation operator' $C(\boldsymbol{o},\cdot):\mathcal{H}\rightarrow\mathbb{R}^{n_y}$ is defined according to
\begin{align}
C(\boldsymbol{o},x)&= (C_1(\boldsymbol{o},x),\dots,C_{n_y}(\boldsymbol{o},x))^T~~~,~~~C_i(\boldsymbol{o},x) = \frac{\int_{\Omega} K_{\boldsymbol{o}_i}(\boldsymbol{s})x(\boldsymbol{s})\mathrm{d}\boldsymbol{s}}{\int_{\Omega} K_{\boldsymbol{o}_i}(\boldsymbol{s})\mathrm{d}s},
\end{align}
where 
$K_{\boldsymbol{o}_i}(\boldsymbol{s}) = \mathds{1}_{\{|\boldsymbol{o}_i-\boldsymbol{s}|<r\}}(\boldsymbol{s})$, for some fixed radius $r>0$. We thus have $n_y$ independent sensors, with each sensor capable of measuring an average value of the signal process within a fixed range of its current location. Finally, the observation noise $v=\{v(t)\}_{t\geq0}$ is a $\mathbb{R}^{n_y}$-valued Wiener process with incremental covariance $R(\theta)
= \tau^2 I_{n_y}$. 

In summary, we consider a partially observed linear diffusion process, which depends on a nine-dimensional parameter vector $\smash{\theta = (\rho_0,\sigma^2,\zeta,\rho_1,\gamma,\alpha,\mu_x,\mu_y,\tau^2)}$, and a set of $n_y$ two-dimensional sensor locations $\smash{\boldsymbol{o}=\{\boldsymbol{o}_i\}_{i=1}^{n_y}}$. The model of interest is linear, and thus the filtering problem admits an analytic, finite-dimensional solution; namely, the Kalman-Bucy Filter (e.g., \cite{Kalman1961}). 
In particular, assuming a Gaussian initialisation $\smash{x_0\sim N(\hat{x}_0,\hat{\Sigma}_0)}$, the normalised conditional distribution of the latent signal process is Gaussian, and determined uniquely by its mean $\hat{x}(\theta,\boldsymbol{o},t)$ and covariance $\hat{\Sigma}(\theta,\boldsymbol{o},t)$. We should remark that, while the filtering problem for this system is now very well understood, the problem of joint online parameter estimation and optimal sensor placement still poses a significant challenge. Indeed, although the model may be linear, the dependence of the log-likelihood on the model parameters, and of the objective function on the sensor placements, is highly non-linear.

The performance of two-timescale stochastic gradient descent algorithm is visualised in Figure \ref{fig4}. As expected, all of the parameter estimates converge to within a small neighbourhood of their true values, and all of the sensors converge to one of the target locations. In this simulation, we assume that the true model parameters and the initial parameter estimates are given, respectively, by
\begin{subequations}
\begin{align}
{{\theta}}^{*} &= (\rho_0 = 0.50,\sigma^2 = 0.20,\zeta=0.50,\rho_1=0.10,\gamma=2.00,\\
&~~~~~~\alpha=\tfrac{\pi}{4},\mu_x=0.30,\mu_y=-0.30,\tau^2 = 0.01), \label{true_model_param} \nonumber \\[1mm]
{{\theta}}_0 &= ({\rho}_{0} = 0.25,{\sigma}^2 = 0.50,{\zeta}=0.20,{\rho}_{1}=0.20,{\gamma}=1.50,\\
&~~~~~~{\alpha}=\tfrac{\pi}{3},{\mu}_{x}=0.10,{\mu}_{y}=-0.15,{\tau}^2 = 0.10).  \nonumber
\end{align}
\end{subequations}
We also assume that we have $n_y=8$ sensors, and that our objective is to obtain the sensor placement which minimises the uncertainty of the state estimate at a discrete set of 8 spatial locations. In particular, we assume that the target sensor locations and the initial sensor locations are given, respectively, by 
\begin{subequations}
\begin{align}
{\boldsymbol{o}}_{*} &= \frac{1}{12}\left\{\begin{pmatrix} 0.00 \\ 7.00\end{pmatrix},\begin{pmatrix} 6.00 \\ 8.00\end{pmatrix},\begin{pmatrix} 4.00 \\ 4.00 \end{pmatrix}, \begin{pmatrix} 9.00 \\ 6.00 \end{pmatrix}, \begin{pmatrix} 1.00 \\ 1.00 \end{pmatrix}, \begin{pmatrix} 7.00 \\ 10.0 \end{pmatrix}, \begin{pmatrix} 10.0 \\ 11.0\end{pmatrix}, \begin{pmatrix} 3.00 \\ 10.0 \end{pmatrix}\right\},  \nonumber \\[2mm]
{\boldsymbol{o}}_0 &= \frac{1}{12}\left\{\begin{pmatrix} 10.1 \\ 7.80\end{pmatrix},\begin{pmatrix} 4.10 \\ 6.01\end{pmatrix},\begin{pmatrix} 5.20 \\ 3.75 \end{pmatrix}, \begin{pmatrix} 7.20 \\ 4.02 \end{pmatrix}, \begin{pmatrix} 3.20 \\ 3.10 \end{pmatrix}, \begin{pmatrix} 6.10 \\ 2.10 \end{pmatrix}, \begin{pmatrix} 1.01 \\ 2.80\end{pmatrix}, \begin{pmatrix} 3.00 \\ 1.00 \end{pmatrix}\right\}.  \nonumber
\end{align}
\end{subequations}
It remains to specify the learning rates $\{\gamma^{i}_{\theta}(t)\}_{t\geq 0}^{i=1,\dots,9}$ and $\{\gamma^{j}_{\boldsymbol{o}}(t)\}^{j=1,\dots,8}_{t\geq 0}$, where the indices $i,j$ now make explicit the fact that the step sizes are permitted to vary between parameters, and between sensors. In this case, we set $\smash{\gamma_{\theta}^{i}(t) = \gamma^{i}_{\theta,0}t^{-\varepsilon^i_{\theta}}}$ and $\smash{\gamma_{\boldsymbol{o}}^{j}(t) = \gamma^{j}_{\theta,0}t^{-\varepsilon^j_{\boldsymbol{o}}}}$, where $\gamma^{i}_{\theta,0},\gamma^{j}_{\boldsymbol{o},0}>0$ and $\smash{0.5<\varepsilon_{\boldsymbol{o}}^j<\varepsilon_{\theta}^{i}\leq1}$ for all $i=1,\dots,9$ and $j=1,\dots,8$, with the specific values of $\smash{\gamma^{i}_{\theta,0}}$, $\smash{\varepsilon_{\theta}^i}$, $\smash{\gamma^{j}_{\boldsymbol{o},0}}$, and $\smash{\varepsilon_{\boldsymbol{o}}^{j}}$ tuned individually. We defer further details of our implementation, and additional numerical results, to \cite{Sharrock2020}.

\begin{figure}[tbhp]
\centering
  \subfloat[Sequence of online parameter estimates.]{\label{fig_4a}\includegraphics[width=0.49\textwidth]{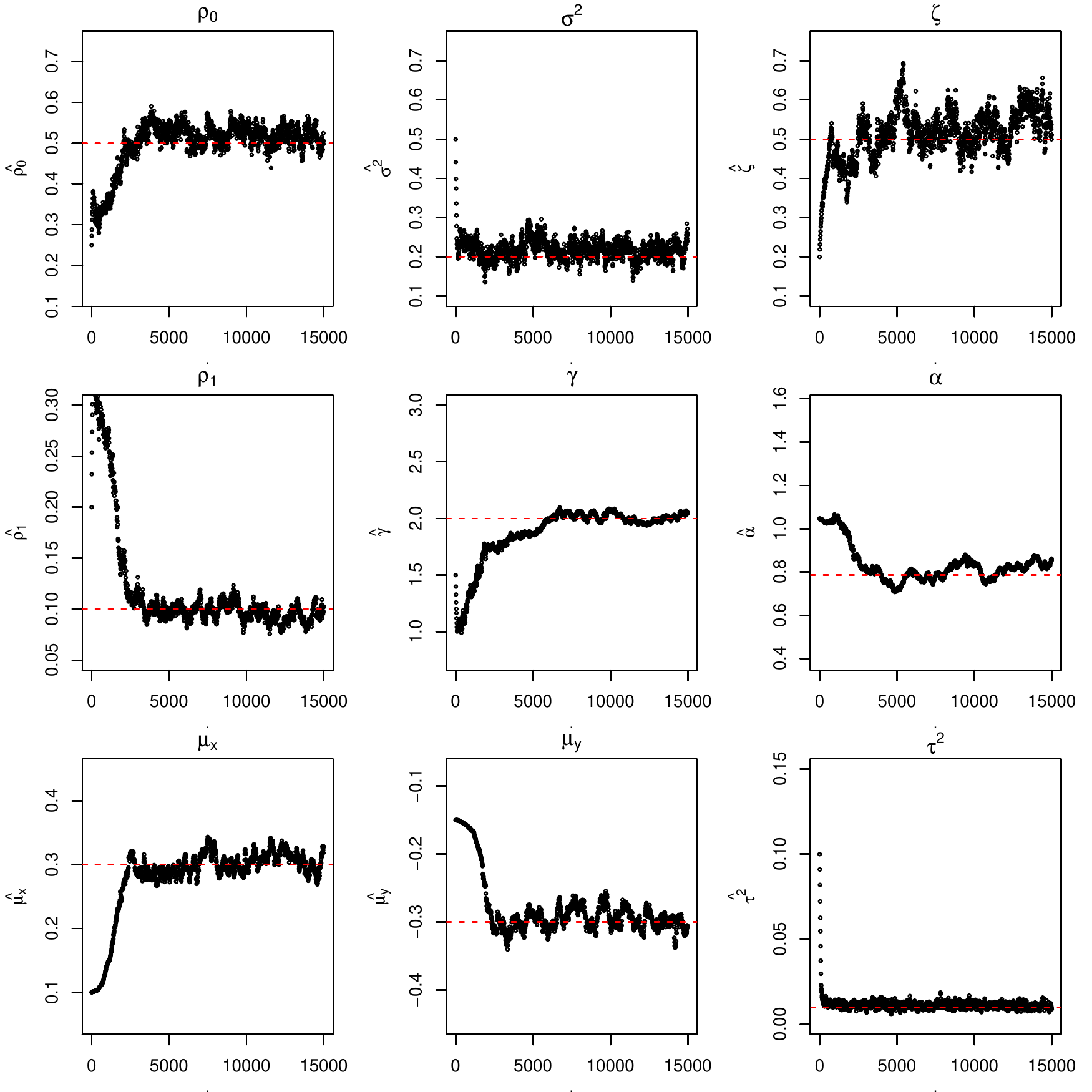}}
  \hspace{1mm}
  \subfloat[Sequence of optimal sensor placements.]{\label{fig_4b}\includegraphics[width=0.49\textwidth]{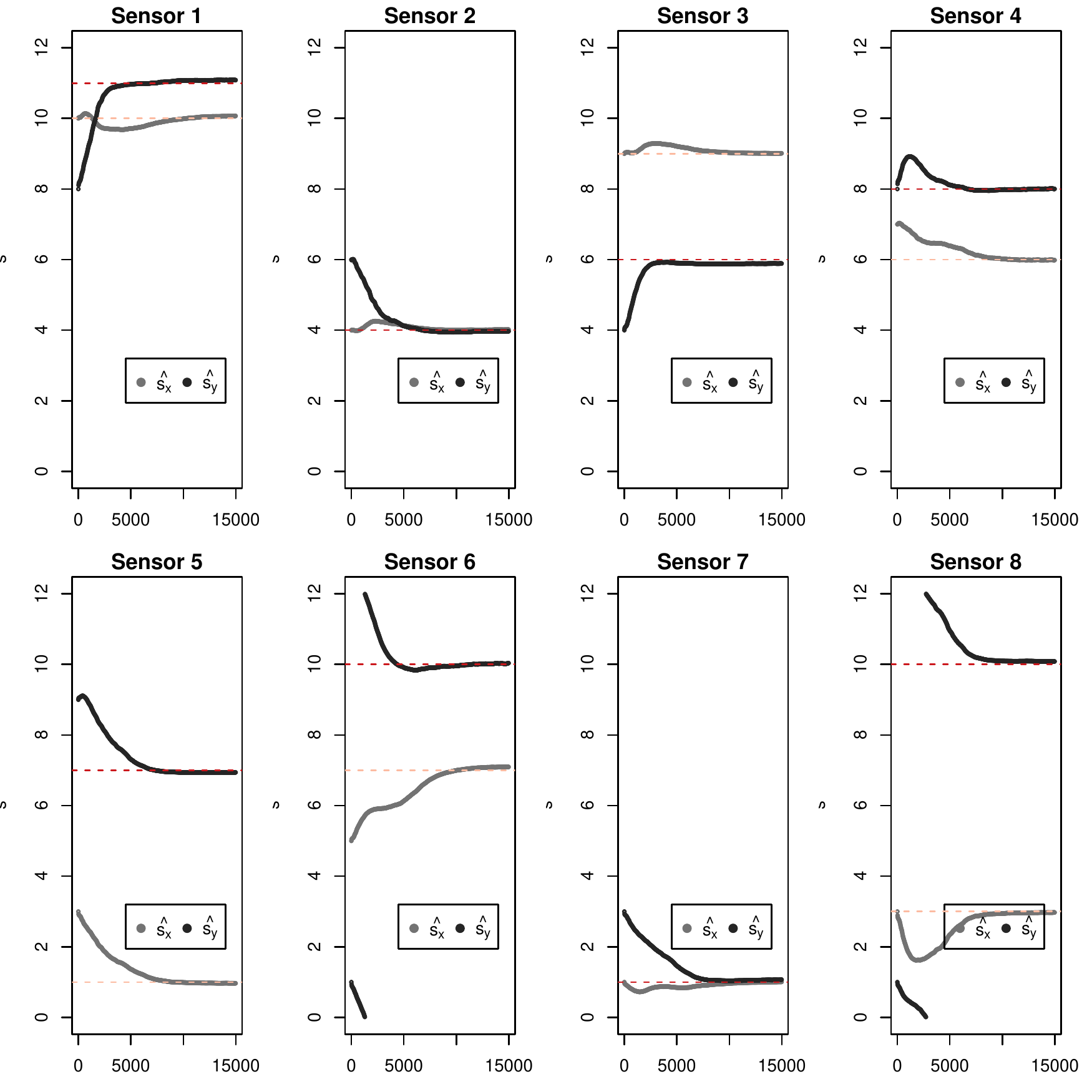}}
\caption{Sequence of online parameter estimates \& optimal sensor placements for the partially observed stochastic-advection diffusion equation.}
\label{fig4}
\end{figure}

\section{Conclusions}
\label{sec:conclusions}
In this paper, we have analysed the almost sure convergence of two-timescale stochastic gradient algorithms in continuous time, under general noise and stability conditions. Moreover, we have demonstrated in detail how such algorithms can be applied to the problem of joint online parameter estimation and optimal sensor placement in continuous-time state space models. Although we focus on this specific application, it is important to emphasise that the proposed methodology is applicable to any problem involving two inter-dependent objective functions, either or both of which may depend on an ergodic diffusion process.

We conclude with some remarks regarding some possible directions for future work. Firstly, there are a number of extensions to Theorems \ref{theorem1} and \ref{theorem1a} which may be of theoretical or practical interest. These include relaxing the assumption that the algorithm iterates are continuous, and thus considering variations of algorithm \eqref{eq1} - \eqref{eq2} in which 
%
%
$\{\xi_i(t)\}_{t\geq 0}$, $i=1,2$, are now general semi-martingales. This algorithm can be regarded, in some sense, as as a two-timescale extension of the Robbins-Monro type semimartingale SDEs studied in, for example, \cite{Lazrieva2008,Valkeila2000}.
Obtaining asymptotic results under this somewhat more general framework is of considerable interest, as such results would apply to two-timescale stochastic gradient descent schemes in both discrete time and continuous time. Other possible extensions to our results include an analysis of the asymptotic convergence rate as discussed in Section \ref{subsec:main1} (e.g., \cite{Konda2004,Mokkadem2006}), or the `lock-in probability' (e.g., \cite{Borkar2018,Dalal2020,Xu2019}). 

Another open problem is to obtain sufficient conditions for some of the assumptions required for Theorems \ref{theorem1} and \ref{theorem1a}. Among these is the somewhat restrictive assumption that the algorithm iterates remain almost surely bounded. While this assumption is necessary in order to prove almost sure convergence, it is generally far from automatic, and not very straightforward to establish, particularly in the two-timescale case. 
Currently, the most promising approaches to this task appear to be extensions of of the randomly varying truncations method in \cite{Chen1987}, the stopping-times approach in \cite{Bertsekas2000,Sirignano2017a}, or the recent results in \cite{Mertikopoulos2020}, to the two-timescale setting.  

Regarding the algorithm considered in Theorem \ref{theorem2}, the main open problem is to obtain conditions on the generative model (i.e., the partially observed diffusion process) which are easy to verify, sufficient for convergence, and not overly restrictive (see the discussion in Section \ref{sec_joint_RML_OSP}). This problem is particularly challenging when the filter is approximate. In this case, even if the latent signal is ergodic, there is no guarantee that the filter is ergodic, let alone the tangent filter. We leave this problem, as well as further numerical experiments investigating the performance of this algorithm for more elaborate continuous time particle filters (e.g., \cite[Chapter 9]{Bain2009}), as future work.


\section*{Acknowledgments}
\label{sec:acknowledgments}
We are very grateful to D. Crisan and A. Forbes for many helpful discussions and suggestions. The first author was funded by the EPRSC CDT in the Mathematics of Planet Earth (grant number EP/L016613/1) and the National Physical Laboratory. The second author was partially funded by JPMorgan Chase \& Co. under a J.P. Morgan A.I. Research Award (2019). 

\appendix

\section{Proof of Theorem \ref{theorem1}} 
\label{sec:proof1}
In this Appendix, we provide a proof of Theorem \ref{theorem1}. Our proof follows the approach in \cite[Chapter~6]{Borkar2008}, adapted appropriately to the continuous-time setting.

\subsection{Additional Notation}

We will require the following additional notation. Firstly, in a slight abuse of notation, we will write $(x_1, x_2)$ to denote the concatenation of $x_1\in\mathbb{R}^{d_1}$ and $x_2\in\mathbb{R}^{d_2}$. We will also write $\{q_i(t)\}_{t\geq 0}$, $\{p_i(t)\}_{t\geq 0}$, $i=1,2$, to denote the processes
\begin{subequations}
\begin{align}
q_i(t) &= \int_0^t \gamma_i(s)\mathrm{d}s \\
p_i(t)&=\left\{s: \int_0^s\gamma_i(v)\mathrm{d}v = t\right\}= q_i^{-1}(t).
\end{align}
\end{subequations}
We then define the time-scaled processes $\{\alpha_{\gamma_i}(t)\}_{t\geq 0}$, $\{\beta_{\gamma_i}(t)\}_{t\geq 0}$, $i=1,2$, by
\begin{subequations}
\begin{align}
\alpha_{\gamma_i}(t) &= \alpha(p_i(t)), \\
\beta_{\gamma_i}(t) &= \beta(p_i(t)).
\end{align}
\end{subequations}

\subsection{The Fast Timescale}
\subsubsection{Additional Notation}
We will write $\{\bar{\alpha}(t)\}_{t\geq 0}$, $\{\bar{\beta}(t)\}_{t\geq 0}$ to denote the solutions of the coupled ordinary differential equations
\begin{subequations}
\begin{alignat}{2}
\dot{\bar{\alpha}}(t) &= 0, \label{eqA1} \\
\dot{\bar{\beta}}(t)&= -\nabla_{\beta}g(\bar{\alpha}(t),\bar{\beta}(t)). \label{eqA2}
\end{alignat}
\end{subequations}
We can then define $\{\bar{\alpha}^{(s)}(t)\}_{0\leq s\leq t}$, $\{\bar{\beta}^{(s)}(t)\}_{0\leq s \leq t}$, as the unique solutions of equations (\ref{eqA1})-(\ref{eqA2}) which `start at $s$', and coincide with the time-scaled processes $\{\alpha_{\gamma_2}(t)\}_{t\geq 0}$, $\{\beta_{\gamma_2}(t)\}_{t\geq 0}$, at $s$. That is, 
\begin{subequations}
\begin{alignat}{3}
\dot{\bar{\alpha}}^{(s)}(t) &= 0~,~~&&\bar{\alpha}^{(s)}(s)= {\alpha}_{\gamma_2}(s),~~~&&t\geq s,  \\
\dot{\bar{\beta}}^{(s)}(t)&= -\nabla_{\beta}g(\bar{\alpha}^s(t),\bar{\beta}^{(s)}(t))~,~~&&\bar{\beta}^{(s)}(s)= {\beta}_{\gamma_2}(s),~~~&&t\geq s.
\end{alignat}
\end{subequations}
We can similarly define $\{\bar{\alpha}^{[s]}(t)\}_{0\leq t\leq s}$, $\{\bar{\beta}^{[s]}(t)\}_{0\leq t\leq s}$, as the unique solutions of equations (\ref{eqA1})-(\ref{eqA2}) which `end at $s$', and coincide with the time-scaled processes $\{\alpha_2(t)\}_{t\geq 0}$, $\{\beta_2(t)\}_{t\geq 0}$, at $s$. That is, 
\begin{subequations}
\begin{alignat}{3}
\dot{\bar{\alpha}}^{[s]}(t) &= 0~,~~&&\bar{\alpha}^{[s]}(s)= \alpha_{\gamma_2}(s),~~~&&t\leq s,  \\
\dot{\bar{\beta}}^{[s]}(t)&= -\nabla_{\beta}g(\bar{\alpha}^{[s]}(t),\bar{\beta}^{[s]}(t))~,~~&&\bar{\beta}^{[s]}(s)= \beta_{\gamma_2}(s),~~~&&t\leq s. \\ \nonumber
\end{alignat}
\end{subequations}

\subsubsection{Proof of Convergence} 
We first establish, using the processes just defined, that the time-scaled process $(\alpha_{\gamma_2}(t),\beta_{\gamma_2}(t))$ is an asymptotic pseudo-trajectory (APT) of the flow induced by the coupled ODEs \eqref{eqA1} - \eqref{eqA2}.   Broadly speaking, this means that $(\alpha_{\gamma_2}(t),\beta_{\gamma_2}(t))$ tracks the flow induced by these coupled ODEs with arbitrary accuracy over windows of arbitrary length as time goes to infinity. This provides a notion of ``asymptotic closeness'' between the paths  generated by Algorithm \eqref{eq1} - \eqref{eq2}, and the flow of the coupled ODEs. The motivation for this comparison is that, provided the trajectories generated by Algorithm \eqref{eq1} - \eqref{eq2} are ``good enough'' approximations to the solutions of the coupled ODEs, one can expect that the two sets of equations will enjoy similar convergence properties. For further details, we refer to \cite{Benaim1996,Benaim1999}.

\begin{lemma} \label{lemmaA1}
Assume that Assumptions \ref{assumption1}-\ref{assumption4} hold. Then, for all $T>0$,
\begin{subequations}
\begin{align}
\lim_{s\rightarrow\infty} \sup_{t\in[s,s+T]}\left|\left|\begin{pmatrix} \alpha_{\gamma_2}(t) \\ \beta_{\gamma_2}(t) \end{pmatrix} - \begin{pmatrix} \bar{\alpha}^{(s)}(t) \\ \bar{\beta}^{(s)}(t) \end{pmatrix} \right|\right| &= 0~,~~\text{ a.s.} \\
\lim_{s\rightarrow\infty} \sup_{t\in[s-T,s]}\left|\left|\begin{pmatrix} \alpha_{\gamma_2}(t) \\ \beta_{\gamma_2}(t) \end{pmatrix} - \begin{pmatrix} \bar{\alpha}^{[s]}(t) \\ \bar{\beta}^{[s]}(t) \end{pmatrix} \right|\right| &= 0~,~~\text{ a.s.}
\end{align}
\end{subequations}
\end{lemma}

\begin{proof}
We will prove only the first part of this Lemma, as the method for proving the second part is entirely analogous. We will begin by considering $\{\alpha(t)\}_{0\leq s\leq t}$. By definition, we have 
\begin{align}
\alpha(t)&=\alpha(s)  -\int_s^t\gamma_{1}(u)\nabla_{\alpha} f(\alpha(u),\beta(u))\mathrm{d}u -\int_s^t\gamma_1(u) \mathrm{d}\xi_1(u) \\
&= \alpha(s)  -\int_s^t \frac{\gamma_{1}(u)}{\gamma_2(u)}\gamma_2(u)\nabla_{\alpha} f(\alpha(u),\beta(u))\mathrm{d}u -\int_s^t\gamma_1(u) \mathrm{d}\xi_1(u)
\end{align}  
It follows immediately from the definition of $\{\alpha_{\gamma_2}(t)\}_{0\leq s\leq t}$ that 
\begin{align}
\alpha_{\gamma_2}(t) &= \alpha_{\gamma_2}(s)  -\int_{p_2(s)}^{p_2(t)} \frac{\gamma_{1}(u)}{\gamma_2(u)}\gamma_2(u)\nabla_{\alpha} f(\alpha(u),\beta(u))\mathrm{d}u -\int_{p_2(s)}^{p_2(t)}\gamma_1(u) \mathrm{d}\xi_1(u) \\
&=\alpha_{\gamma_2}(s) -\int_{s}^{t} \frac{\gamma_{1}(p_2(u))}{\gamma_2(p_2(u))}\nabla_{\alpha} f(\alpha_{\gamma_2}(u),\beta_{\gamma_2}(u))\mathrm{d}u -\int_{p_2(s)}^{p_2(t)}\gamma_1(u) \mathrm{d}\xi_1(u). \label{A17}
\end{align}  
We also have, making use of the ODE for $\{\bar{\alpha}^{(s)}(t)\}_{0\leq s\leq t}$, that
\begin{align}
\bar{\alpha}^{(s)}(t) = \alpha_{\gamma_2}(s). \label{A18}
\end{align}
It follows straightforwardly from equations (\ref{A17}), (\ref{A18}) that 
\begin{align}
||\alpha_{\gamma_2}(t)-\bar{\alpha}^{(s)}(t)|| = \left|\left|-\int_{s}^{t} \frac{\gamma_{1}(p_2(u))}{\gamma_2(p_2(u))}\nabla_{\alpha} f(\alpha_{\gamma_2}(u),\beta_{\gamma_2}(u))\mathrm{d}u -\int_{p_2(s)}^{p_2(t)}\gamma_1(u) \mathrm{d}\xi_1(u)\right|\right| \\
\leq \underbrace{\left|\left|\int_{s}^{t} \frac{\gamma_{1}(p_2(u))}{\gamma_2(p_2(u))}\nabla_{\alpha} f(\alpha_{\gamma_2}(u),\beta_{\gamma_2}(u))\mathrm{d}u\right|\right|}_{\Omega_{1,\alpha}(s,t)}+\underbrace{\left|\left|\int_{p_2(s)}^{p_2(t)}\gamma_1(u) \mathrm{d}\xi_1(u)\right|\right|}_{\Omega_{2,\alpha}(s,t)} \label{A20}
\end{align}
For the first term, by Assumptions \ref{assumption2} and \ref{assumption4}, which together imply the boundedness of $||\nabla_{\alpha} f(\cdot,\cdot)||$, we have that for all $T>0$, 
\begin{align}
\sup_{t\in[s,s+T]}\Omega_{1,\alpha}(s,t) &= \sup_{t\in[s,s+T]}\left|\left|\int_{s}^{t} \frac{\gamma_{1}(p_2(u))}{\gamma_2(p_2(u))}\nabla_{\alpha} f(\alpha_{\gamma_2}(u),\beta_{\gamma_2}(u))\mathrm{d}u\right|\right| \\
&\leq  \sup_{t\in[s,s+T]} \left|\left|\nabla_{\alpha} f(\alpha_{\gamma_2}(t),\beta_{\gamma_2}(t))\right|\right|   \int_{s}^{s+T} \frac{\gamma_{1}(p_2(u))}{\gamma_2(p_2(u))}\mathrm{d}u \\
&\leq K  \int_{s}^{s+T} \frac{\gamma_{1}(p_2(u))}{\gamma_2(p_2(u))}\mathrm{d}u \\
&\leq KT  \sup_{t\in[s,s+T]}\frac{\gamma_{1}(p_2(t))}{\gamma_2(p_2(t))}.
\end{align}
It follows immediately, using also Assumption \ref{assumption1}, that, for all $T>0$, 
\begin{equation}
\lim_{s\rightarrow\infty}\sup_{t\in[s,s+T]}\Omega_{1,\alpha}(s,t)=0~,~~\text{a.s.} \label{A25}
\end{equation}
For the second term, using the definition of $\{p(t)\}_{t\geq 0}$, we have that,  for sufficiently large $s$,  
\begin{align}
\sup_{t\in[s,s+T]}\Omega_{2,\alpha}(s,t) &=\sup_{t\in[s,s+T]}\left|\left|\int_{p_2(s)}^{p_2(t)}\gamma_1(u) \mathrm{d}\xi_1(u)\right|\right| \label{A26} \\
&\leq \sup_{t\in[s,s+T]}\left|\left|\int_{s}^{p_2(t)}\gamma_1(u) \mathrm{d}\xi_1(u)\right|\right| \\
&=\sup_{t\in[s,s+\tau]}\left|\left|\int_{s}^{t}\gamma_1(u) \mathrm{d}\xi_1(u)\right|\right| \label{A29}
\end{align}
 where, in the second line, we have used the fact that $s\leq p_2(s)$ for sufficiently large $s$,  and in the final line, we have defined $\tau=\tau(T)=p_2(T)$. It then follows directly from the first part of Assumption \ref{assumption3} that, for all $T>0$, 
\begin{equation}
\lim_{s\rightarrow\infty}\sup_{t\in[s,s+T]}\Omega_{2,\alpha}(s,t)  = 0 ~,~~\text{a.s.} \label{A30}
\end{equation}
We will now consider $\{\beta(t)\}_{0\leq s\leq t}$. By definition, we have that 
\begin{align}
\beta(t) = \beta(s) - \int_{s}^t \gamma_{2}(u)\nabla_{\beta} g(\alpha(u),\beta(u))\mathrm{d}u -\int_{s}^t \mathrm{d}\xi_2(u)
\end{align}
It follows immediately from the definition of $\{\beta_{\gamma_2}(t)\}_{0\leq s\leq t}$ that 
\begin{align}
\beta_{\gamma_2}(t) &= \beta_{\gamma_2}(s) - \int_{p_2(s)}^{p_2(t)} \gamma_{2}(u)\nabla_{\beta} g(\alpha(u),\beta(u))\mathrm{d}u -\int_{p_2(s)}^{p_2(t)} \gamma_{2}(u)\mathrm{d}\xi_2(u) \\
&= \beta_{\gamma_2}(s) - \int_{s}^{t} \nabla_{\beta} g(\alpha_{\gamma_2}(u),\beta_{\gamma_2}(u))\mathrm{d}u -\int_{p_2(s)}^{p_2(t)}\gamma_{2}(u) \mathrm{d}\xi_2(u).  \label{A33}
\end{align}
We also have, now making use of the ODE for $\{\bar{\beta}^{(s)}(t)\}_{0\leq s\leq t}$, that
\begin{align}
\bar{\beta}^{(s)}(t) = \beta_{\gamma_2}(s)-\int_s^t \nabla_{\beta}g(\bar{\alpha}^{(s)}(u),\bar{\beta}^{(s)}(u))\mathrm{d}u \label{A34}
\end{align}
It follows straightforwardly from equations (\ref{A33}), (\ref{A34}) that
\begin{align}
||\beta_{\gamma_2}(t)-\bar{\beta}^{(s)}(t)|| &= \left|\left| - \int_{s}^{t} \left[\nabla_{\beta} g(\alpha_{\gamma_2}(u),\beta_{\gamma_2}(u)) - \nabla_{\beta}g(\bar{\alpha}^{(s)}(u),\bar{\beta}^{(s)}(u))\right] \mathrm{d}u -\int_{p_2(s)}^{p_2(t)}\gamma_{2}(u) \mathrm{d}\xi_2(u) \right|\right| \hspace{-27mm}  \\
&\leq \underbrace{\left|\left|\int_{p_2(s)}^{p_2(t)}\gamma_{2}(u) \mathrm{d}\xi_2(u) \right|\right|}_{\Omega_{1,\beta}(s,t)}+ \underbrace{\left|\left| \int_{s}^{t} \left[\nabla_{\beta} g(\alpha_{\gamma_2}(u),\beta_{\gamma_2}(u)) - \nabla_{\beta}g(\bar{\alpha}^{(s)}(u),\bar{\beta}^{(s)}(u))\right] \mathrm{d}u \right|\right|}_{\Omega_{2,\beta}(s,t)} \hspace{-27mm} \label{A35} 
\end{align} 
For the first term, using the second part of Assumption \ref{assumption3}, and arguing as in equations (\ref{A26})-(\ref{A29}), we have that, for all $T>0$,
\begin{equation}
\lim_{s\rightarrow\infty}\sup_{t\in[s,s+T]}\Omega_{1,\beta}(s,t)  = 0~,~~\text{a.s.} \label{A40}
\end{equation}
For the second term, using elementary properties of the Euclidean norm, and Assumption \ref{assumption2}  (i.e., Lipschitz continuity of $\nabla_{\alpha}g(\cdot,\cdot)$), we have that, for all $T>0$,
\begin{align}
\Omega_{2,\beta}(s,t) &= \left|\left| \int_{s}^{t} \left[\nabla_{\beta} g(\alpha_{\gamma_2}(u),\beta_{\gamma_2}(u)) - \nabla_{\beta}g(\bar{\alpha}^{(s)}(u),\bar{\beta}^{(s)}(u))\right] \mathrm{d}u \right|\right| \\
&\leq\int_{s}^{t} \left|\left|\nabla_{\beta} g(\alpha_{\gamma_2}(u),\beta_{\gamma_2}(u)) - \nabla_{\beta}g(\bar{\alpha}^{(s)}(u),\bar{\beta}^{(s)}(u))\right|\right| \mathrm{d}u \\
&\leq \int_{s}^{t} L_{\beta} \left|\left| \begin{pmatrix} \alpha_{\gamma_2}(u)- \bar{\alpha}^{(s)}(u) \\ \beta_{\gamma_2}(u)-\bar{\beta}^{(s)}(u)\end{pmatrix} \right| \right|\mathrm{d}u
\end{align}
It remains to observe that, combining inequalities (\ref{A20}) and (\ref{A35}), and using Gr\"onwall's Inequality, we have
\begin{align}
\left|\left|\begin{pmatrix} \alpha_{\gamma_2}(t) \\ \beta_{\gamma_2}(t) \end{pmatrix} - \begin{pmatrix} \bar{\alpha}^{(s)}(t) \\ \bar{\beta}^{(s)}(t) \end{pmatrix} \right|\right|&\leq ||\alpha_{\gamma_2}(t)-\bar{\alpha}^{(s)}(t)|| + ||\beta_{\gamma_2}(t)-\bar{\beta}^{(s)}(t)|| \\[2mm]
&\leq \underbrace{\Omega_{1,\alpha}(s,t) + \Omega_{2,\alpha}(s,t) + \Omega_{1,\beta}(s,t)}_{\Omega(s,t)} + \Omega_{2,\beta}(s,t) \\
& = \Omega(s,t) + \int_{s}^{t} L_{\beta} \left|\left| \begin{pmatrix} \alpha_{\gamma_2}(u)- \bar{\alpha}^{(s)}(u) \\ \beta_{\gamma_2}(u)-\bar{\beta}^{(s)}(u)\end{pmatrix} \right| \right|\mathrm{d}u \\[2mm]
& \leq \Omega(s,t) \exp\left[\int_s^t L_{\beta}\mathrm{d}u\right] \\[2mm]
& = \Omega(s,t)\exp\left[L_{\beta}(t-s)\right], \label{A42}
\end{align}
where, from (\ref{A25}), (\ref{A30}) and (\ref{A40}), we have that, for all $T>0$, 
\begin{equation}
\lim_{s\rightarrow\infty}\sup_{t\in[s,s+T]}\Omega(s,t) 
= 0~,~~\text{a.s.} \label{A43}
\end{equation}
It follows immediately from (\ref{A42}) and (\ref{A43}) that, for all $T>0$, 
\begin{align}
\lim_{s\rightarrow\infty} \sup_{t\in[s,s+T]}\left|\left|\begin{pmatrix} \alpha_{\gamma_2}(t) \\ \beta_{\gamma_2}(t) \end{pmatrix} - \begin{pmatrix} \bar{\alpha}^{(s)}(t) \\ \bar{\beta}^{(s)}(t) \end{pmatrix} \right|\right| &\leq \lim_{s\rightarrow\infty}\sup_{t\in[s,s+T]} \left[\Omega(s,t)\exp\left[L_{\beta}(t-s)\right]\right]  \\
&\leq \exp\left[L_{\beta}T\right] \lim_{s\rightarrow\infty}\sup_{t\in[s,s+T]} \Omega(s,t) \\
&=0 ~,~~\text{a.s.}
\end{align}
\end{proof}

\begin{lemma} \label{lemmaA2}
Assume that Assumptions \ref{assumption1}-\ref{assumption5} hold. Then, almost surely, for some $i\geq 1$, 
\begin{equation}
(\alpha(t),\beta(t))\stackrel{ t\rightarrow\infty }{\longrightarrow}\{(\alpha,\beta^{*}_i(\alpha))~:~\alpha\in\mathbb{R}^{d_1}\}.
\end{equation}
\end{lemma}

\begin{proof}

We begin with the observation that, by Lemma \ref{lemmaA1}, $(\alpha_{\gamma_2}(t),\beta_{\gamma_2}(t))$ are asymptotic pseudo-trajectories of \eqref{eqA1} - \eqref{eqA2}. Moreover, by Assumption \ref{assumption4}, they are pre-compact. We can thus apply Theorem 5.7 in Bena\"{i}m \cite{Benaim1999} to conclude that $(\alpha_{\gamma_2}(t),\beta_{\gamma_2}(t))$ converges to an internally chain transitive set for \eqref{eqA1} - \eqref{eqA2}. 

We next observe that the function $g:\mathbb{R}^{d_1}\times\mathbb{R}^{d_2}\rightarrow\mathbb{R}$ is a strict Lyapunov function for \eqref{eqA1} - \eqref{eqA2} in the sense of Bena\"{i}m \cite[Chapter 6.2]{Benaim1999}. In particular, $g(\bar{\alpha}(t),\bar{\beta}(t))$ is strictly decreasing in $t$, unless $(\bar{\alpha}(t),\bar{\beta}(t))$ is an equilibrium point of \eqref{eqA1} - \eqref{eqA2}. This follows straightforwardly from
\begin{equation}
\dot{g}(\bar{\alpha}(t),\bar{\beta}(t)) = - || \nabla_{\beta} g(\bar{\alpha}(t),\bar{\beta}(t))||^2\leq 0.
\end{equation}
with equality if and only if $\nabla_{\beta}g(\bar{\alpha}(t),\bar{\beta}(t))=0$. 
By Assumption \ref{assumption5}, the set of critical values of $g$ is given by $E_g= \cup_{i=1}\{(\alpha,\beta^{*}_i(\alpha):\alpha\in\mathbb{R}^d\}$. Since $\beta_{i}^{*}(\cdot)$ are discrete and countable, this set has Lebesgue measure zero, and hence empty topological interior. Thus, by Proposition 6.4 in Bena\"{i}m \cite{Benaim1999}, every internally chain transitive set for \eqref{eqA1} - \eqref{eqA2} is contained in $E_g$. Moreover, by Assumption \ref{assumption5}, the internally chain transitive sets of $E_g$ are precisely the sets $\{(\alpha,\beta^{*}_{i}(\alpha)):\alpha\in\mathbb{R}^{d_1}\}$. 

It follows from our two observations that, for some $i\geq 1$, $(\alpha_{\gamma_2}(t),\alpha_{\gamma_2}(t)):= (\alpha(p_2(t)),\beta(p_2(t)))\rightarrow\{(\alpha,\beta^{*}_{i}(\alpha)):\alpha\in\mathbb{R}^{d_1}\}$ as $t\rightarrow\infty$. Finally, noting that $t\geq p_2(t)$ for sufficiently large $t$, the result holds.
 
\end{proof}

\subsection{The Slow Timescale}

\subsubsection{Additional Notation} For $i=1,2,\dots$, we will write $\{\underline{\alpha}_i(t)\}_{t\geq 0}$ to denote the solutions of the ordinary differential equations
\begin{align}
\dot{\underline{\alpha}}_i(t) &= -\nabla_{\alpha}f(\underline{\alpha}_i(t),\beta^{*}_i(\underline{\alpha}_i(t))) \label{A12}
\end{align}
where $\beta^{*}_i(\cdot):\mathbb{R}^{d_1}\rightarrow\mathbb{R}^{d_2}$, $i=1,\dots,$ are defined in Assumption \ref{assumption5}.

We can then define $\{\underline{\alpha}_i^{(s)}(t)\}_{0\leq s\leq t}$, $i=1,2,\dots$, as the unique solutions of (\ref{A12}) which `start at $s$', and coincide with the time-scaled process $\{\alpha_{\gamma_1}(t)\}_{t\geq 0}$ at $s$. That is, 
\begin{align}
\dot{\underline{\alpha}}_i^{(s)}(t) = -\nabla_{\alpha} f(\underline{\alpha}_i^{(s)}(t),\beta^{*}_i(\underline{\alpha}_i^{(s)}(t)))~,~~~\underline{\alpha}_i^{(s)}(s) = \alpha_{\gamma_1}(s).
\end{align}

We can also define $\{\underline{\alpha}_{i}^{[s]}(t)\}_{0\leq s\leq t}$, $i=1,2,\dots$, as the unique solutions of (\ref{A12}) which `end at $s$', and coincide with the time-scaled process $\{\alpha_{\gamma_1}(t)\}_{t\geq 0}$ at $s$. That is, 
\begin{align}
\dot{\underline{\alpha}}_{i}^{[s]}(t) = -\nabla_{\alpha} f(\underline{\alpha}_{i}^{[s]}(t),\beta^{*}_i(\underline{\alpha}_i^{[s]}(t)))~,~~~\underline{\alpha}_{i}^{[s]}(s) = \alpha_{\gamma_1}(s).
\end{align}

\subsubsection{Proof of Convergence} We now demonstrate, using these processes, that for some $i\geq 1$, the time-scaled process $\alpha_{\gamma_1}(t)$ is an asymptotic pseudo-trajectory of the flow induced by the ODE for $\underline{\alpha}_i(t)$.

\begin{lemma}  \label{lemmaA3}
Assume that Assumptions \ref{assumption1}-\ref{assumption5} hold. Then, for any $T>0$, and the $i\geq 1$ given in Lemma \ref{lemmaA2}, 
\begin{subequations}
\begin{alignat}{2}
\lim_{s\rightarrow\infty}\sup_{t\in[s,s+T]} ||\alpha_{\gamma_1}(t) - \underline{\alpha}_i^{(s)}(t)||&=0~,~~\text{a.s.} \\
\lim_{s\rightarrow\infty}\sup_{t\in[s-T,s]} ||\alpha_{\gamma_1}(t) - \underline{\alpha}_{i}^{[s]}(t)||&=0~,~~\text{a.s.}
\end{alignat}
\end{subequations}
\end{lemma}

\begin{proof}
This proof is similar in style to the proof of Lemma \ref{lemmaA1}. Once more, we will prove only the first part of this Lemma, as the method for proving the second part is entirely analogous. By definition of the process $\{\alpha(t)\}_{0\leq s\leq t}$, we have 
\begin{align}
\alpha(t)&=\alpha(s)  -\int_s^t\gamma_{1}(u)\nabla_{\alpha} f(\alpha(u),\beta(u))\mathrm{d}u -\int_s^t\gamma_1(u) \mathrm{d}\xi_1(u) 
\end{align}  
It follows immediately from the definition of $\{\alpha_{\gamma_2}(t)\}_{0\leq s\leq t}$ that 
\begin{align}
\alpha_{\gamma_1}(t) &= \alpha_{\gamma_1}(s)  -\int_{p_1(s)}^{p_1(t)} \gamma_1(u)\nabla_{\alpha} f(\alpha(u),\beta(u))\mathrm{d}u -\int_{p_1(s)}^{p_1(t)}\gamma_1(u) \mathrm{d}\xi_1(u) \\
&= \alpha_{\gamma_1}(s)  -\int_{s}^{t} \nabla_{\alpha} f(\alpha_1(u),\beta_1(u))\mathrm{d}u -\int_{p_1(s)}^{p_1(t)}\gamma_1(u) \mathrm{d}\xi_1(u).  \label{A57} 
\end{align}  
We also have, making use of the ODE for $\{\underline{\alpha}_i^{(s)}(t)\}_{0\leq s\leq t}$, that 
\begin{equation}
{\underline{\alpha}}_i^{(s)}(t) = {\underline{\alpha}}_i^{(s)}(s) - \int_{s}^t \nabla_{\alpha} f(\underline{\alpha}_{i}^{(s)}(u),\beta^{*}_i(\underline{\alpha}_i^{(s)}(u)))\mathrm{d}s.  \label{A58}
\end{equation}
It follows straightforwardly from equations (\ref{A57}), (\ref{A58}) that
\begin{align}
\left|\left|\alpha_{\gamma_1}(t) - \underline{\alpha}_i^{(s)}(t)\right|\right|&=\left|\left|-\int_s^t \left[\nabla_{\alpha}f(\alpha_{\gamma_1}(u),\beta_{\gamma_1}(u))-\nabla_{\alpha}f(\underline{\alpha}_{i}^{(s)}(u),\beta^{*}_i(\underline{\alpha}_i^{(s)}(u)))\right]\mathrm{d}u - \int_{p_1(s)}^{p_1(t)}\gamma(u)\mathrm{d}\xi_1(u)\right|\right|  \hspace{-35mm} \\
&\leq \underbrace{\left|\left|\int_{p_1(s)}^{p_1(t)}\gamma_1(u)\mathrm{d}\xi_1(u)\right|\right|}_{\Pi_{1,\alpha}(s,t)} + \underbrace{\left|\left|\int_s^t \left[\nabla_{\alpha}f(\alpha_{\gamma_1}(u),\beta_{\gamma_1}(u))-\nabla_{\alpha}f(\underline{\alpha}_{i}^{(s)}(u),\beta^{*}_i(\underline{\alpha}_i^{(s)}(u)))\right]\mathrm{d}u\right|\right|}_{\Pi_{2,\alpha}(s,t)} \hspace{-35mm} \label{A59}
\end{align}
For the first term, using the first part of Assumption \ref{assumption3}, and arguing as in equations (\ref{A26})-(\ref{A29}), we have that, a.s., for all $T>0$, 
\begin{equation}
\lim_{s\rightarrow\infty}\sup_{t\in[s,s+T]}\Pi_{1,\alpha}(s,t) = 0. \label{A61}
\end{equation}
For the second term, using the triangle inequality, Assumptions \ref{assumption2} and \ref{assumption4}, which together imply the boundedness of $||\nabla_{\alpha} f(\cdot,\cdot)||$, and Assumption \ref{assumption5}, which guarantees the Lipschitz continuity of $\beta^{*}_i(\cdot)$, we have that, a.s., for all $T>0$, 
\begin{align}
\Pi_{2,\alpha}(s,t) 
&\leq \int_s^t \big|\big|\nabla_{\alpha}f(\alpha_{\gamma_1}(u),\beta_{\gamma_1}(u))-\nabla_{\alpha}f(\underline{\alpha}_{i}^{(s)}(u),\beta^{*}_i(\underline{\alpha}_i^{(s)}(u)))\big|\big|\mathrm{d}u \\[2mm]
&\leq \int_s^t \big|\big|\nabla_{\alpha}f(\alpha_{\gamma_1}(u),\beta_{\gamma_1}(u))-\nabla_{\alpha}f(\alpha_{\gamma_1}(u),\beta^{*}_i(\alpha_{\gamma_1}(u)))\big|\big|\mathrm{d}u \nonumber \\
&~~~~+\hspace{.5mm} \int_s^t \big|\big|\nabla_{\alpha} f(\alpha_{\gamma_1}(u),\beta^{*}_i(\alpha_{\gamma_1}(u)))-\nabla_{\alpha}f(\underline{\alpha}_{i}^{(s)}(u),\beta^{*}_i(\underline{\alpha}_i^{(s)}(u))\big|\big|\mathrm{d}u \\[2mm]
& \leq \int_{s}^t L_{\alpha}\left[\big|\big|\alpha_{\gamma_1}(u) - \alpha_{\gamma_1}(u) \big|\big|+\big|\big|\beta_{\gamma_1}(u) - \beta^{*}_i(\alpha_{\gamma_1}(u)) \big|\big|\right]\mathrm{d}u \nonumber \\
&~~~~+ \int_{s}^t L_{\alpha}\left[\big|\big| \alpha_{\gamma_1}(u) - \underline{\alpha}_i^{(s)}(u) \big|\big|+ \big|\big|\beta^{*}_i(\alpha_{\gamma_1}(u)) - \beta^{*}_i(\underline{\alpha}_i^{(s)}(u)) \big|\big|\right]\mathrm{d}u \\[2mm]
& \leq \underbrace{ \int_{s}^t L_{\alpha}\big|\big|\beta_{\gamma_1}(u) - \beta^{*}_i(\alpha_{\gamma_1}(u)) \big|\big|\mathrm{d}u}_{\Pi_{2,\alpha}^{(1)}(s,t)}+ \underbrace{\int_{s}^t L_{\alpha}(1+L_{\beta^{*}_i})\big|\big| \alpha_{\gamma_1}(u) - \underline{\alpha}_i^{(s)}(u) \big|\big|\mathrm{d}u}_{\Pi_{2,\alpha}^{(2)}(s,t)}. \hspace{-5mm} \label{A63}
\end{align}
For the first term, we have that, a.s., for all $T>0$, 
\begin{align} 
\sup_{t\in[s,s+T]}\Pi_{2,\alpha}^{(1)}(s,t) &= \sup_{t\in[s,s+T]}\int_{s}^t L_{\alpha}\big|\big|\beta_{\gamma_1}(u) - \beta^{*}_i(\alpha_{\gamma_1}(u)) \big|\big|\mathrm{d}u \\
&\leq L_{\alpha}T \sup_{t\geq s}\big|\big|\beta_{\gamma_1}(t) - \beta^{*}_i(\alpha_{\gamma_1}(t)) \big|\big|. 
\end{align}
It then follows, using Lemma \ref{lemmaA2}, that, a.s., for all $T>0$, 
\begin{align}
\lim_{s\rightarrow\infty}\sup_{t\in[s,s+T]}\Pi_{2,\alpha}^{(1)}(s,t)
&\leq L_{\alpha}T \limsup_{s\rightarrow\infty}\big|\big|\beta_{\gamma_1}(s) - \beta^{*}_i(\alpha_{\gamma_1}(s)) \big|\big| \\
&=L_{\alpha}T \lim_{s\rightarrow\infty}\big|\big|\beta_{\gamma_1}(s) - \beta^{*}_i(\alpha_{\gamma_1}(s)) \big|\big|\\
&=0. \label{A73} 
\end{align}
It remains to observe, combining inequalities (\ref{A59}) and (\ref{A63}), and making use of Gr\"onwall's Inequality, that
\begin{align}
\left|\left|\alpha_{\gamma_1}(t) - \underline{\alpha}_i^{(s)}(t)\right|\right|&\leq\underbrace{\Pi_{1,\alpha}(s,t) + \Pi_{2,\alpha}^{(1)}(s,t)}_{\Pi(s,t)} + \Pi_{2,\alpha}^{(2)}(s,t) \\[-1.5mm]
&=\Pi(s,t) + \int_{s}^t L_{\alpha}(1+L_{\beta^{*}_i})\big|\big| \alpha_{\gamma_1}(u) - \underline{\alpha}_i^{(s)}(u) \big|\big|\mathrm{d}u \\
&\leq \Pi(s,t)\exp\left[\int_s^t L_{\alpha}(1+L_{\beta^{*}_i})\mathrm{d}u\right] \\[1.5mm]
&= \Pi(s,t)\exp\left[L_{\alpha}(1+L_{\beta^{*}_i})(t-s)\right] \label{A71}
\end{align}
where, from (\ref{A61}) and (\ref{A73}), we have that, a.s., for all $T>0$, 
\begin{equation}
\lim_{s\rightarrow\infty} \sup_{t\in[s,s+T]} \Pi(s,t) = 0. \label{A72}
\end{equation}
It follows immediately from (\ref{A71}) and (\ref{A72}) that, a.s., for all $T>0$, 
\begin{align}
\lim_{s\rightarrow\infty} \sup_{t\in[s,s+T]}  \left|\left|\alpha_{\gamma_1}(t) - \underline{\alpha}_i^{(s)}(t)\right|\right| &\leq \lim_{s\rightarrow\infty} \sup_{t\in[s,s+T]} \left[ \Pi(s,t)\exp\left[L_{\alpha}(1+L_{\beta^{*}_i})(t-s)\right]\right] \\
&\leq \exp\left[L_{\alpha}(1+L_{\beta^{*}_i})\right] \lim_{s\rightarrow\infty} \sup_{t\in[s,s+T]}\Pi(s,t) \\
&= 0.
\end{align}
\end{proof}

\begin{lemma} \label{lemmaA4}
Assume that Assumptions \ref{assumption1}-\ref{assumption6} hold. Then, almost surely
\begin{equation}
\alpha(t)\stackrel{ t\rightarrow\infty }{\longrightarrow} \{\alpha\in\mathbb{R}^{d_1}:\nabla_{\alpha}f(\alpha,\beta^{*}_i(\alpha))=0\}.
\end{equation}
\end{lemma}

\begin{proof}

The proof follows a similar trajectory to the proof of Lemma \ref{lemmaA2}, now with some simplifications. By Lemma \ref{lemmaA3}, $\alpha_{\gamma_1}(t)$ is an asymptotic pseudo-trajectory for \eqref{A12}. Moreover, it is pre-compact by Assumption \ref{assumption4}. Thus, applying Theorem 5.7 in Bena\"{i}m  \cite{Benaim1999}, it follows that $\alpha_{\gamma_1}(t):=\alpha(p_1(t))$ converges to an internally chain transitive set of \eqref{A12}. The same is thus also true for $\alpha(t)$, noting as before that $t\geq p_1(t)$ for sufficiently large $t$. Finally, by Assumption \ref{assumption6}, the only internally chain transitive sets of \eqref{A12} are its (possibly non-isolated) equilibrium points. The result follows immediately.


 
\end{proof}

\subsection{Proof of Theorem \ref{theorem1}}

\begin{theorem_recall_1}
Assume that Assumptions \ref{assumption1}-\ref{assumption6} hold. Then, almost surely,
\begin{equation}
\lim_{t\rightarrow\infty}\nabla_{\alpha}f(\alpha(t),\beta(t)) = \lim_{t\rightarrow\infty}\nabla_{\beta}g(\alpha(t),\beta(t)) = 0.
\end{equation}
\end{theorem_recall_1}

\begin{proof}
The result is an immediate consequence of Lemmas \ref{lemmaA2} and \ref{lemmaA4}. By Lemma \ref{lemmaA2}, the process $(\alpha(t),\beta(t))\rightarrow \{(\alpha,\beta^{*}_i(\alpha)):\alpha\in\mathbb{R}^{d_1}\}$ a.s., for some $i\geq 1$. By Lemma \ref{lemmaA4}, the process $\alpha(t)\rightarrow\{\alpha\in\mathbb{R}^{d_1}:\nabla_{\alpha}f(\alpha,\beta^{*}_i(\alpha))=0\}$ a.s.. Together, these lemmas imply that, for some $i\geq 1$, 
\begin{equation}
(\alpha(t),\beta(t))\rightarrow \{(\alpha,\beta^{*}_i(\alpha))\in\mathbb{R}^{d_1}\times\mathbb{R}^{d_2}:\nabla_{\alpha}f(\alpha,\beta^{*}_i(\alpha))=0\}~\text{ a.s.}
\end{equation}
It follows, in particular, that $\nabla_{\alpha}f(\alpha(t),\beta(t)) \rightarrow 0$ and $\nabla_{\beta}g(\alpha(t),\beta(t))\rightarrow 0$ as $t\rightarrow\infty$ with probability one.
\end{proof}

\section{Extensions to Theorem \ref{theorem1}} \label{theorem1_ext}
In this Appendix, we provide details of several possible extensions to Theorem \ref{theorem1}. We first discuss how to obtain an almost sure convergence result for an alternative version of Algorithm \eqref{eq1} - \eqref{eq2}. We then outline the additional assumptions required in order to establish convergence of Algorithm \eqref{eq1} - \eqref{eq2} to the set of local and global minima of the two objective functions (in a sense to be made precise below).

\subsection{Two-Timescale Stochastic Gradient Descent in Continuous Time: An Alternative Algorithm} \label{theorem1_ext1}
In Theorem \ref{theorem1}, we analysed a two-timescale stochastic gradient descent algorithm designed to solve a weak formulation of the original bilevel optimisation problem in \eqref{global_min}, as stated in \eqref{opt2_}. This, we recall, refers to the task of obtaining $(\alpha^{*},\beta^{*})\in\Lambda_{\alpha}\times\Lambda_{\beta}$ which jointly satisfy
\begin{equation}
\alpha^{*} =\argmin_{\alpha\in U_{\alpha^{*}}}f(\alpha,\beta^{*}) ~~~,~~~\beta^{*} = \argmin_{\beta\in U_{\beta^{*}}} g(\alpha^{*},\beta) \label{opt2}
\end{equation}
where $U_{\alpha^{*}}\subset \mathbb{R}^{d_1}$ and $U_{\beta^{*}}\subset \mathbb{R}^{d_2}$ are local neighbourhoods of $\alpha^{*}$ and $\beta^{*}$, respectively. That is, equivalently, values $(\alpha^{*},\beta^{*})$ such that $\alpha^{*}$ locally minimises $f(\alpha,\beta^{*})$ with respect to $\alpha$, and $\beta^{*}$ which locally minimises $g(\alpha^{*},\beta)$ with respect to $\beta$. 
To tackle this problem using gradient methods, it is natural to consider an algorithm which only utilises (noisy estimates of) the partial derivatives $\nabla_{\alpha}f(\alpha,\beta)$ and $\nabla_{\beta}g(\alpha,\beta)$. This is precisely the form of Algorithm \eqref{eq1} - \eqref{eq2}. 

Suppose, instead, that we would like to solve (a local version of) the original bilevel optimisation problem \eqref{global_min} directly. This is, suppose that we wish to obtain $(\alpha^{*},\beta^{*}(\alpha^{*}))\in\Lambda_{\alpha}\times\Lambda_{\beta}$ which satisfy
\begin{alignat}{2}
\alpha^{*}&= \argmin_{\alpha\in U_{\alpha^{*}}} f\big(\alpha,\beta^{*}(\alpha)\big)~~~\text{s.t.}~~~\beta^{*}(\alpha)&&=\argmin_{\beta\in U_{\beta*(\alpha)}}g(\alpha,\beta) \label{global_min2}
\end{alignat}
where, similarly to above, $U_{\alpha^{*}}\subset \mathbb{R}^{d_1}$ and $U_{\beta^{*}(\alpha)}\subset \mathbb{R}^{d_2}$ are local neighbourhoods of $\alpha^{*}$ and $\beta^{*}(\alpha)$, respectively. The crucial difference between \eqref{opt2} and \eqref{global_min2} is that, in the latter, we insist that $\alpha^{*}$ minimises $f(\alpha,\beta^{*}(\alpha))$ with respect to $\alpha$. In particular, the second argument in the upper level optimisation function now depends explicitly on $\alpha$. 

To tackle this problem using gradient methods, we can still use the partial derivative $\nabla_{\beta}g(\alpha,\beta)$ to minimise the lower-level objective function. If possible, however, we should now use the total derivative $\nabla f(\alpha,\beta^{*}(\alpha))$ to minimise the upper-level objective function. This will, of course, require additional assumptions on the two-objective functions (to be specified below). To make progress in this direction, first note that, via the chain rule, we have
\begin{equation}
\nabla f(\alpha,\beta^{*}(\alpha)) = \nabla_{\alpha} f(\alpha,\beta^{*}(\alpha)) + [\nabla_{\alpha}\beta^{*}(\alpha)]^T \nabla_{\beta}f(\alpha,\beta^{*}(\alpha)).  \label{eqB13}
\end{equation}
Moreover, owing to the first order optimality condition for $\beta^{*}(\alpha)$, under appropriate additional assumptions on $g$, it holds that (see, e.g., \cite{Ghadimi2018})
\begin{equation}
\nabla_{\alpha} \beta^{*}(\alpha)  = -\nabla_{\alpha\beta}^2 g(\alpha,\beta^{*}(\alpha))\left[\nabla^2_{\beta\beta} g(\alpha,\beta^{*}(\alpha))\right]^{-1}. \label{eqB14}
\end{equation}
In practice, $\beta^{*}(\alpha)$ is not available in closed form. Thus, one typically approximates $\nabla_{\alpha}f(\alpha,\beta^{*}(\alpha))$ by replacing $\beta^{*}(\alpha)$ with $\beta\in\mathbb{R}^{d_2}$. This yields, instead of \eqref{eqB13}, the `surrogate' gradient
\begin{equation}
\bar{\nabla} f(\alpha,\beta) = \nabla_{\alpha} f(\alpha,\beta) -\nabla_{\alpha\beta}^2 g(\alpha,\beta)\left[\nabla^2_{\beta\beta} g(\alpha,\beta)\right]^{-1} \nabla_{\beta}f(\alpha,\beta).  \label{surrogate}
\end{equation}
Using this surrogate gradient, we can now obtain an alternative version of Algorithm \eqref{eq1} - \eqref{eq2}. In particular, suppose that we continuously observe noisy gradients of $\bar{\nabla} f(\alpha,\beta)$ and $\nabla_{\beta}g(\alpha,\beta)$, as in \eqref{obs1} - \eqref{obs2}. Then it is natural to consider the following continuous-time two-timescale stochastic gradient descent algorithm:
\begin{subequations}
\begin{align}
\mathrm{d}\alpha(t) &= -\gamma_{1}(t)\left[\bar{\nabla} f(\alpha(t),\beta(t))\mathrm{d}t + \mathrm{d}\xi_1(t)\right],~~~t\geq0,~~~\alpha(0)=\alpha_0, \label{eq1_alt} \\
\mathrm{d}\beta(t) &= -\gamma_{2}(t)\left[\nabla_{\beta} g(\alpha(t),\beta(t))\mathrm{d}t + \mathrm{d}\xi_2(t)\right],~~~t\geq0,~~~\beta(0)=\beta_0, \label{eq2_alt}
\end{align}
\end{subequations}
where $f:\mathbb{R}^{d_1}\times\mathbb{R}^{d_2}\rightarrow\mathbb{R}$ is a continuously differentiable function, $g:\mathbb{R}^{d_1}\times\mathbb{R}^{d_2}\rightarrow\mathbb{R}$ is now a \emph{twice} continuously differentiable function, $\bar{\nabla}f :\mathbb{R}^{d_1}\times\mathbb{R}^{d_2}\rightarrow\mathbb{R}^{d_1}$ is defined in \eqref{surrogate}, and all other terms are as defined in Section \ref{subsec:main1}. This represents the continuous time version of the two-timescale stochastic gradient descent algorithm recently introduced in \cite{Hong2020}. 

We can analyse this algorithm using similar assumptions to those used to establish the convergence of Algorithm \eqref{eq1} - \eqref{eq2} in Theorem \ref{theorem1}. Let us briefly highlight the required modifications. We will first require the following stronger version of Assumption \ref{assumption2}. 

\begin{manualassumption}{2.1.2.i'} \label{assumption2_alt_1}
The outer function $f:\mathbb{R}^{d_1}\times\mathbb{R}^{d_2}\rightarrow\mathbb{R}$ has the following properties 
\begin{itemize}
\item The function $\nabla_{\alpha} f:\mathbb{R}^{d_1}\times\mathbb{R}^{d_2}\rightarrow\mathbb{R}^{d_1}$ is locally Lipschitz continuous. 
\item The function $\nabla_{\beta}f:\mathbb{R}^{d_1}\times\mathbb{R}^{d_2}\rightarrow\mathbb{R}^{d_2}$ is locally Lipschitz continuous.
\end{itemize}
\end{manualassumption} 

\begin{manualassumption}{2.1.2.ii'}  \label{assumption2_alt_2}
The inner function $g:\mathbb{R}^{d_1}\times\mathbb{R}^{d_2}\rightarrow\mathbb{R}$ has the following properties.
\begin{itemize}
\item The function $\nabla_{\beta} g:\mathbb{R}^{d_2}\times\mathbb{R}^{d_2}\rightarrow\mathbb{R}^{d_2}$ is locally Lipschitz continuous. 
\item For all $\alpha\in\mathbb{R}^{d_1}$, the function $g(\alpha,\cdot):\mathbb{R}^{d_1}\rightarrow\mathbb{R}$ is strongly convex. 
\item The function $\smash{\nabla_{\alpha\beta}^2g:\mathbb{R}^{d_1}\times\mathbb{R}^{d_2}\rightarrow\mathbb{R}^{d_1\times d_2}}$ is bounded.
\end{itemize}  
\end{manualassumption}

These assumptions, which also appear in \cite{Ghadimi2018,Hong2020}, ensure that the surrogate $\bar{\nabla}f(\alpha,\beta)$ is well-defined and locally Lipschitz continuous. In particular, it is necessary to assume strong convexity for $g(\alpha,\cdot)$ to ensure that the Hessian of this function, whose inverse appears in the definition of $\bar{\nabla}f(\alpha,\beta)$ in \eqref{surrogate}, is bounded away from zero. An immediate consequence of this assumption is that, for all $\alpha\in\mathbb{R}^{d_1}$, $g(\alpha,\cdot)$ has a single global minimiser. This, in turn, implies that, for all $\alpha\in\mathbb{R}^{d_1}$, the equation $\dot{\beta}(t) = -\nabla_{\beta}g(\alpha,\beta(t))$ has a globally asymptotically stable equilibrium $\beta^{*}(\alpha)$, thus doing away with with the need for Assumption \ref{assumption5}.

Finally, we will now replace Assumption \ref{assumption6} with the following condition.
\begin{manualassumption}{2.1.6'} \label{assumption6_alt}
The set $\smash{f(E_f,\beta_i^{*}(E_f))}$ contains no open sets of $\mathbb{R}^{d_1}$ other than the empty set (i.e., has empty interior), where
\begin{equation}
E_f = \big\{\alpha\in\mathbb{R}^{d_1}:\nabla f(\alpha,\beta^{*}(\alpha))=0\big\}.
\end{equation}
\end{manualassumption}

Interestingly, this assumption is actually slightly weaker than Assumption \ref{assumption6}. This condition was first introduced in \cite{Benaim1999}, in the context of single-timescale stochastic approximation, and later also appeared in \cite{Tadic1997} in a slightly different form: namely, that the set $\smash{f(E_f,\beta^{*}(E_f)) \cap f(E_f^{c},\beta^{*}(E_f^{c}))}$ has Lebesgue measure zero. Both versions have since also appeared in the two-timescale setting \cite{Karmakar2018,Tadic2004}. Broadly speaking, this condition ensures that the function $f(\cdot,\beta^{*}(\cdot))$ admits a certain topological property: namely, that each closed continuous path starting and ending in $\smash{E_{f}^c}$ has a subpath contained in $\smash{E_{f}^c}$ along which $f(\cdot,\beta^{*}(\cdot))$ does not increase. This property prevents the noise processes from forcing the slow process to drift from one connected component of $E_i$ to another. In turn, this ensures that the slow process converges to a connected component of $E_f$. This assumption is satisfied under several more easily verifiable conditions. In particular, it holds if $E_f$ or $f(E_f)$ are countable (e.g., \cite{Benaim1999}). By the Morse-Sard Theorem \cite{Hirsch1976}, it also holds if the function $f(\cdot,\beta^{*}(\cdot))$ is $d_1$-times differentiable, a situation which is somewhat common in two-timescale stochastic approximation algorithms (e.g., \cite{Konda2003a}). 

Our main result on the convergence of Algorithm \eqref{eq1_alt} - \eqref{eq2_alt} is contained in the following theorem. 
\begin{manualtheorem}{2.1'} \label{theorem1_alt}
Assume that Assumptions \ref{assumption1}, \ref{assumption2_alt_1} - \ref{assumption2_alt_2}, \ref{assumption3}, \ref{assumption4}, and \ref{assumption6} hold. Then, almost surely,
\begin{equation}
\lim_{t\rightarrow\infty}\nabla_{\alpha} f(\alpha(t),\beta^{*}(\alpha(t)))= \lim_{t\rightarrow\infty} \nabla_{\beta} g(\alpha(t),\beta(t))=0.
\end{equation}
The second limit implies, in particular, that $\lim_{t\rightarrow\infty}||\beta(t)-\beta^{*}(\alpha(t))||=0$, where $\beta^{*}(\alpha) = \argmin_{\beta\in\mathbb{R}^{d_2}} g(\alpha,\beta)$. 
\end{manualtheorem}

\begin{proof}
Under the stated assumptions, the proof of Theorem \ref{theorem1_alt} is essentially identical to the proof of Theorem \ref{theorem1}. Let us briefly highlight the main changes. Lemma \ref{lemmaA1} goes through unchanged, replacing $\nabla_{\alpha} f(\alpha,\beta)$ with $\bar{\nabla} f(\alpha,\beta)$ where required. Lemma \ref{lemmaA2}, which now states that $(\alpha(t),\beta(t))\rightarrow\{(\alpha,\beta^{*}(\alpha)):\alpha\in\mathbb{R}^{d_1}\}$, is now much more straightforward, since the only internally chain transitive set for (the analogue of) \eqref{eqA1} - \eqref{eqA2} is now the globally asymptotically stable equilibrium point $\beta^{*}(\alpha)$.

Lemma \ref{lemmaA3} is essentially unchanged, again replacing $\nabla_{\alpha}f(\alpha,\beta)$ with $\bar{\nabla} f(\alpha,\beta)$, and also replacing $\beta_{i}^{*}(\alpha)$ with $\beta^{*}(\alpha)$. Finally, Lemma \ref{lemmaA4} is proved along the same lines as the original proof of Lemma \ref{lemmaA2}. In particular, this proof begins by noting that, by (the modified version of) Lemma \ref{lemmaA3}, $\alpha_{\gamma_1}(t)$ is an asymptotic pseudo-trajectory for the ODE
\begin{equation}
\dot{\underline{\alpha}}(t) = -\bar{\nabla} f(\underline{\alpha}(t),\beta^{*}(\underline{\alpha}(t)). \label{f_alt_ode}
\end{equation}
By Assumption \ref{assumption4}, this trajectory is also pre-compact. Thus, by Theorem 5.7 in Benaim \cite{Benaim1999}, $\alpha_{\gamma_1}(t)$ converges to an internally chain transitive set for \eqref{f_alt_ode}. We next observe that $f:\mathbb{R}^{d_1}\times\mathbb{R}^{d_2}\rightarrow\mathbb{R}$ is a strict Lyapunov function for \eqref{f_alt_ode}. Indeed, this follows immediately from
\begin{align}
\dot{f}(\underline{\alpha}(t),\beta^{*}(\underline{\alpha}(t))) &= \dot{\underline{\alpha}}(t)\nabla_{\alpha}(\underline{\alpha}(t),\beta^{*}(\underline{\alpha}(t))) + \dot{\beta}^{*}(\underline{\alpha}(t))\nabla_{\beta}f(\underline{\alpha}(t),\beta^{*}(\underline{\alpha}(t))) \\
&=\dot{\underline{\alpha}}(t)\left[\nabla_{\alpha}(\underline{\alpha}(t),\beta^{*}(\underline{\alpha}(t))) + \left[\nabla_{\alpha} {\beta}^{*}(\underline{\alpha}(t))\right]^T\nabla_{\beta}f(\underline{\alpha}(t),\beta^{*}(\underline{\alpha}(t)))
\right] \\
&=- ||\nabla f(\underline{\alpha}(t),\beta^{*}(\underline{\alpha}(t)))||^2 \leq 0.
\end{align}
In addition, by Assumption \ref{assumption6_alt}, the set of critical values of $f$, namely $E_f$, has Lebesgue measure zero. We can thus apply Proposition 6.4 in Bena\"{i}m \cite{Benaim1999} to conclude that every internally chain transitive set for \eqref{f_alt_ode} is contained in $E_f$. Lemma \ref{lemmaA4} now follows straightforwardly. 

Finally, combining the results of the modified versions of Lemma \ref{lemmaA2} and Lemma \ref{lemmaA4}, one obtains the result of Theorem \ref{theorem1_alt}.
\end{proof}

\subsection{Convergence to Local Minima} \label{theorem1_ext2}
In this section, we outline the (minimal) additional assumptions required in order to establish almost sure convergence of Algorithm \eqref{eq1} - \eqref{eq2} to local minima of the two objective functions. In this case, we mean local minima in the sense of \eqref{opt2}. That is, $(\alpha^{*},\beta^{*})$ such that $\alpha^{*}$ locally minimises $f(\alpha,\beta^{*})$ with respect to $\alpha$, and $\beta^{*}$ which locally minimises $g(\alpha^{*},\beta)$ with respect to $\beta$. We will first require the following definition. 

We will first require the following definition. A twice differentiable function $h:\mathbb{R}^{d}\rightarrow\mathbb{R}^d$ is said to be {strict saddle} if all of its local minima satisfy $\nabla_{x}^2h(x)\succ 0$, and all of its other stationary points satisfy $\lambda_{\min}(\nabla_{x}^2h(x))<0$ (i.e., the minimum eigenvalue of the Hessian evaluated at the critical points is negative). 
We can now introduce the additional assumptions, which will be required in addition to Assumptions \ref{assumption1} - \ref{assumption6}.

\begin{manualassumption}{2.1.3 (L)}
The quadratic variations of the noise processes $\{\xi_i(t)\}_{t\geq 0}\hspace{-1mm}$, $i=1,2$, are uniformly positive definite. 
\end{manualassumption}

\begin{manualassumption}{2.1.5 (L)}
For all $\alpha\in\mathbb{R}^{d_1}$, the function $g(\alpha,\cdot):\mathbb{R}^{d_2}\rightarrow\mathbb{R}$ is twice continuously differentiable. Moreover, this function is strict saddle. 
\end{manualassumption}

\begin{manualassumption}{2.1.6 (L)}
For all $\beta\in\mathbb{R}^{d_2}$, the function $f(\cdot,\beta):\mathbb{R}^{d_1}\rightarrow\mathbb{R}$ is twice continuously differentiable. Moreover, this function is strict saddle.\footnote{ Strictly speaking one would only require that this holds for $\beta\in\mathbb{R}^{d_2}$ such that $\beta=\beta_{i}^{*}(\alpha)$ for some $i\geq 1$, $\alpha\in\mathbb{R}^{d_1}$. }
\end{manualassumption}

The analogue of these assumptions (for a single objective function) appear in both classical \cite{Brandiere1998,Pemantle1990} and more recent \cite{Ge2015,Mertikopoulos2020} 
results on the `avoidance of saddles' in the discrete-time stochastic approximation literature. See also \cite{Yang2020} for a related result in continuous time. Broadly speaking, the first of these assumptions is required in order to ensure that the additive noise processes are `sufficiently exciting', that is, that they have sufficiently large components in all directions. Meanwhile, the final two assumptions rule out degenerate cases in which the Hessian does not contain sufficient information to characterise the nature of a critical point. 

Using these assumptions, one can establish (in single-timescale, discrete-time stochastic gradient descent) that unstable equilibria (i.e., saddle points) are avoided with probability one via the central manifold theorem (e.g., \cite{Shub1987}). We leave the rigorous extension of these results to the continuous-time, two-timescale framework to future work.

\subsection{Convergence to Global Minima} \label{theorem1_ext3}
We conclude this appendix by detailing the assumptions required to establish convergence to global minima of the two objective functions. Similarly to before, we use the term `global minima' to mean $(\alpha^{*},\beta^{*})$ such that, simultaneously, $\alpha^{*}$ globally minimises $f(\alpha,\beta^{*})$, and $\beta^{*}$ globally minimises $g(\alpha,\beta^{*})$. 

In this case, we can replace Assumption \ref{assumption5} and Assumption \ref{assumption6} with the following.
\begin{manualassumption}{2.1.5 (G)} \label{alt1}
For all $\alpha\in\mathbb{R}^{d_1}$, the ordinary differential equation
\begin{equation}
\dot{\beta}(t) = -\nabla_{\beta} g(\alpha,\beta(t))
\end{equation} 
has a globally asymptotically stable equilibrium $\beta^{*}(\alpha)$, where $\beta^{*}:\mathbb{R}^{d_1}\rightarrow\mathbb{R}^{d_2}$ is a Lipschitz-continuous map. 
\end{manualassumption}

\begin{manualassumption}{2.1.6 (G)} \label{alt2}
The ordinary differential equation 
\begin{equation}
\dot{\alpha}(t) = -\nabla_{\alpha}f(\alpha(t),\beta^{*}(\alpha(t)))
\end{equation}
has a globally asymptotically stable equilibrium $\alpha^{*}$. 
\end{manualassumption}

These are rather classical assumptions used to establish almost sure convergence of discrete-time two-timescale stochastic approximation algorithms to a unique equilibrium point $(\alpha^{*},\beta^{*}(\alpha^{*}))$ (e.g., \cite{Borkar1997,Borkar2008,Konda1999,Tadic2004}). In the context of two-timescale stochastic gradient descent, it is common to instead use the assumptions that the functions $g(\alpha,\cdot):\mathbb{R}^{d_2}\rightarrow\mathbb{R}^{d_2}$ and $f(\alpha,\beta^{*}(\alpha)):\mathbb{R}^{d_1}\rightarrow\mathbb{R}^{d_1}$ are (strongly) convex, with some additional assumptions on the mixed partial derivatives of $g$ to ensure that $\beta^{*}(\alpha)$ is Lipschitz continuous (e.g., \cite{Doan2021,Ghadimi2018,Hong2020}).

Under these assumptions, the proof of Theorem \ref{theorem1} only requires minor modifications in order to establish convergence to the global minima, that is, $(\alpha(t),\beta(t))\rightarrow(\alpha^{*},\beta^{*}(\alpha^{*}))$ as $t\rightarrow\infty$. In particular, Assumption \ref{alt1} implies that Lemma \ref{lemmaA2} can now conclude $(\alpha(t),\beta(t))\rightarrow\{(\alpha,\beta^{*}(\alpha)):\alpha\in\mathbb{R}^{d_1}\}$. Meanwhile, Assumption \ref{alt2} implies that Lemma \ref{lemmaA4} yields $\alpha(t)\rightarrow \alpha^{*}$. The remainder of the proof is unchanged.

\section{Proof of Theorem \ref{theorem1a}}
\label{sec:proof2} 
In this Appendix, we provide a proof of Theorem \ref{theorem1a}. Our proof combines the methods in \cite[Lemma 3.1]{Sirignano2017a} and \cite[Lemma 1]{Surace2019}, adapted appropriately to the two-timescale setting, with the results of Theorem \ref{theorem1}. 

\begin{lemma} \label{lemmaB1}
For $0\leq s\leq t$, define
\begin{subequations}
\begin{align}
\Gamma_{\alpha}(s,t) &= \int_s^t \gamma_1(u)\left[F(\alpha(u),\beta(u),\mathcal{X}(u))-\nabla_{\alpha}f(\alpha(u),\beta(u))\right]\mathrm{d}u \\
\Gamma_{\beta}(s,t) &= \int_s^t \gamma_2(u)\left[G(\alpha(u),\beta(u),\mathcal{X}(u))-\nabla_{\beta}g(\alpha(u),\beta(u))\right]\mathrm{d}u.
\end{align}
\end{subequations}

Assume that Assumptions \ref{assumption1a}-\ref{assumption5a} hold. Then, for all $T\in[0,\infty)$, with probability one,
\begin{subequations}
\begin{align}
\lim_{s\rightarrow\infty} \sup_{t\in[s,s+T]} ||\Gamma_{\alpha}(s,t)|| &= 0, \\
 \lim_{s\rightarrow\infty} \sup_{t\in[s,s+T]} ||\Gamma_{\beta}(s,t)|| &=  0.
\end{align}
\end{subequations}
\end{lemma}

\begin{proof}
We will prove only the first part of the Lemma, as the method for proving the second part is entirely analogous. By Assumption \ref{assumption4a}, there exists a differentiable function $f:\mathbb{R}^{d_1}\times\mathbb{R}^{d_2}\rightarrow\mathbb{R}$, and a unique Borel-measurable function $\tilde{F}:\mathbb{R}^{d_1}\times\mathbb{R}^{d_2}\times\mathbb{R}^{d_3}$ such that $\nabla_{\alpha}f(\cdot)$ is Lipschitz continuous, and moreover, such that the Poisson equation
\begin{equation}
\mathcal{A}_{\mathcal{X}}\tilde{F}(\alpha,\beta,x) = F(\alpha,\beta,x) - \nabla_{\alpha}f(\alpha,\beta) \label{eqb03}
\end{equation}
has a unique, twice-differentiable solution which grows at most polynomially in $x$. In particular, there exist $K',q'>0$ such that
\begin{subequations}
\begin{align}
\sum_{i=0}^2 ||\partial_{\alpha}^{i} \tilde{F}(\alpha,\beta,x)|| + ||\partial_x\partial_{\alpha}\tilde{F}(\alpha,\beta,x)||   &\leq K(1+||x||^{q'}),\label{poisson_eq}  \\
\sum_{i=0}^2 ||\partial_{\beta}^{i} \tilde{F}(\alpha,\beta,x)|| + ||\partial_x\partial_{\beta}\tilde{F}(\alpha,\beta,x)|| &\leq K(1+||x||^{q'}).
\end{align}
\end{subequations}
Now consider the vector-valued function $\hat{F}(\alpha,\beta,x,t) = \gamma_1(t)\tilde{F}(\alpha,\beta,x)$,  with $\tilde{F}$ as defined in \eqref{eqb03}.  Applying It\^o's Lemma to each component of $\hat{F}$, we obtain, for $i=1,\dots,d_1$,
{\allowdisplaybreaks
\begin{align}
&\hat{F}_i(\alpha(t),\beta(t),\mathcal{X}(t),t)-\hat{F}_i(\alpha(s),\beta(s),\mathcal{X}(s),s) \\
&=\int_s^t \partial_{\tau} \hat{F}_i(\alpha(\tau),\beta(\tau),\mathcal{X}(\tau),\tau) \mathrm{d}\tau + \int_s^t \mathcal{A}_{{\mathcal{X}}} \hat{F}_i(\alpha(\tau),\beta(\tau),\mathcal{X}(\tau),\tau)\mathrm{d}\tau \nonumber\\
&- \int_s^t \gamma_1(\tau)F(\alpha(\tau),\beta(\tau),\mathcal{X}(\tau))\cdot \nabla_{\alpha} \hat{F}_i(\alpha(\tau),\beta(\tau),\mathcal{X}(\tau),\tau)\mathrm{d}\tau \nonumber\\
&- \int_s^t \gamma_2(\tau)G(\alpha(\tau),\beta(\tau),\mathcal{X}(\tau))\cdot \nabla_{\beta} \hat{F}_i(\alpha(\tau),\beta(\tau),\mathcal{X}(\tau),\tau)\mathrm{d}\tau \nonumber \\
&+\frac{1}{2}\int_{s}^t\gamma_1^2(\tau)\nabla_{\alpha}\nabla_{\alpha} \hat{F}_i(\alpha(\tau),\beta(\tau),\mathcal{X}(\tau),\tau):\mathrm{d}\left[\zeta_1,\zeta_1\right](\tau)\nonumber \\
&+\frac{1}{2}\int_{s}^t\gamma_2^2(\tau)\nabla_{\beta} \nabla_{\beta} \hat{F}_i(\alpha(\tau),\beta(\tau),\mathcal{X}(\tau),\tau):\mathrm{d}\left[\zeta_2,\zeta_2\right](\tau) \nonumber\\
&-\int_s^t \gamma_1(\tau) \nabla_{\alpha} \hat{F}_i(\alpha(\tau),\beta(\tau),\mathcal{X}(\tau),\tau) \cdot \mathrm{d}\zeta_1(\tau) \nonumber\\
&-\int_s^t \gamma_2(\tau) \nabla_{\beta} \hat{F}_i(\alpha(\tau),\beta(\tau),\mathcal{X}(\tau),\tau) \cdot \mathrm{d}\zeta_2(\tau)\nonumber \\
&+\int_s^t \nabla_x \hat{F}_i(\alpha(\tau),\beta(\tau),\mathcal{X}(\tau),\tau) \cdot \Psi(\alpha(\tau),\beta(\tau),\mathcal{X}(\tau))\mathrm{d}b(\tau) \nonumber\\
&+ \int_s^t \gamma_1(\tau)\gamma_2(\tau) \nabla_{\alpha}\nabla_{\beta}\hat{F}_i(\alpha(\tau),\beta(\tau),\mathcal{X}(\tau),\tau):\mathrm{d}\left[\zeta_1,\zeta_2\right](\tau) \nonumber\\
&-\int_s^t \gamma_1(\tau) \nabla_{\alpha}\nabla_{x}\hat{F}_i(\alpha(\tau),\beta(\tau),\mathcal{X}(\tau),\tau):\Psi(\alpha(\tau),\beta(\tau),\mathcal{X}(\tau))\mathrm{d}\left[\zeta_1,b\right](\tau) \nonumber\\
&-\int_s^t \gamma_2(\tau) \nabla_{\beta}\nabla_{x}\hat{F}_i(\alpha(\tau),\beta(\tau),\mathcal{X}(\tau),\tau):\Psi(\alpha(\tau),\beta(\tau),\mathcal{X}(\tau))\mathrm{d}\left[\zeta_2,b\right](\tau) \nonumber
\end{align}
}
For the sake of brevity, we will proceed under the assumption that the continuous semi-martingales $\{\zeta_i(t)\}_{t\geq 0}$, $i=1,2$ are, in fact, diffusion processes. We should emphasise, however, that this assumption does not change the subsequent analysis in any meaningful way, and can be easily relaxed. In particular, we will assume that
\begin{equation}
\mathrm{d}\zeta_i(t) = \zeta_i^{(1)}(\alpha(t),\beta(t),\mathcal{X}(t))\mathrm{d}t + \zeta_i^{(2)}(\alpha(t),\beta(t),\mathcal{X}(t)) \mathrm{d}z_i(t)
\end{equation}
where $\smash{\zeta_i^{(1)}(\alpha,\beta,\cdot):\mathbb{R}^{d_3}\rightarrow}$ $\smash{\mathbb{R}^{d_i}}$ and $\smash{\zeta_i^{(2)}(\alpha,\beta,\cdot):\mathbb{R}^{d_3}\rightarrow \mathbb{R}^{d_i\times d_5^{i}}}$ are Borel measurable functions
; and $\{z_i(t)\}_{t\geq 0}$ are $\mathbb{R}^{d_5^{i}}$ valued Wiener processes. 
In this case, recalling the definition of the functions $c_{z_1,z_2}$, $c_{z_1,b}$ and $c_{z_2,b}$ in Assumption \ref{assumption5c}, the previous equation becomes
{\allowdisplaybreaks
\begin{align}
&\hat{F}_i(\alpha(t),\beta(t),\mathcal{X}(t),t)-\hat{F}_i(\alpha(s),\beta(s),\mathcal{X}(s),s) \\
&=\int_s^t \partial_{\tau} \hat{F}_i(\alpha(\tau),\beta(\tau),\mathcal{X}(\tau),\tau) \mathrm{d}\tau + \int_s^t \mathcal{A}_{{\mathcal{X}}} \hat{F}_i(\alpha(\tau),\beta(\tau),\mathcal{X}(\tau),\tau)\mathrm{d}\tau \nonumber\\
&+\int_s^t \mathcal{A}_{\alpha}\hat{F}_i(\alpha(\tau),\beta(\tau),\mathcal{X}(\tau),\tau)\mathrm{d}\tau +\int_s^t \mathcal{A}_{\beta}\hat{F}_i(\alpha(\tau),\beta(\tau),\mathcal{X}(\tau),\tau)\mathrm{d}\tau \nonumber\\
&-\int_s^t \gamma_1(\tau) \nabla_{\alpha} \hat{F}_i(\alpha(\tau),\beta(\tau),\mathcal{X}(\tau),\tau) \cdot \zeta_1^{(2)}(\tau)\mathrm{d}z_1(\tau)\nonumber \\
&-\int_s^t \gamma_2(\tau) \nabla_{\beta} \hat{F}_i(\alpha(\tau),\beta(\tau),\mathcal{X}(\tau),\tau) \cdot \zeta_2^{(2)}(\tau)\mathrm{d}z_2(\tau) \nonumber\\
&+\int_s^t \nabla_x \hat{F}_i(\alpha(\tau),\beta(\tau),\mathcal{X}(\tau),\tau) \cdot \Psi(\tau)\mathrm{d}b(\tau) \nonumber\\
&-\int_s^t \gamma_1(\tau) \mathrm{Tr}\left[\nabla_{\alpha}\nabla_{x}\hat{F}_i(\alpha(\tau),\beta(\tau),\mathcal{X}(\tau),\tau){\Psi}(\tau) \zeta_1^{(2)}(\tau)c_{z_1,b}(\tau) \right]\mathrm{d}\tau \nonumber\\
&-\int_s^t \gamma_2(\tau)\mathrm{Tr}\left[\nabla_{\beta}\nabla_{x}\hat{F}_i(\alpha(\tau),\beta(\tau),\mathcal{X}(\tau),\tau){\Psi}(\tau) \zeta_2^{(2)}(\tau)c_{z_2,b}(\tau)\right]\mathrm{d}\tau \nonumber\\
&+\int_s^t \gamma_1(\tau)\gamma_2(\tau) \mathrm{Tr}\left[\nabla_{\alpha}\nabla_{\beta}\hat{F}_i(\alpha(\tau),\beta(\tau),\mathcal{X}(\tau),\tau)\zeta_1^{(2)}(\tau)\zeta_2^{(2)}(\tau)c_{z_1,z_2}(\tau)\right]\mathrm{d}\tau, \nonumber
\end{align}
}
where, $\mathcal{A}_{\alpha}$ and $\mathcal{A}_{\beta}$ are the infinitesimal generators of the processes $\{\alpha(t)\}_{t\geq 0}$ and $\{\beta(t)\}_{t\geq 0}$; and $\nabla_{\alpha}\nabla_{\beta} u_k(\alpha,\beta,x,\tau)_{ij} = \partial_{\alpha_i}\partial_{\beta_j} u_k(\alpha,\beta,x,\tau)$, with $\nabla_{\alpha}\nabla_{x}$ and $\nabla_{\beta}\nabla_{x}$ defined similarly. For the sake of simplicity, we have temporarily suppressed the dependence of the functions $\smash{\zeta_i^{(1)},\zeta_i^{(2)}}$, $i=1,2$, and $\Psi$ on $\smash{\{\alpha(t)\}_{t\geq 0},\{\beta(t)\}_{t\geq 0}}$ and $\smash{\{\mathcal{X}(t)\}_{t\geq0}}$. It follows straightforwardly that
{\allowdisplaybreaks
\begin{align}
\Gamma_{\alpha}(s,t) &= \int_s^t \gamma_1(\tau)\left[F(\alpha(\tau),\beta(\tau),\mathcal{X}(\tau))-\nabla_{\alpha}f(\alpha(\tau),\beta(\tau))\right]\mathrm{d}\tau  \\
&=\int_s^t \gamma_1(\tau)\mathcal{A}_{{\mathcal{X}}}\tilde{F}(\alpha(\tau),\beta(\tau),\mathcal{X}(\tau))\mathrm{d}\tau \\
&=\int_s^t \mathcal{A}_{{\mathcal{X}}}\hat{F}(\alpha(\tau),\beta(\tau),\mathcal{X}(\tau))\mathrm{d}\tau  \\[2mm]
&=\gamma_1(t)\tilde{F}(\alpha(t),\beta(t),\mathcal{X}(t))-\gamma_1(s)\tilde{F}(\alpha(s),\beta(s),\mathcal{X}(s)) \\[1mm]
&-\int_s^t \dot{\gamma}_1(\tau)\partial_{\tau} \tilde{F}(\alpha(\tau),\beta(\tau),\mathcal{X}(\tau)) \mathrm{d}\tau \nonumber\\
&-\int_s^t \gamma_1(\tau)\mathcal{A}_{\alpha}  \tilde{F}(\alpha(\tau),\beta(\tau),\mathcal{X}(\tau))\mathrm{d}\tau -\int_s^t \gamma_1(\tau) \mathcal{A}_{\beta} \tilde{F}(\alpha(\tau),\beta(\tau),\mathcal{X}(\tau))\mathrm{d}\tau \nonumber \\
&+\int_s^t \gamma_1^2(\tau) \nabla_{\alpha} \tilde{F}(\alpha(\tau),\beta(\tau),\mathcal{X}(\tau),\tau) \cdot \zeta_1^{(2)}(\tau)\mathrm{d}z_1(\tau) \nonumber\\
&+\int_s^t \gamma_1(\tau)\gamma_2(\tau) \nabla_{\beta} \tilde{F}(\alpha(\tau),\beta(\tau),\mathcal{X}(\tau),\tau) \cdot \zeta_2^{(2)}(\tau)\mathrm{d}z_2(\tau)  \nonumber\\
&-\int_s^t \gamma_1(\tau) \nabla_x \tilde{F}(\alpha(\tau),\beta(\tau),\mathcal{X}(\tau)) \cdot \Psi(\tau)\mathrm{d}b(\tau) \nonumber\\
&+\int_s^t \gamma^2_1(\tau) \mathrm{Tr}\left[\nabla_{\alpha}\nabla_{x}\tilde{F}(\alpha(\tau),\beta(\tau),\mathcal{X}(\tau)){\Psi}(\tau) \zeta_1^{(2)}(\tau)c_{z_1,b}(\tau) \right]\mathrm{d}\tau \nonumber\\
&+\int_s^t \gamma_1(\tau)\gamma_2(\tau)\mathrm{Tr}\left[\nabla_{\beta}\nabla_{x}\tilde{F}(\alpha(\tau),\beta(\tau),\mathcal{X}(\tau)){\Psi}(\tau) \zeta_2^{(2)}(\tau)c_{z_2,b}(\tau)\right]\mathrm{d}\tau \hspace{-5mm}  \nonumber\\
&- \int_s^t \gamma^2_1(\tau)\gamma_2(\tau) \mathrm{Tr}\left[\nabla_{\alpha}\nabla_{\beta}\tilde{F}(\alpha(\tau),\beta(\tau),\mathcal{X}(\tau))\zeta_1^{(2)}(\tau)\zeta_2^{(2)}(\tau)c_{z_1,z_2}(\tau)\right]\mathrm{d}\tau \nonumber 
\end{align}
}
We will now bound each of these terms in turn. We first define
\begin{equation}
J^{(1)}(t) = \gamma_{1}(t)\sup_{\tau\in[0,t]}||\tilde{F}(\alpha(\tau),\beta(\tau),\mathcal{X}(\tau))||.
\end{equation}
By Assumption \ref{assumption4a} and Assumption \ref{assumption4aii}, there exists $q>0$, and $K,K'>0$ such that for all $t$ sufficiently large, we have
\begin{align}
\mathbb{E}\big[\big(J^{(1)}(t)\big)^2\big]&= \mathbb{E}\big[\gamma_1^2(t)\sup_{\tau\in[0,t]}||\tilde{F}(\alpha(\tau),\beta(\tau),\mathcal{X}(\tau))||^2\big] \\
&\leq K\gamma_1^2(t)\big[1+\mathbb{E}\sup_{\tau\in[0,t]}||\mathcal{X}(\tau)||^q\big] \\
&\leq K\gamma_1^2(t)\big[1+K'\sqrt{t}\big] \\[2mm]
&\leq K''\gamma_1^2(t)\sqrt{t}.\label{B55}
\end{align}
By Assumption \ref{assumption1a}, there exists $r_1>0$ such that $\lim_{t\rightarrow\infty}\gamma_1^2(t)t^{\frac{1}{2}+2r_1}=0$. In particular, there exists $T>0$ such that for all $t\geq T$, 
\begin{equation}
\gamma_1^2(t)t^{\frac{1}{2}+2r_1}\leq 1.\label{B56}
\end{equation}
Now suppose that, for any $0< \delta <r_1$, we define the event $A_{\delta}(t) = \{J^{(1)}(t)\cdot t^{r_1-\delta}\geq 1\}$. Then, by Markov's inequality, equation (\ref{B55}), and equation (\ref{B56}), we have that, for all $t\geq T$, 
\begin{equation}
\mathbb{P}(A_{\delta}(t))\leq \mathbb{E}\left[(J^{(1)}(t))^2\right]t^{2(r_{1}-\delta)}\leq K''\gamma_{1}^2(t)t^{\frac{1}{2}+2r_{1}-2\delta}\leq K''t^{-2\delta}
\end{equation}
It follows that $\sum_{n=1}^{\infty} \mathbb{P}(A_{\delta}(2^n))<\infty$. By the Borel-Cantelli Lemma, this observation implies that only finitely many events $A_{\delta}(2^n)$ can occur. Therefore, there exists a random index $n_0(\omega)$ such that 
\begin{equation}
J^{(1)}(2^n) \cdot 2^{n(r_{1}-\delta)}\leq 1
\end{equation}
for all $n\geq n_0$. Equivalently, there exists a finite positive random variable $d(\omega)$ and a deterministic $0<n_1<\infty$ such that for all $n\geq n_1$, 
\begin{equation}
J^{(1)}(2^n)\cdot 2^{n(r-\delta)}\leq d(\omega).
\end{equation}
Thus, for $t\in[2^n,2^{n+1}]$, and $n\geq n_1$, we have, for some constant $0<K<\infty$,
\begin{align}
J^{(1)}(t)&=\gamma_{1}(t)\sup_{\tau\in[0,t]}||\tilde{F}(\alpha(\tau),\beta(\tau),\mathcal{X}(\tau))|| \\[2mm]
& \leq K \gamma_{1}(2^{n+1})\sup_{\tau\in[0,2^{n+1}]}||\tilde{F}(\alpha(\tau),\beta(\tau),\mathcal{X}(\tau))|| \\
& = KJ^{(1)}(2^{n+1})\\
&\leq K\frac{d(\omega)}{2^{(n+1)(r_{1}-\delta)}} \\
&\leq K\frac{d(\omega)}{t^{r_{1}-\delta}}.
\end{align}
It follows that, for all $t\geq 2^{n_0}$, with probability one,
\begin{equation}
J^{(1)}(t)\leq K\frac{d(\omega)}{t^{r_1-\delta}}\rightarrow 0~\text{as $t\rightarrow\infty$.} \label{B65}
\end{equation}
We next define 
\begin{align}
J^{(2)}(t) &= \int_0^t \bigg|\bigg| \dot{\gamma}_1(\tau)\partial_{\tau} \tilde{F}(\alpha(\tau),\beta(\tau),\mathcal{X}(\tau)) \\
&\hspace{8.5mm}+ \gamma_1(\tau)\mathcal{A}_{\alpha} \tilde{F}(\alpha(\tau),\beta(\tau),\mathcal{X}(\tau)) +\gamma_1(\tau) \mathcal{A}_{\beta} \tilde{F}(\alpha(\tau),\beta(\tau),\mathcal{X}(\tau)) \hspace{-5mm} \nonumber \\[2mm] 
&\hspace{8.5mm}-\gamma^2_1(\tau) \mathrm{Tr}\left[\nabla_{\alpha}\nabla_{x}\tilde{F}(\alpha(\tau),\beta(\tau),\mathcal{X}(\tau)){\Psi}(\tau) \zeta_1^{(2)}(\tau)c_{z_1,b}(\tau) \right] \nonumber\\[1mm]
&\hspace{8.5mm}- \gamma_1(\tau)\gamma_2(\tau)\mathrm{Tr}\left[\nabla_{\beta}\nabla_{x}\tilde{F}(\alpha(\tau),\beta(\tau),\mathcal{X}(\tau)){\Psi}(\tau) \zeta_2^{(2)}(\tau)c_{z_2,b}(\tau)\right] \hspace{-8mm}\nonumber \\
&\hspace{8.5mm}+\gamma^2_1(\tau)\gamma_2(\tau) \mathrm{Tr}\left[\nabla_{\alpha}\nabla_{\beta}\tilde{F}(\alpha(\tau),\beta(\tau),\mathcal{X}(\tau))\zeta_1^{(2)}(\tau)\zeta_2^{(2)}(\tau)c_{z_1,z_2}(\tau)\right] \bigg|\bigg| \hspace{.5mm}\mathrm{d}\tau. \nonumber
\end{align}
By Assumptions \ref{assumption1a}, \ref{assumption4aiii}, \ref{assumption4a}, \ref{assumption4aii}, \ref{assumption5b} and \ref{assumption5c}, there exists $q>0$, and constants $K,K',K''>0$ such that
\begin{align}
\sup_{t\geq 0} \mathbb{E}[J^{(2)}(t)]
&\leq K \int_0^{\infty}\left(\dot{\gamma}_{1}(\tau)+\gamma_{1}^2(\tau)+\gamma_1(\tau)\gamma_2(\tau)+\gamma_1^2(\tau)\right.\\[-1mm]
&\hspace{20mm}+\left.\gamma_1(\tau)\gamma_2(\tau)+\gamma_1^2(\tau)\gamma_2(\tau)\right)(1+\mathbb{E}||{\mathcal{X}}(\tau)||^q)\mathrm{d}\tau \nonumber \\[2mm]
&\leq KK' \int_0^{\infty}\left(\dot{\gamma}_{1}(\tau)+\gamma_{1}^2(\tau)+\gamma_2^2(\tau)+\gamma_1^2(\tau)\right.\\[-1mm]
&\hspace{25mm}+\left.\gamma_1(\tau)\gamma_2(\tau)+\gamma_1^2(\tau)\gamma_2(\tau)\right)\mathrm{d}\tau \nonumber \\[2mm]
&\leq KK'K''<\infty.
\end{align}
In particular, the first inequality follows from Assumptions \ref{assumption4aiii}, \ref{assumption4a}, \ref{assumption5b} and \ref{assumption5c}, using additionally the fact that $\mathcal{A}_{\alpha}$ contains at least a factor of $\gamma_1(t)$, and $\mathcal{A}_{\beta}$ contains at least a factor of $\gamma_2(t)$. The second inequality follows from the first part of Assumption \ref{assumption4aii}. The final inequality follows from Assumption \ref{assumption1a}. It follow that there exists a finite random variable, say $\smash{\bar{J}^{(2)}_{\infty}}$, such that, with probability one,
\begin{equation}
\lim_{t\rightarrow\infty} J^{(2)}(t)= \bar{J}^{(2)}_{\infty}. \label{B74}
\end{equation}
Finally, we define 
\begin{align}
J^{(3)}(t) &= \int_0^t \gamma_1^2(\tau) \nabla_{\alpha} \tilde{F}(\alpha(\tau),\beta(\tau),\mathcal{X}(\tau),\tau) \cdot \zeta_1^{(2)}(\tau)\mathrm{d}z_1(\tau) \\
&\hspace{6mm}+\gamma_1(\tau)\gamma_2(\tau) \nabla_{\beta} \tilde{F}(\alpha(\tau),\beta(\tau),\mathcal{X}(\tau),\tau) \cdot \zeta_2^{(2)}(\tau)\mathrm{d}z_2(\tau) \nonumber \\[1.5mm]
&\hspace{6mm}-\gamma_1(\tau) \nabla_x \tilde{F}(\alpha(\tau),\beta(\tau),\mathcal{X}(\tau)) \cdot \Psi(\tau)\mathrm{d}b(\tau) \nonumber
\end{align}
By the It\^o Isometry, and Assumptions \ref{assumption1a}, \ref{assumption4aiii}, \ref{assumption4a}, \ref{assumption4aii}, \ref{assumption5b} and \ref{assumption5c}, similar calculations to those for $J^{(2)}(t)$ show there exists $q>0$, and constants $K,K',K''>0$ such that
\begin{align}
\sup_{t\geq 0} \mathbb{E}[||J^{(3)}(t)||^2]
&\leq K \int_0^{\infty}\left(\gamma_{1}^4(\tau)+\gamma_1^2(\tau)\gamma^2_2(\tau)+\gamma_1^2(\tau)+2\gamma_1^3(\tau)\gamma_2(\tau) \right. \\[1mm]
&\hspace{18mm}+\left.2\gamma_1^2(\tau)\gamma_2(\tau)+2\gamma_1^3(\tau)\right)(1+\mathbb{E}||{\mathcal{X}}(\tau)||^q)\hspace{.5mm}\mathrm{d}\tau \nonumber \\[2mm]
&\leq KK' \int_0^{\infty}\left(\gamma_{1}^4(\tau)+\gamma_1^2(\tau)\gamma^2_2(\tau)+\gamma_1^2(\tau) \right. \\[1mm]
&\hspace{20mm}+\left.2\gamma_1^3(\tau)\gamma_2(\tau)+2\gamma_1^2(\tau)\gamma_2(\tau)+2\gamma_1^3(\tau)\right)\hspace{.5mm}\mathrm{d}\tau \nonumber \\[2mm]
&\leq KK'K''<\infty.
\end{align}
Thus, by Doob's martingale convergence theorem, there exists a square integrable random variable, say $\bar{J}^{(3)}_{\infty}$, such that, with probability one and in $L^2$, 
\begin{equation}
\lim_{t\rightarrow\infty} J^{(3)}(t)= \bar{J}^{(3)}_{\infty}. \label{B82}
\end{equation}
It remains only to observe that 
\begin{equation}
||\Gamma_{\alpha}(s,t)|| \leq J^{(1)}(t) + J^{(1)}(s) + J^{(2)}(t) - J^{(2)}(s)  + ||J^{(3)}(t) - J^{(3)}(s)||. 
\end{equation}
Together with (\ref{B65}), (\ref{B74}) and (\ref{B82}), this expression implies that for all $T\in[0,\infty)$, with probability one, 
\begin{equation}
\lim_{s\rightarrow\infty}\sup_{t\in[s,s+T]}||\Gamma_{\alpha}(s,s+T)||= 0.
\end{equation} 

\end{proof}

\begin{theorem_recall_2}
Assume that Assumptions \ref{assumption1a} - \ref{assumption5c} and \ref{assumption4} hold. In addition, assume that Assumptions \ref{assumption5} - \ref{assumption6} hold for the functions $f(\cdot)$ and $g(\cdot)$ defined in Assumption \ref{assumption4a}. Then, almost surely, 
\begin{equation}
\lim_{t\rightarrow\infty}\nabla_{\alpha}f(\alpha(t),\beta(t)) = \lim_{t\rightarrow\infty}\nabla_{\beta}g(\alpha(t),\beta(t)) = 0.
\end{equation}
\end{theorem_recall_2}

\begin{proof}
We begin with the observation that Algorithm \eqref{alg2_1} - \eqref{alg2_2} can be written in the form of Algorithm \eqref{eq1} - \eqref{eq2}, viz
\begin{align}
\mathrm{d}\alpha(t) = -\gamma_{1}(t)\bigg[&\nabla_{\alpha} f(\alpha(t),\beta(t))\mathrm{d}t + \\[-3.5mm]
& \underbrace{\big(F(\alpha(t),\beta(t),\mathcal{X}(t))-\nabla_{\alpha} f(\alpha(t),\beta(t))\big)\mathrm{d}t + \mathrm{d}\zeta_1(t)}_{=\hspace{.5mm}\mathrm{d}\xi_1(t)}\bigg], \nonumber \\
\mathrm{d}\beta(t) = -\gamma_{2}(t)\bigg[&\nabla_{\beta} g(\alpha(t),\beta(t))\mathrm{d}t + \\[-3.5mm]
& \underbrace{\big(G(\alpha(t),\beta(t),\mathcal{X}(t))-\nabla_{\beta} g(\alpha(t),\beta(t))\big)\mathrm{d}t + \mathrm{d}\zeta_2(t)}_{=\hspace{.5mm}\mathrm{d}\xi_2(t)}\bigg]. \nonumber
\end{align}
It is thus sufficient to prove that the alternative conditions in Theorem \ref{theorem1a} (Assumptions \ref{assumption1a} - \ref{assumption5c}) imply the original conditions of Theorem \ref{theorem1} (Assumptions \ref{assumption1} - \ref{assumption3}). Indeed, if this is the case, then Theorem \ref{theorem1a} follows directly from Theorem \ref{theorem1}. This statement holds trivially for all conditions except those relating to the noise processes. It thus remains to establish that, under the noise conditions in Theorem \ref{theorem1a} (Assumptions \ref{assumption2a} - \ref{assumption4aii}, \ref{assumption5a} - \ref{assumption5c}), the noise condition in Theorem \ref{theorem1} (Assumption \ref{assumption3}) holds for the noise processes $\{\xi_i(t)\}_{t\geq 0}$, $i=1,2$, as defined above. That is, for all $T>0$, and $i=1,2$, 
\begin{equation}
\lim_{s\rightarrow\infty} \sup_{t\in[s,s+T]} \left|\left| \int_{s}^{t} \gamma_i(v)\mathrm{d}\xi_i(v)\right|\right|=0
\end{equation}
But this is an immediate consequence of Assumption \ref{assumption5a} and Lemma \ref{lemmaB1}. The result follows immediately. 
\end{proof}

\section{Sufficient Conditions for Theorem \ref{theorem1a}} \label{appendix_sufficient_conditions_elliptic}
In this Appendix, we provide sufficient conditions for Theorem \ref{theorem1a} in the case that $\smash{\{\mathcal{X}(\alpha,\beta,t)\}_{t\geq 0}}$  in \eqref{diffusion}  is a non-degenerate elliptic diffusion process for all $\alpha\in\mathbb{R}^{d_1}$, $\beta\in\mathbb{R}^{d_2}$. 

\subsection{Preliminaries}
We first recall that $\smash{\{\mathcal{X}(\alpha,\beta,t)\}_{t\geq 0}}$ is an elliptic diffusion process if $\smash{\mathcal{A}_{\mathcal{X}}(\alpha,\beta)}$ is an elliptic operator on $\smash{\mathbb{R}^{d_3}}$. That is, if the matrix $a(\alpha,\beta,x)$, defined according to $\smash{\left(a(\alpha,\beta,x)\right)_{ij} \hspace{-.5mm} = \hspace{-.5mm} \tfrac{1}{2}(\Psi(\alpha,\beta,x)\Psi^T(\alpha,\beta,x))_{ij}}$
 is symmetric and positive semi-definite for all $x\in\mathbb{R}^{d_3}$, i.e., $\forall x,y\in\mathbb{R}^{d_3}$,  
\begin{equation}
0\leq \sum_{i,j=1}^d a_{ij}(\alpha,\beta,x)y_iy_j.
\end{equation}
Meanwhile, we say that $\smash{\{\mathcal{X}(\alpha,\beta,t)\}_{t\geq 0}}$ is a non-degnerate elliptic diffusion process if $\mathcal{A}_{\mathcal{X}}(\alpha,\beta)$ is a \emph{uniformly elliptic operator} on $\mathbb{R}^d$. That is, if there exists $0<\lambda<\infty$ such that, for all $\smash{x,y\in\mathbb{R}^{d_3}}$, the matrix $a(\alpha,\beta,x)$ satisfies
\begin{equation}
\lambda |y|^2\leq \sum_{i,j=1}^d a_{ij}(\alpha,\beta,x)y_iy_j.
\end{equation}




\subsection{Sufficient Conditions}

\subsubsection{Assumption \ref{assumption2a}}
Assumption \ref{assumption2a} relates to the ergodicity of the diffusion process $\{\mathcal{X}(\alpha,\beta,t)\}_{t\geq 0}$. It is implied by the following two sufficient conditions, which correspond to Assumptions (H$_a$), (H$_b$) in \cite{Pardoux2003}. They also appear, in a somewhat less general form, as Assumption 2.2 in the analysis of the continuous-time stochastic gradient algorithm with Markovian dynamics in \cite{Sirignano2017a}.

We remark that, in the case that $\{\mathcal{X}(\alpha,\beta,t)\}_{t\geq 0}$ is a degenerate elliptic diffusion process, one can also obtain sufficient conditions which guarantee ergodicity (e.g., \cite{Hormander1967,Ichihara1974}). 
However, these conditions are significantly more difficult to verify, and we do not state them here.

\begin{manualassumption}{2.2.2a.i'} \label{elliptic_assumption1}
The diffusion coefficient $\Psi(\alpha,\beta,x)$ is uniformly non degenerate and bounded in $\mathbb{R}^{d_3}$ for all $\alpha\in\mathbb{R}^{d_1}$, $\beta\in\mathbb{R}^{d_2}$. That is, there exist constants $0<\kappa_1<\kappa_2<\infty$ such that, for all $\alpha\in\mathbb{R}^{d_1},\beta\in\mathbb{R}^{d_2},x\in\mathbb{R}^{d_3}$, 
\begin{equation}
\kappa_1 I \leq a(\alpha,\beta,x)\leq\kappa_2I.
\end{equation}
\end{manualassumption}

\begin{manualassumption}{2.2.2a.ii'} \label{elliptic_assumption2}
The following recurrence condition holds for all $\alpha\in\mathbb{R}^{d_1}$, $\beta\in\mathbb{R}^{d_2}$,
\begin{equation}
\lim_{|x|\rightarrow\infty} \Phi(\alpha,\beta,x) \cdot x = - \infty. \\[2mm]
\end{equation}
\end{manualassumption}

\begin{proposition}
Assumptions \ref{elliptic_assumption1} - \ref{elliptic_assumption2} imply Assumption \ref{assumption2a}.
\end{proposition}
\begin{proof}
This is a standard result. See, for example, \cite{Pardoux2003}. 
\end{proof}

\subsubsection{Assumption \ref{assumption2ai}}
Assumptions \ref{assumption2ai} relates to the regularity of the invariant measure and its derivatives. It is guaranteed by the following sufficient condition, in addition to the previous two. This corresponds to Assumption (H$^{2+\alpha,1}$) in \cite{Pardoux2003}. It also appears, in a slightly less general form, as Assumption 2.3.3 in \cite{Sirignano2017a}.

\begin{manualassumption}{2.2.2b'} \label{elliptic_assumption3}
The functions $\Phi(\alpha,\beta,x)$, $\Psi(\alpha,\beta,x)$ are in $\mathbb{H}^{2+\delta,1}(\mathbb{R}^{d_3})$, for all $\alpha\in\mathbb{R}^{d_1},\beta\in\mathbb{R}^{d_2}$. Namely, these functions have two bounded derivatives in $x$, and one bounded derivative in $\alpha,\beta$, with all partial derivatives H\"older continuous with exponent $\delta$ with respect to $x$, uniformly in $\alpha,\beta$. 
\end{manualassumption}

\begin{proposition}
Assumptions \ref{elliptic_assumption1} - \ref{elliptic_assumption3} imply Assumption \ref{assumption2ai}
\end{proposition}

\begin{proof}
The result follows by a simple extension of the result in \cite[Lemma 3.3]{Sirignano2017a}, making use of the bounds established in \cite[Theorem 1]{Pardoux2003}. 
\end{proof}


\subsubsection{Assumption \ref{assumption4a}}

Assumption \ref{assumption4a} relates to the existence, uniq-ueness, and properties of solutions of the Poisson equation associated with the ergodic diffusion process. It can be replaced by the following sufficient condition, in addition to the previous three. This corresponds to the second condition of \cite[Theorem 3]{Pardoux2003}; see also Assumption 2.3.1 - 2.3.2 in \cite{Sirignano2017a}.

\begin{manualassumption}{2.2.2d'} \label{elliptic_assumption4}
The functions $F(\alpha,\beta,x)$, $G(\alpha,\beta,x)$ are in $\bar{\mathbb{H}}^{\delta,2}(\mathbb{R}^d)$, for all $\alpha\in\mathbb{R}^{d_1},\beta\in\mathbb{R}^{d_2}$. Namely, these functions have two bounded derivatives in $\alpha,\beta$, with all partial derivatives Hold\"er continuous with exponent $\delta$ with respect to $x$, uniformly in $\alpha,\beta$. Furthermore, $F(\alpha,\beta,x)$ and $G(\alpha,\beta,x)$, and all of their first and second derivatives with respect to $\alpha,\beta$, have the PGP. 
\end{manualassumption}

\begin{proposition}
Assumptions \ref{elliptic_assumption1} - \ref{elliptic_assumption4} imply Assumption \ref{assumption4a}.
\end{proposition}

\begin{proof}
This result follows immediately form \cite[Theorem 3]{Pardoux2003}.
\end{proof}

\subsubsection{Assumption \ref{assumption4aii}}
Assumption \ref{assumption4aii} relates to bounds on the moments of the ergodic diffusion process. It can be replaced by identical sufficient conditions to those used for the existence and uniqueness of an invariant measure, under the additional requirement that the diffusion process is independent of $\alpha$ and $\beta$. These conditions represent a particular case of Assumptions (A$_{\sigma}$), (A$_b$) in \cite{Pardoux2001}. 
They also appear, in a slightly less general setting, as Assumption 2.2 in \cite{Sirignano2017a}.

\begin{proposition}
Suppose that $\{\mathcal{X}(\alpha,\beta,t)\}_{t\geq 0} = \{\mathcal{X}(t)\}_{t\geq 0}$ is independent of $\alpha$, $\beta$. Then Assumptions \ref{elliptic_assumption1} - \ref{elliptic_assumption2} imply Assumption \ref{assumption4aii}.
\end{proposition}

\begin{proof}
This result follows immediately from \cite[Proposition 1,2]{Pardoux2001}. 
\end{proof}

\section{Proof of Proposition \ref{theorem2}} \label{appendix:proof_theorem2}
In this Appendix, we provide a proof of Proposition \ref{theorem2}. For convenience, let us recall the statement of this result.

\begin{proposition_recall_1}
Assume that Assumptions \ref{assumption1a}, \ref{assumption2a} -  \ref{assumption4aii}, \ref{assumption5b}, \ref{assumption4} - \ref{assumption6}, and \ref{filter_assumption1} - \ref{filter_assumption3} hold. Then, with probability one, 
\begin{align}
\lim_{t\rightarrow\infty}\nabla_{\theta}\tilde{\mathcal{L}}^{ \mathrm{(filter)}}(\theta(t),\boldsymbol{o}(t)) = \lim_{t\rightarrow\infty}\nabla_{\boldsymbol{o}}\tilde{\mathcal{J}}^{ \mathrm{(filter)}}(\theta(t),\boldsymbol{o}(t)) = 0,
\end{align}
or
\begin{align}
\lim_{t\rightarrow\infty}(\theta(t),\boldsymbol{o}(t)) \in \{(\theta,\boldsymbol{o}):\theta\in \partial\Theta\cup \boldsymbol{o}\in \partial\Omega^{n_y}\}.
\end{align}
\end{proposition_recall_1}

\begin{remark__}
We assume here that the stated assumptions hold for the diffusion process $\mathcal{X}$ defined in \eqref{compact_SDE}, the functions $F$ and $G$ defined in \eqref{F_def_} and \eqref{G_def_}, and the semi-martingale $\zeta_1$ defined in \eqref{zeta_1_def}. Moreover, where necessary, we replace the algorithm iterates $(\alpha,\beta)$ by $(\theta,\boldsymbol{o})$, and the functions $f$ and $g$ by $\tilde{\mathcal{L}}^{(\text{filter})}$ and $\tilde{\mathcal{J}}^{(\text{filter})}$, as defined in \eqref{approxL} and \eqref{approxJ}.
\end{remark__}

\begin{proof}
We begin by defining the first exit times from $\Theta$, $\Omega^{n_y}$, respectively, as 
\begin{align}
\tau_{\theta} &= \inf\{t\geq 0:\theta(t)\not\in\Theta\}, \\
\tau_{\boldsymbol{o}} &= \inf\{t\geq 0:\boldsymbol{o}(t)\not\in\Omega^{n_y}\}.
\end{align}
First suppose that $\tau_{\theta}<\infty$. Since the paths of $\{\theta(t)\}_{t\geq 0}$ are continuous, it follows that $\theta(\tau_{\theta})\in\partial\Theta$. Furthermore, since $\mathrm{d}\theta(t)=0$ on $\partial\Theta$, we in fact have $\theta(t)\in\partial\Theta$ for all $t\geq \tau_{\theta}$. In particular, it follows that 
\begin{equation}
\lim_{t\rightarrow\infty}(\theta(t),\boldsymbol{o}(t)) \in \{(\theta,\boldsymbol{o}):\theta\in \partial\Theta\cup \boldsymbol{o}\in \partial\Omega^{n_y}\}.
\end{equation}
Using an identical argument, the same conclusion holds under the assumption that $\tau_{\boldsymbol{o}}<\infty$. 

It remains to consider the case when $\tau_{\theta}=\tau_{\boldsymbol{o}}=\infty$. That is, equivalently, when $\theta(t)\in\Theta$ and $\boldsymbol{o}(t)\in\Omega^m$ for all $t\geq 0$. In this instance, it is straightforward to see that Algorithm \eqref{RML_ROSP_eq1} - \eqref{RML_ROSP_eq2} is a special case of Algorithm \eqref{alg2_1} - \eqref{alg2_2}, in which the noise sequences  $\{\zeta_i(t)\}_{t\geq 0}$, $i=1,2$, are defined according to
\begin{align}
\mathrm{d}\zeta_1(t) &= -\zeta_1^{(2)}(\theta(t),\boldsymbol{o}(t),\mathcal{X}(t))\mathrm{d}w(t), \label{noise1} \\
\mathrm{d}\zeta_2(t) &= 0,\label{noise2}
\end{align}
and where the function $\zeta_1^{(2)}:\mathbb{R}^{n_{\theta}}\times\mathbb{R}^{n_yn_{\boldsymbol{o}}}\times\mathbb{R}^{N}\rightarrow\mathbb{R}^{n_{\theta}\times n_y}$ is defined in equation \eqref{zeta_1_def}. It is thus sufficient to prove that the single condition relating to these noise sequences in Proposition \ref{theorem2} (Assumption \ref{assumption5b}) is sufficient for the additional conditions in Theorem \ref{theorem1a} (Assumptions \ref{assumption5a}, \ref{assumption5c}). Indeed, in this case, it follows immediately from Theorem \ref{theorem1a} that 
\begin{align}
\lim_{t\rightarrow\infty}\nabla_{\theta}\tilde{\mathcal{L}}(\theta(t),\boldsymbol{o}(t)) = \lim_{t\rightarrow\infty}\nabla_{\boldsymbol{o}}\tilde{\mathcal{J}}(\theta(t),\boldsymbol{o}(t)) = 0.
\end{align}
We begin by considering Assumption \ref{assumption5a}. We wish to prove that for all $T>0$, the noise sequences $\{\zeta_i(t)\}_{t\geq 0}$, $i=1,2$, almost surely satisfy
\begin{equation}
\lim_{s\rightarrow\infty}\sup_{t\in[s,s+T]}\left|\left|\int_{s}^{t}\gamma_i(v)\mathrm{d}\zeta_i(v)\right|\right|=0
\end{equation}
This condition holds trivially for $\{\zeta_2(t)\}_{t\geq 0}$. We thus turn our attention to $\{\zeta_1(t)\}_{t\geq 0}$. Using the It\^{o} Isometry, and Assumptions \ref{assumption1a}, \ref{assumption4aii}, and \ref{assumption5b}, there exist constants $q>0$ and constants $K,K'>0$ such that
\begin{align}
\sup_{t\geq 0}\mathbb{E}\left[\left(\int_0^{\infty}\gamma_1(t)\mathrm{d}\zeta_1(t)\right)^2\right] &= \mathbb{E}\left[\left(\int_0^t \gamma_{1}(t)\zeta_1^{(2)}(\theta(t),\boldsymbol{o}(t),\mathcal{X}(t))\mathrm{d}w(t)\right)^2\right] \\
&\leq K \mathbb{E}\left[\int_0^t \gamma_1^2(t) (1+\mathbb{E}||\mathcal{X}(t)||^q)\mathrm{d}t\right] \\
&\leq KK'\int_0^t\gamma_1^2(t)\mathrm{d}t< \infty.
\end{align}
Thus, by Doob's martingale convergence theorem, there exists a square integrable random variable, say $M_{\infty}$, such that, both a.s. and in $L^2$,  $\lim_{t\rightarrow\infty} \int_0^{t}\gamma_1(t)\mathrm{d}\zeta_1(t) = M_{\infty}$. The required result follows. 

It remains to consider Assumption \ref{assumption5c}. We wish to prove that there exist constants $A_{z_1,z_2},A_{z_i,b}>0$, $i=1,2$,  such that, componentwise, 
\begin{equation}
c_{z_1,z_2}(t) = \frac{\mathrm{d}[z_1,z_2](t)}{\mathrm{d}t}\leq A_{z_1,z_2}~,~c_{z_i,b}(t) = \frac{\mathrm{d}[z_i,b](t)}{\mathrm{d}t}\leq A_{z_i,b}.
\end{equation}
where, in the general case, $\smash{\{z_i(t)\}_{t\geq 0}}$ are the $\smash{\mathbb{R}^{d_5^i}}$-valued Wiener processes appearing in the definition of the noise processes $\smash{\{\zeta_i(t)\}_{t\geq 0}}$, c.f. \eqref{add_noise}, and $\smash{\{b(t)\}_{t\geq 0}}$ is the $\mathbb{R}^{d_4}$-valued Wiener process appearing in the definition of the ergodic diffusion process $\smash{\{\mathcal{X}(t)\}_{t\geq 0}}$, c.f. \eqref{diffusion}. 

In the case of Algorithm \eqref{RML_ROSP_eq1} - \eqref{RML_ROSP_eq2}, we identify $z_1(t) = w(t)$, $z_2(t) = 0$ from \eqref{noise1} - \eqref{noise2}, and $\smash{b(t) = (v(t),w(t),a(t))^T}$ from \eqref{compact_SDE}. Thus, using elementary properties of the quadratic variation, we have, componentwise,
\begin{equation}
\frac{\mathrm{d}[z_1,z_2](t)}{\mathrm{d}t} = 0~,~\frac{\mathrm{d}[z_1,b](t)}{\mathrm{d}t} = 1 \text{ or } 0~,~\frac{\mathrm{d}[z_2,b](t)}{\mathrm{d}t} = 0.
\end{equation}

In particular, Assumption \ref{assumption5c} is satisfied. The result now follows immediately from our previous remarks.
\end{proof}

\section{Sufficient Conditions for Proposition \ref{theorem2}} \label{appendix:loglik_sufficient}
In this Appendix, we provide sufficient conditions for some of the assumptions in Proposition \ref{theorem2}. 


\subsection{Assumption \ref{assumption4a}}
Assumption \ref{assumption4a} relates to the existence and properties of the filter representations of the asymptotic log-likelihood and the asymptotic sensor placement objective function, namely

\begin{equation}
\tilde{\mathcal{L}}^{\text{(filter)}}(\theta,\boldsymbol{o}) = \lim_{t\rightarrow\infty}\frac{1}{t}\mathcal{L}_t^{\text{(filter)}}(\theta,\boldsymbol{o})~~,~~\tilde{\mathcal{J}}^{\text{(filter)}}(\theta,\boldsymbol{o}) = \lim_{t\rightarrow\infty}\frac{1}{t}\mathcal{J}_t^{\text{(filter)}}(\theta,\boldsymbol{o})
\end{equation}

It also relates to the existence, uniqueness, and properties of solutions of the Poisson equations associated with these functions, and the $\mathbb{R}^{N}$-valued ergodic diffusion process $\{\mathcal{X}(t)\}_{t\geq 0}$, consisting of the latent state, the filter, and the tangent filter. It can be replaced by the following two sufficient conditions, which represent a two-timescale extension of two of the conditions appearing in the analysis of the continuous time, online parameter estimation algorithm in \cite{Surace2019}.

\begin{manualassumption}{2.2.2c.i''}  \label{loglik_assumption2di}
For all $H\in\mathbb{H}_c^{1+\delta,2}(\mathbb{R}^{N})$, the Poisson equation \linebreak $\smash{\mathcal{A}_{\mathcal{X}}v(\theta,\boldsymbol{o},x) = H(\theta,\boldsymbol{o},x)}$ 
has a unique solution $v(\theta,\boldsymbol{o},x)$ that lies in $\mathbb{H}^{1+\delta,2}(\mathbb{R}^{N})$, with $v(\theta,\boldsymbol{o},\cdot)\in C^2(\mathbb{R}^{N})$. Moreover, if $H\in\bar{\mathbb{H}}^{1+\delta,2}(\mathbb{R}^{N})$, then $v\in\bar{\mathbb{H}}^{1+\delta,2}(\mathbb{R}^{N})$, and its mixed first partial derivatives with respect to $(\theta,x)$ and $(\boldsymbol{o},x)$ have the PGP.
\end{manualassumption}

\begin{manualassumption}{2.2.2c.ii''}  \label{loglik_assumption2dii}
The functions $F$, $G$ as defined in \eqref{F_def_} - \eqref{G_def_}, are in $\bar{\mathbb{H}}^{1+\delta,2}(\mathbb{R}^{N})$.
The functions $\psi_{C}$, $\psi_{j}$, as defined in \eqref{phi_C_def} - \eqref{phi_j_def}, are in $\mathbb{H}^{1+\delta,2}(\mathbb{R}^{p})$, and have the PGP componentwise.
The functions $C$, $\psi_{C}^{\theta}$, and $\psi_{j}^{\boldsymbol{o}}$, as defined in (\ref{obs_finite_dim_inf}), (\ref{hatC_theta_0}), and (\ref{hatj_o_0}), have the PGP componentwise.
\end{manualassumption}

\begin{proposition}
Assumptions \ref{assumption2a}, \ref{assumption2ai}, \ref{assumption4aiii}, \ref{loglik_assumption2di} - \ref{loglik_assumption2dii} and \ref{assumption4aii} imply Assumption \ref{assumption4a}. 
\end{proposition}

\begin{proof}
The result follows as an extension of Proposition 1 in \cite{Surace2019}.  In particular, under Assumptions \ref{assumption2a}, \ref{assumption2ai}, \ref{loglik_assumption2di} and \ref{assumption4aii}, Proposition 1(i)-(ii) in \cite{Surace2019} guarantee the existence of $\tilde{\mathcal{L}}^{\text{(filter)}}(\theta,\boldsymbol{o})$ and $\nabla_{\theta}\tilde{\mathcal{L}}^{\text{(filter)}}(\theta,\boldsymbol{o})$. Moreover, under the same assumptions, Proposition 1(iii) in \cite{Surace2019} guarantees that the function $\nabla_{\theta}\tilde{\mathcal{L}}^{\text{(filter)}}(\theta,\boldsymbol{o}) - F(\theta,\boldsymbol{o},x)$ is in $\mathbb{H}_c^{1+\delta,2}(\mathbb{R}^{N})\cap\bar{\mathbb{H}}^{1+\delta,2}(\mathbb{R}^N)$. 

Using the same assumptions, and very similar arguments, one can similarly establish the existence of $\tilde{\mathcal{J}}^{\text{(filter)}}(\theta,\boldsymbol{o})$ and $\smash{\nabla_{\boldsymbol{o}}\tilde{\mathcal{J}}^{\text{(filter)}}(\theta,\boldsymbol{o})}$, and that the function $\smash{\nabla_{\boldsymbol{o}}\tilde{\mathcal{J}}^{\text{(filter)}}(\theta,\boldsymbol{o}) - G(\theta,\boldsymbol{o},x)}$ is in $\smash{\mathbb{H}_c^{1+\delta,2}(\mathbb{R}^{N})\cap\bar{\mathbb{H}}^{1+\delta,2}(\mathbb{R}^N)}$. Given the similarity with the proof of Proposition 1 in \cite{Surace2019}, the details are omitted. 

Having established these properties, one can now use Assumption \ref{loglik_assumption2di} to conclude that the Poisson equations 
\begin{align}
\mathcal{A}_{X}\tilde{F}(\theta,\boldsymbol{o},x) = \nabla_{\theta}\tilde{\mathcal{L}}^{\text{(filter)}}(\theta,\boldsymbol{o}) - F(\theta,\boldsymbol{o},x) \\
\mathcal{A}_{X}\tilde{G}(\theta,\boldsymbol{o},x) = \nabla_{\boldsymbol{o}}\tilde{\mathcal{J}}^{\text{(filter)}}(\theta,\boldsymbol{o}) - G(\theta,\boldsymbol{o},x) 
\end{align}
have unique solutions in $\smash{\bar{\mathbb{H}}^{1+\delta,2}(\mathbb{R}^N)}$, whose mixed first partial derivatives with respect to $(\theta,x)$ and $(\boldsymbol{o},x)$ have the PGP. Thus, in particular, all of the conditions in Assumption \ref{assumption4a} are satisfied.  
\end{proof}

\subsection{Assumption \ref{assumption5b}} Assumption \ref{assumption5b} relates to the properties of the additive, state-dependent noise sequences $\{\zeta_i(t)\}_{t\geq 0}$, $i=1,2$, appearing in the update equations for $\smash{\{\theta(t)\}_{t\geq 0}}$ and $\smash{\{\boldsymbol{o}(t)\}_{t\geq 0}}$. It is guaranteed by the following sufficient condition. 

\begin{manualassumption}{2.2.3b'}  \label{loglik_assumption23b}
The function $\psi_{C}^{\theta}$, as defined in equation \eqref{hatC_theta_0}, has the polynomial growth property componentwise. 
\end{manualassumption}

\begin{proposition}
Assumption \ref{loglik_assumption23b} implies Assumption \ref{assumption5b}
\end{proposition}

\begin{proof}
We are required to prove that the functions $\smash{\zeta_i^{(2)}}$, $i=1,2$, have the PGP componentwise. By Assumption \ref{loglik_assumption23b}, $\smash{\zeta_1^{(2)}}$ has the PGP componentwise. Meanwhile, $\smash{\zeta_2^{(2)}}$ is identically zero, c.f. \eqref{noise2}, and thus trivially has the PGP.
\end{proof}

\section{Extensions to Proposition \ref{theorem2}} \label{app:theorem2_ext}
In this Appendix, we discuss extensions to Proposition \ref{theorem2} in the case the Kushner-Stratonovich equation only admits a finite-dimensional recursive approximation. Under the stated assumption, Proposition \ref{theorem2} guarantees that,\footnote{For the purpose of this discussion, we will ignore the projection device which ensures that the algorithm iterates remain within $\Theta$ and $\Omega$, respectively.}
\begin{equation}
\lim_{t\rightarrow\infty}\nabla_{\theta}\tilde{\mathcal{L}}^{\text{(filter)}}(\theta(t),\boldsymbol{o}(t))=
\lim_{t\rightarrow\infty}\nabla_{\theta}\tilde{\mathcal{J}}^{\text{(filter)}}(\theta(t),\boldsymbol{o}(t))=0~,~~~\text{a.s.}
\end{equation}
where $\smash{\tilde{\mathcal{L}}^{\text{(filter)}}}(\theta,\boldsymbol{o})$ and $\smash{\tilde{\mathcal{J}}^{\text{(filter)}}}(\theta,\boldsymbol{o})$ are the representations of the asymptotic log-likelihood and the asymptotic sensor placement objective in terms of the approximate finite dimensional filter, and $\theta(t)$ and $\boldsymbol{o}(t)$ are the parameter estimates and optimal sensor placements generated by Algorithm \eqref{RML_ROSP_eq1} - \eqref{RML_ROSP_eq2}. In the case that one only has access to an approximate filter, the functions $\smash{\tilde{\mathcal{L}}^{\text{(filter)}}}(\theta,\boldsymbol{o})$ and $\smash{\tilde{\mathcal{J}}^{\text{(filter)}}}(\theta,\boldsymbol{o})$ are only approximations of the true objective functions $\tilde{\mathcal{L}}(\theta,\boldsymbol{o})$ and $\tilde{\mathcal{J}}(\theta,\boldsymbol{o})$. As such, it would clearly be preferable to obtain a result of the form
\begin{equation}
\lim_{t\rightarrow\infty}\nabla_{\theta}\tilde{\mathcal{L}}(\theta(t),\boldsymbol{o}(t))=
\lim_{t\rightarrow\infty}\nabla_{\theta}\tilde{\mathcal{J}}(\theta(t),\boldsymbol{o}(t))=0~,~~~\text{a.s.}
\end{equation}
For now, we will consider the slightly easier task of trying to obtain a result of the form
\begin{equation}
\lim_{t\rightarrow\infty} ||\nabla_{\theta}\tilde{\mathcal{L}}(\theta(t),\boldsymbol{o}(t))|| =
\lim_{t\rightarrow\infty}||\nabla_{\theta}\tilde{\mathcal{J}}(\theta(t),\boldsymbol{o}(t))||=0 \label{easy_converge}
\end{equation}
where the mode of convergence is to be specified. To make progress towards this goal, let us consider the simple decomposition
\begin{align}
||\nabla_{\theta}\tilde{\mathcal{L}}(\theta(t),\boldsymbol{o}(t))|| &\leq  ||\nabla_{\theta} \tilde{\mathcal{L}}(\theta(t),\boldsymbol{o}(t)) - \nabla_{\theta}\tilde{\mathcal{L}}^{\text{(filter)}}(\theta(t),\boldsymbol{o}(t))|| \label{decomp1} \\
&+ ||\nabla_{\theta} \tilde{\mathcal{L}}^{\text{(filter)}}(\theta(t),\boldsymbol{o}(t))|| \nonumber
\end{align} 
where, for the sake of brevity, we have now restricted our attention to the log-likelihood. Proposition \ref{theorem2} guarantees that the second term in this decomposition converges to zero a.s. as $t\rightarrow\infty$. It thus remains to bound the first term. Evidently this bound will depend on the properties of the filter, and vanishes if the filter is exact. To obtain such a bound, we will write
\begin{align}
&||\nabla_{\theta} \tilde{\mathcal{L}}(\theta(t),\boldsymbol{o}(t)) - \nabla_{\theta}\tilde{\mathcal{L}}^{\text{(filter)}}(\theta(t),\boldsymbol{o}(t))|| \\
&\hspace{20mm}\leq  ||\nabla_{\theta} \tilde{\mathcal{L}}(\theta(t),\boldsymbol{o}(t)) - \frac{1}{t}\nabla_{\theta}\mathcal{L}_t (\theta(t),\boldsymbol{o}(t))|| \label{decomp2a} \\
&\hspace{20mm}+||\frac{1}{t}\nabla_{\theta}\mathcal{L}_t (\theta(t),\boldsymbol{o}(t)) - \frac{1}{t}\nabla_{\theta}\mathcal{L}_t^{\text{(filter)}} (\theta(t),\boldsymbol{o}(t))||  \nonumber\\ 
&\hspace{20mm}+||\frac{1}{t}\nabla_{\theta}\mathcal{L}_t^{\text{(filter)}} (\theta(t),\boldsymbol{o}(t)) - \nabla_{\theta}\tilde{\mathcal{L}}^{\text{(filter)}}(\theta(t),\boldsymbol{o}(t))|| \nonumber 
\end{align}
where $\mathcal{L}_t(\theta,\boldsymbol{o})$ denotes the true log-likelihood at time $t$, and $\mathcal{L}_t^{\text{(filter)}}(\theta,\boldsymbol{o})$ denotes the filter representation of the log-likelihood at time $t$. Under our assumptions, it is straightforward to show that the first term and the third term in this decomposition converge to zero a.s. as $t\rightarrow\infty$ (see also \cite[Proposition 1]{Surace2019}). 
We thus turn our attention to the central term. Using our previous expression for the log-likelihood function, c.f. \eqref{ll_func}, and its representation in terms of the approximate filter, we have that 
\begin{align}
& \frac{1}{t} \nabla_{\theta} \mathcal{L}_t(\theta,\boldsymbol{o}) - \frac{1}{t} \nabla_{\theta} \mathcal{L}_t^{\text{(filter)}}(\theta,\boldsymbol{o}) \\
&=\frac{1}{t}\nabla_{\theta} \left[ \int_0^t R^{-1}(\boldsymbol{o}) \hat{C}(\theta,\boldsymbol{o},s)\cdot\mathrm{d}y(s)-\frac{1}{2}\int_0^t ||R^{-\frac{1}{2}}(\boldsymbol{o})\hat{C}(\theta,\boldsymbol{o},s)||^2\mathrm{d}s\right] \\
&-\frac{1}{t}\nabla_{\theta} \left[ \int_0^t R^{-1}(\boldsymbol{o}) \psi_{C}(\theta,\boldsymbol{o},M(s))\cdot\mathrm{d}y(s)-\frac{1}{2}\int_0^t ||R^{-\frac{1}{2}}(\boldsymbol{o})\psi_{C}(\theta,\boldsymbol{o},M(s))||^2\mathrm{d}s\right]. \hspace{-10mm}
\end{align}
It follows that, after some rearrangement of this expression, that given suitable bounds on the quantities $||\psi_{C}(\theta,\boldsymbol{o},M(s)) - \hat{C}(\theta,\boldsymbol{o},s)||$ and $||\psi_C^{\theta}(\theta,\boldsymbol{o},s) - \hat{C}^{\theta}(\theta,\boldsymbol{o},s)||$, it will be possible to bound this term. In many cases (e.g., linear observations), this corresponds to bounds on $||\psi_{x}(\theta,\boldsymbol{o},M(s)) - \hat{x}(\theta,\boldsymbol{o},s)||$ and $||\psi_{x}^{\theta}(\theta,\boldsymbol{o},M(s)) - \hat{x}^{\theta}(\theta,\boldsymbol{o},s)||$, where, for example, $\psi_{x}$ denotes the estimate of the conditional mean $\hat{x}$ in terms of the approximate filter.
In general, it will be necessary to verify these bounds on a case by case basis. There are, however, some notable exceptions, including the Ensemble Kalman-Bucy Filter (EnKBF) (e.g., \cite{DeWiljes2018,Moral2018}). Let us briefly demonstrate how existing results on this filter can be applied in our context, in the simplified setting where the observations are linear: that is, $C(\theta,\boldsymbol{o},x) = C(\theta,\boldsymbol{o})x$.

Suppose that $\smash{(v^{i,N}(t),w^{i,N}(t),x^{i,N}_0)_{i=1}^N}$ are independent copies of $\smash{(v(t),w(t),x_0)}$. The EnKBF consists of $N$ interacting particles $\smash{(x^{i,N}(\theta,\boldsymbol{o},t)_{i=1}^N}$ which evolve according to following system of interacting stochastic differential equations
\begin{align}
\mathrm{d}x^{i,N}(\theta,\boldsymbol{o},t) &= A(\theta,x^{i,N}(\theta,\boldsymbol{o},t))\mathrm{d}t + B(\theta,x^{i,N}(\theta,\boldsymbol{o},t))\mathrm{d}v^{i,N}(t)  \\
&+P^N(\theta,\boldsymbol{o},t){C}^T(\theta,\boldsymbol{o})R^{-1}(\boldsymbol{o})(\mathrm{d}y(t) - C(\theta,\boldsymbol{o})x^{i,N}(\theta,\boldsymbol{o},t)\mathrm{d}t - \mathrm{d}w^{i,N}(t))  \nonumber
\intertext{where}
P_N(\theta,\boldsymbol{o},t) &= \frac{1}{N-1}\sum_{i=1}^N(x^{i,N}(\theta,\boldsymbol{o},t) - {m}_N(\theta,\boldsymbol{o},t))(x^{i,N}(\theta,\boldsymbol{o},t) - {m}_N(\theta,\boldsymbol{o},t))^T  \\
{m}_N(\theta,\boldsymbol{o},t) &= \frac{1}{N}\sum_{i=1}^N x^{i,N}(\theta,\boldsymbol{o},t).  \nonumber
\end{align}
represent the (empirical) filter estimates of the conditional covariance $\hat{\Sigma}(\theta,\boldsymbol{o},t)$ and the conditional mean $\hat{x}(\theta,\boldsymbol{o},t)$. Under additional assumptions,\footnote{In particular, one requires that $A(\theta,\boldsymbol{o},x)$ is linear: $A(\theta,\boldsymbol{o},x) = A(\theta,\boldsymbol{o})x$.} it is possible to show that, for all $p\geq 1$, and for sufficiently large $N$, (e.g., \cite[Theorem 3.6]{Moral2018})
\begin{align}
\mathbb{E}\left[||m_N(\theta,\boldsymbol{o},t) - \hat{x}(\theta,\boldsymbol{o},t)||^p\right]^\frac{1}{p}\leq \frac{K(p)}{N^{\frac{1}{2}}} 
\end{align}
where $K(p)<\infty$ is a constant independent of $N$.  
Suppose, in addition, that one could establish a similar bound for the tangent EnKBF (this remains an open problem in the general case). Then, using these results, our existing assumptions (e.g., polynomial growth, uniformly bounded moments for the filter and the tangent filter), and the H\"older inequality, after some algebra one arrives at
\begin{align}
\mathbb{E}\left[||\frac{1}{t} \nabla_{\theta} \mathcal{L}_t(\theta,\boldsymbol{o})  - \frac{1}{t}  \nabla_{\theta} \mathcal{L}^{\text{(filter)}}_t(\theta,\boldsymbol{o}) ||^p \right]^{\frac{1}{2p}} \leq \frac{K(p)}{N^{\frac{1}{2}}} \label{eq_bound1} 
\end{align}
It follows, substituting this bound into \eqref{decomp2a}, substituting \eqref{decomp2a} into \eqref{decomp1}, that 
\begin{align}
\lim_{t\rightarrow\infty}\mathbb{E}\left[||\nabla_{\theta}\tilde{\mathcal{L}}(\theta(t),\boldsymbol{o}(t))||^{p}\right]^{\frac{1}{2p}} & \leq \frac{K(p)}{N^{\frac{1}{2}}}  
\end{align}
One can follow the same argument to obtain an identical bound for $\nabla_{\theta}\tilde{\mathcal{J}}(\theta,\boldsymbol{o})$. It follows immediately from these bounds that the limit \eqref{easy_converge} holds in $\mathbb{L}^p$, for all $p\geq 1$, under the additional limit that $N\rightarrow\infty$ (i.e., as the number of particles goes to infinity). We remark that a rigorous result of this type has recently been established for (discrete-time) recursive maximum likelihood estimation in non-linear state-space models \cite{Tadic2021}.

\bibliographystyle{siamplain}
\bibliography{references}

\end{document}